\documentclass[11pt,reqno]{amsart}

\usepackage[utf8]{inputenc}
\usepackage[margin=1.25in]{geometry}
\usepackage{hyperref}
\usepackage{appendix}
\usepackage{amsfonts}
\usepackage{amssymb}
\usepackage{stmaryrd} 
\usepackage{amsmath}
\usepackage{amsthm}
\usepackage[dvipsnames]{xcolor}
\usepackage{mathrsfs}

\theoremstyle{plain}
\newtheorem{theorem}{Theorem}[section]
\newtheorem{corollary}[theorem]{Corollary}
\newtheorem{lemma}[theorem]{Lemma}
\newtheorem{Proposition}[theorem]{Proposition}

\newtheorem{fact}[theorem]{Fact}

\newtheorem{notation}[theorem]{Notation}

\theoremstyle{remark}

\newtheorem{remark}[theorem]{Remark}

\usepackage{biblatex} %Imports biblatex package
\numberwithin{equation}{section}
\title[Circular law for block band matrices]{Circular law for non-Hermitian block band matrices with slowly growing bandwidth}

\author{Yi HAN}
\address{Institute for Advanced Study, 1 Einstein Drive, Princeton, NJ
}
\email{hanyi@ias.edu}

\begin{document}

\begin{abstract}
  We consider the empirical eigenvalue distribution for a class of non-Hermitian random block tridiagonal matrices $T$ with independent entries. The matrix has $n$ blocks on the diagonal and each block has size $\ell_n$, so the whole matrix has size $n\ell_n$. We assume that the nonzero entries are i.i.d. with mean 0, variance $(3\ell_n)^{-1}$ and having finite moments of all orders. We prove that when the entries have a bounded density, then whenever $\lim_{n\to\infty}\ell_n=\infty$ and $\ell_n=O(\operatorname{Poly}(n))$, the normalized empirical spectral distribution of $T$ converges almost surely to the circular law. The growing bandwidth condition $\lim _{n\to\infty}\ell_n=\infty$ is the optimal condition for the circular law with small bandwidth.  This confirms the folklore conjecture that the circular law holds whenever the bandwidth increases with the dimension; earlier circular-law proofs required the delocalized regime $\ell_n\gg n$. 
\end{abstract}

\maketitle

\section{Introduction}

Let $A$ be an $N\times N$ matrix with eigenvalues $\lambda_1,\cdots,\lambda_N$. We let $\mu_A:=\frac{1}{N}\sum_{i=1}^N
\delta_{\lambda_i}$ denote the empirical measure of its eigenvalues, where $\delta_\cdot$ is the delta measure. For a symmetric random matrix $A$ with independent entries of mean $0$ and variance $\frac{1}{N}$, then $\mu_A$ converges to the celebrated Wigner semicircle law \cite{wigner1958distribution}. For non-Hermitian random matrices, proving the convergence of $\mu_A$ is a much more challenging task because the method of moments no longer applies.
For a random matrix $A$ with i.i.d. entries, circular law was proven through a long list of partial results \cite{MR1428519}, \cite{tao2008random}, \cite{MR2663633} until Tao and Vu \cite{WOS:000281425000010} obtained the optimal condition.
Later, the convergence of limiting spectral distribution is proven for many other random matrix ensembles, most notably for heavy-tailed random matrices \cite{bordenave2011spectrum}, sparse directed graphs \cite{MR4195739}, and sparse i.i.d. matrices \cite{MR3945840}, \cite{sah2023limiting}. In a related direction, the convergence of ESDs for polynomials in Ginibre matrices has recently been proven in \cite{han2026brown}.

Much less is known about the convergence of $\mu_A$ for a non-Hermitian random matrix with a highly structured variance profile, and the technical reason behind it is that least singular value estimates are notoriously hard to obtain when the matrix does not have a flat variance profile. The most important examples in this category are non-Hermitian random band matrices, which can simply be obtained from a Hermitian band matrix by making every entry independently distributed. However, there is a crucial difference between a Hermitian and a non-Hermitian band matrix model: assuming that the entries have a density, then the random potentials on the diagonal  immediately provide a least singular value estimate for the Hermitian model (see for instance \cite{schenker2009eigenvector}, \cite{peled2019wegner}) which is polynomially small in the dimension. For non-Hermitian band matrices however, if we take the Hermitization, we cannot apply the same proof because the diagonal entries of the Hermitization are all zero! Thus very little is known in the latter setting, even for existence of the limiting density.

The most crucial parameter that governs a band matrix is its bandwidth versus size.
For a band matrix $A$ of size $N$ with bandwidth $W$, the spectral properties of $A$, including the convergence of $\mu_A$, depend in a crucial way on the relative magnitude of $W$ and $N$. The study of circular law on non-Hermitian band matrices was initiated in \cite{cook2018lower} and \cite{jain2021circular}, but these works focus on the range $W\sim N$ or $W\geq N^{1-c}$ for a small $c>0$. This is because the proof techniques use properties that hold in the mean-field (or Ginibre) case that degenerates when $W$ gets smaller. In the ICM survey \cite{MR4680362}, the study of limiting ESDs for inhomogeneous random matrices was also listed as a major open problem, about which we currently have little understanding. The revised version of \cite{han2025circular} proves the circular law for real or circular complex Gaussian entries when $W\geq N^{2/3+c}$, for symmetric or transitively invariant doubly stochastic variance profiles bounded by $C/W$, and gives non-Gaussian extensions at larger bandwidths under additional hypotheses. The scale $W\sim\sqrt{N}$ is the expected transition scale between localized and delocalized regimes of band matrix models. The high-bandwidth methods in these works do not reach deep into the localized regime, for example when $W\sim N^{\alpha}$ with small $\alpha>0$. 

In this work, we push beyond these high-bandwidth regimes for the circular law by showing that for a special class of models, the circular law holds whenever $W$ is slowly growing in $N$. This confirms the widespread belief that the ESD should converge to the circular law whenever $W$ grows, and provides the first instance of circular law proof in any part of the localized regime with growing $W$: $\Omega(1)\leq W\leq\sqrt{N}$.

\subsection{Models and main results}

Let $A_n$ be an ensemble of random matrices with empirical eigenvalue distribution $\mu_{A_n}$, defined on a common probability space. Let $\mu_c$ be the uniform probability measure on the unit disk in the center of the complex plane (i.e., the circular law). Then we say $\mu_{A_n}$ converges almost surely (resp., in probability) to $\mu_c$ if, for any smooth and compactly supported $f:\mathbb{C}\to\mathbb{R}$, the following expression$$
\int_\mathbb{C} f(z)d\mu_{A_n}(z)-\int_\mathbb{C} f(z)d\mu_c(z)
$$
converges to 0 almost surely (resp., in probability).

\begin{theorem}\label{maincircularlaw1.1} Let $\zeta$ be a random variable satisfying that $\mathbb{E}[\zeta]=0,\mathbb{E}[|\zeta|^2]=1,$ and having all moments finite: for all $p\geq 2$, $\mathbb{E}[|\zeta|^{p}]<c_p<\infty$
    for a $c_p>0$. We further assume that
    \begin{itemize}
        \item Either $\zeta$ is a real-valued random variable with distributional density on $\mathbb{R}$ bounded by some $L>0$;
\item or $\zeta$ has independent real and imaginary parts $\Re\zeta,\Im\zeta,$ and at least one of $\Re\zeta,\Im\zeta,$  has a 
distributional density on $\mathbb{R}$ bounded by some $L>0$.
    \end{itemize}
    Take two integers $n,\ell\in\mathbb{N}$ and we denote $\ell=\ell_n$. Consider the following block tridiagonal matrix $T$
\begin{equation}
\label{equationforalargeT}T=\begin{bmatrix}
A_1 &B_1&&&\\C_2&A_2&B_2&&\\&\ddots&\ddots&\ddots&\\&&C_{n-1}&A_{n-1}&B_{n-1}\\&&&C_n&A_n
\end{bmatrix} \end{equation}
where each block $A_i,B_i,C_i$ has size $\ell\times\ell$ and $T$ has size $\ell n$. Assume that each entry of $A_i,B_i,C_i$ is an i.i.d. copy of $\frac{1}{\sqrt{3\ell}}\zeta$ so that the variance is normalized. Then whenever $\lim_{n\to\infty}\ell_n=\infty$ with $n\geq \ell_n^{d}$ for some sufficiently small fixed constant $d>0$, $\mu_T$ converges almost surely to the circular law.
\end{theorem} 

The matrix $T$ has dimension $N:=\ell_n n$ and bandwidth $W:=\ell_n$.
Theorem \ref{maincircularlaw1.1} shows that with a continuous density, the circular law holds for $\mu_T$ whenever the bandwidth $\ell_n$ increases with the matrix dimension and whenever $n$, the number of blocks, is not too small.

The condition on $\ell_n$ relative to $n$ in Theorem \ref{maincircularlaw1.1} is essentially optimal. First, when $\ell_n$ is a fixed constant, then $\mu_T$ does not converge to the circular law in general. Convergence of ESD in the special case $\ell=1$ was studied before in \cite{MR2191234}, and the structure of the limiting density can be remarkably complex \cite{holz2003remarkable} despite simplicity of the model. It is fairly reasonable to expect that for any finite bandwidth $\ell=O(1)$, $\mu_T$ would have a distribution-dependent limiting density which is not the circular law. Therefore, universality kicks in only if $\ell_n$ is growing in $n$, and Theorem \ref{maincircularlaw1.1} justifies that this is exactly the case: the minimal condition $\ell_n\to\infty$ is sufficient to guarantee global universality to the circular law ($\zeta$ is not fixed to be a Gaussian). Our matrix $T$ shares properties both in the ergodic side with length of transfer governed by $n$, and in the mean-field side governed by the square blocks of size $\ell$. Universality in entry distribution arises from the mean-field side, and we have proven that in this non-Hermitian setting, even very mild mean-field mixing with scale $\ell_n\to\infty$ can restore global universality no matter how large the ergodic scale $n$ is. Finally, we remark that the condition $n\geq\ell_n^d$ for a small $d>0$ is also natural and almost necessary because our matrix $T$ is not translationally invariant because it lacks two corner blocks, so we need $n$ to be not too small to make the boundary effect negligible.

\begin{remark}(The density assumption) We use the continuous density of $\zeta$ at two places of the proof: the first is to factor out the determinant as a transfer matrix recursion in Corollary \ref{specialcaselemma1.3}, which uses invertibility of off-diagonal blocks $B_k$. When $\ell\gg\log n$, Bernoulli entries are admissible since almost surely all the matrices are nonsingular. But when $\ell\ll\log n$, almost surely many matrix blocks are singular and the determinant factorization does not hold. The second and more crucial usage is in the proof of the least singular value estimate for the block band matrices in Theorem \ref{theorem1.112} and \ref{updatedtails}. Although part of the estimate may still carry over without a bounded density, the bound in \cite{jain2021circular}, Theorem 2.1, has an $O(\ell^{-1/2})$ failure term and does not control the log determinant on the exceptional event. The transfer matrix approach requires a joint control over all these small blocks, and thus the estimate cannot work through when $n$ is not small relative to $\ell$. This technical difficulty disappears under a bounded density assumption, see Theorem \ref{updatedtails}, where we can control the least singular value with probability 1.

The main focus of this paper is the ergodic regime $n\gg \ell$, and a bounded density is typically assumed to guarantee sufficient regularity in this ergodic, or localized regime. For a Hermitian random band matrix, the Wegner estimate and dynamical localization for $W\ll\sqrt{N}$ are all proven \cite{drogin2025localization} \cite{cipolloni2024dynamical} \cite{schenker2009eigenvector} only when assuming a bounded density, and no Wegner estimate is known without a density assumption when $\Omega(1)\leq W\ll\sqrt{N}$.
 For the non-Hermitian tridiagonal matrix model \cite{MR2191234}, which corresponds to $W=O(1)$ in our paper, their proof of ESD convergence also uses a similar non-singularity assumption on off-diagonal entries.

Finally, we can assume that $\zeta$ has a density with a polynomially (in $\ell$) bounded $\ell^\infty$ norm, say we assume that $\zeta$ has a density on $\mathbb{R}$ bounded by $L\ell^C$ for any $C>0$. For example, this can originate from a vanishing smoothing with $\zeta=\sqrt{1-\ell^{-2C}}\zeta_1+\ell^{-C}g$, where $\zeta_1$ is a centered Rademacher variable and $g$ is an independent Gaussian. The result of Theorem \ref{maincircularlaw1.1} continues to hold under this assumption and the modification is straightforward. 
\end{remark}

\begin{remark}(Model)
The proof of Theorem \ref{maincircularlaw1.1} crucially uses the block tridiagonal structure of $T$, and thus does not immediately generalize to other band matrix models with a more general variance profile, such as
 the periodic case where entry $(i,j)$ has variance $\frac{1}{2W+1}$ if $\min(|i-j|,|N-|i-j||)\leq W$ 
and is identically $0$ otherwise. Interestingly, the periodic block band case (see \eqref{eqperiodic}), where we have an additional matrix $C_1$ on the top right corner and a matrix $B_n$ on the lower left corner of $T$, is not covered by Theorem 
\ref{maincircularlaw1.1}  either. Subsequent work \cite{han2026optimal} proves the circular law in probability for these periodic models whenever $W\to\infty$, under its bounded-density and finite-third-moment assumptions.

From a universality perspective, we expect the circular law to be proven for $\mu_T$ for a much more general class of variance profiles whenever the bandwidth is growing, and Theorem \ref{maincircularlaw1.1} provides the first rigorous justification of this heuristic for at least one type of variance profile.
\end{remark}

\begin{remark}(On the high moment assumption) The high moment assumptions in Theorem \ref{maincircularlaw1.1} are most essentially used in the singular value rigidity result of Proposition \ref{propositionrigidity}. We have not tried to optimize the moment condition here; the earlier rigidity arguments in \cite{cook2018lower,jain2021circular} use a finite $4+\epsilon$ moment, whereas this proof uses moments of all orders.
The moment assumption is also used in bounding the transfer matrix in Section \ref{section2.11110} and estimating the least singular value in Section \ref{leastsingularvalueplugin}, although a weak moment condition might work in these cases.
As rigidity results are indispensable in all these papers, proving a circular law for these inhomogeneous models under only a two moment condition as in \cite{WOS:000281425000010} might require a complete reworking and is beyond the current scope of this paper.
\end{remark}

\begin{remark}(On almost sure circular law) Theorem \ref{maincircularlaw1.1} proves circular law for $\mu_T$ in the almost sure sense. If one adopts the normalization $N=n\ell$ and $W=\ell$ to align with prior works, one can check that the almost sure circular law is proven whenever $W\to\infty$ and $W\leq N^{1-d}$ for any small $d>0$. For the square case $W=N$, the almost sure circular law was proven much earlier, see \cite{WOS:000281425000010}. In contrast, previous circular law papers for band matrices \cite{jain2021circular}, \cite{tikhomirov2023pseudospectrum} only prove the weak (in probability) circular law. The reason why we get strong circular law is twofold: first, a bounded density assumption leads to least singular value estimates with good tails (see Theorem \ref{theorem1.112}), and more importantly, we adopt a transfer matrix approach with $n$ transfer operators and apply martingale concentration inequality, so the larger $n$ is the stronger probabilistic concentration we gain.  
\end{remark}

The last 12 months have seen spectacular progress in the study of random band matrices. With the highest relevance to this work, Shcherbina and Shcherbina \cite{shcherbina2025characteristic} recently studied the second correlation function of the characteristic polynomial of a non-Hermitian band matrix model and showed distinct behaviors for $W\gg\sqrt{N}$ and $W\ll\sqrt{N}$, which serves as a first step towards the proof of Anderson type transition in non-Hermitian band matrices. For Hermitian band matrix models, Yau and Yin \cite{yau2025delocalization} proved for the block band model that whenever $W\geq N^{1/2+c}$ for any $c>0$, in the bulk of the spectrum, the eigenvectors of this band matrix are completely delocalized and quantum unique ergodicity holds. For more general 1-d random band matrices, the same results are proven by Erdős and Riabov in \cite{erdHos2025zigzag}. Higher-dimensional band matrices are also studied in \cite{dubova2025delocalization} and \cite{dubova2025delocalization3d}. In the localized regime $W\leq N^{1/2-c}$, the exponential localization of eigenfunctions was recently proven in Drogin \cite{drogin2025localization}, following a long line of previous papers \cite{cipolloni2024dynamical}, \cite{chen2022random}, \cite{schenker2009eigenvector},\cite{goldstein2022fluctuations}. The invertibility of general non-Hermitian random band matrices with Bernoulli entries was also studied in \cite{han2025invertibility}.
Although this paper does not directly touch upon localization/delocalization of eigenvectors, Theorem \ref{maincircularlaw1.1} studies the most fundamental spectral statistics, the ESD, in particular in the complete localized regime $W\ll \sqrt{N}$. To our best knowledge, almost no spectral results are known in the localized regime for the non-Hermitian band matrix models, and this work may serve as a first step to the investigation in this localized regime.

\subsection{A brief outline of the proof}
As with all modern proofs of circular law, we begin with the following replacement principle by Tao and Vu (\cite{WOS:000281425000010}, Theorem 2.1), which offers a sufficient condition for the circular law via convergence of the log determinant:

\begin{theorem}\label{replacement}(Replacement principle) Suppose that for each $n\in\mathbb{N}$, there are two ensembles of random matrices $A_n,B_n\in M_n(\mathbb{C})$ such that\begin{enumerate}
    \item The expression
    \begin{equation}
        \frac{1}{n^2}\|A_n\|_F^2+\frac{1}{n^2}\|B_n\|_F^2
    \end{equation}
    is bounded in probability (resp. almost surely), where $\|\cdot\|_F$ denotes the matrix Frobenius norm, and
    \item For almost all complex numbers $z\in\mathbb{C}$, the quantity
    $$
\frac{1}{n}\log|\det(\frac{1}{\sqrt{n}}A_n-zI)|-\frac{1}{n}\log|\det(\frac{1}{\sqrt{n}}B_n-zI)|
    $$ converges in probability (resp. almost surely) to zero. Then $\mu_{\frac{1}{\sqrt{n}}A_n}-\mu_{\frac{1}{\sqrt{n}}B_n}$ converges in probability (resp. almost surely) to zero.
\end{enumerate}
    
\end{theorem}
To verify condition (2) of Theorem \ref{replacement}, we need to estimate $\log|\det A_z|$ where $A_z:=A_n-zI$. The standard method for doing this (see the survey \cite{MR2908617}) is as follows: Let $s_1\geq\cdots\geq s_n$ denote the singular values of $A_z$, then $\frac{1}{n}\log|\det A_z|=\frac{1}{n}\sum_{i=1}^n\log s_i$. One can prove that the empirical measure for singular values of $A_z$ converges weakly to a deterministic limit, so the convergence of $\log\det A_z$ follows once we show that the smallest (and a $o(n)$ number of small-ish) singular values among $s_i$ make negligible contribution to the summation. Then we can truncate the summation at a threshold and apply the weak convergence result for a bounded test function. The control of $s_n$ and small-ish $s_i$ is equivalent to obtaining a least (and small-ish) singular value estimate for $A_z$, which typically has size $n^{-C}$ for some $C>0$ with high probability when $A$ has i.i.d. entries \cite{WOS:000281425000010}, \cite{rudelson2008littlewood}.

For a non-Hermitian random band matrix of size $N$ and bandwidth $W$, the currently known lower bounds for the least singular values \cite{jain2021circular}, \cite{han2025circular}, \cite{tikhomirov2023pseudospectrum} are all exponentially small and have the order $s_N\geq\exp(-\frac{N}{W}N^\epsilon)$.  Whether or not these bounds are sharp or can be considerably improved remains an open problem, especially when $W\ll\sqrt{N}$.
Thus to guarantee the convergence of log determinant, a popular method is to get a rigidity estimate for the singular values, namely can we guarantee that with high probability only an $N^c$ number of singular values are smaller than $N^{-1}$, say? Coarse resolvent estimates yield such counting bounds; see \cite{han2025circular}. The regimes treated by these high-bandwidth arguments lie above the scale $W\asymp N^{1/2}$, and these estimates alone do not handle very slowly growing bandwidths, say $W=\log\log N$.

The limitation of this rigidity estimate is that it requires all the singular values be simultaneously controlled, but in reality we may still guarantee convergence of the log determinant without uniform control over all singular values. Indeed, when $W\ll\sqrt{N}$, we are more in an ergodic regime than a mean field regime, and ideas from ergodic theory become more relevant. It is enlightening to observe that block tridiagonal matrices, a special class of band matrices, have a clear connection to ergodic theory via the transfer matrix recursion formula (Corollary \ref{specialcaselemma1.3}) when computing its determinant. The recursion formula was also used in \cite{MR2191234} for the tridiagonal case. Thus a natural option is to try to generalize the strategy of \cite{MR2191234}. However, an immediate generalization is not feasible. The work \cite{MR2191234} uses in an essential way ergodic theory of matrix products in $\operatorname{GL}_2(\mathbb{C})$ (see \cite{bougerol2012products} and \cite{oseledec1968multiplicative} for an account), but here we have the matrix product of size $2W$ and a total number of $N/W$ matrices, with both quantities tending to infinity. No general quantitative ergodic theory exists for a general random matrix that controls all Lyapunov exponents (but see \cite{cohen1984stability}, \cite{isopi1992triangle}, \cite{MR3737935} for the i.i.d. setting), and when $N/W\gg W$ this matrix product is notoriously difficult to study. As a reference, \cite{MR4268303} studied the Lyapunov exponents of products of Gaussian matrices with growing size, but our transfer matrix $M_k^{(B)}$ is a lot more complicated with general entry law, scalar shift and even an inverse matrix in it, so a direct analysis of Lyapunov exponent is clearly out of reach. 

A surprisingly effective way to tackle this problem is that we can decompose the transfer matrix product in Corollary \ref{specialcaselemma1.3} as the concatenation of many independent systems of transfer matrix products of a smaller size. The starting phase of one system (a unitary $\ell$- frame in $\mathbb{C}^{2\ell}$) depends on the output of the previous system, and the difference in every system is that we start the transfer matrix product at a different phase. Taking a filtration and conditioning, we apply martingale concentration (after we sum over many blocks) and show that the log determinant converges with high probability to its mean. By ergodic theory heuristics, the long term behavior should be independent of the initial phase, so we expect each such smaller system should have approximately  the same large scale limit. Why does this idea of decomposing into smaller systems become helpful? This is because, by Proposition \ref{proposition168}, we can naturally associate each smaller system with a tridiagonal block matrix with the same block size $\ell=W$ but having a much smaller number of blocks $n$. When $\ell$ is large and $n$ is small, say $\ell\gg n^{12}$, we establish convergence of the log determinant for this smaller system by combining the least singular value and rigidity estimates proved below; see Theorem \ref{theorem1.115}. The place to be cautious is that in Proposition \ref{proposition168}, we associate these subsystems to a tridiagonal block matrix with an upper and lower boundary layer determined by the initial phase of recursion, and a large portion of technical effort is to obtain analogous, boundary-independent estimates for this matrix with a general boundary. Combining the convergence in each subsystem allows us to prove the convergence of logarithmic determinant for the large tridiagonal block matrix even if $W\ll \sqrt{N}$.

In short, the proof takes an effective interplay between two distinct methods of computing the log determinant: as the summation of log-singular values and as a transfer matrix recursion. The proof combines ideas from ergodic theory and the circular law proof for near-mean-field band matrices, by introducing a general boundary to account for the different initial phase and showing boundary independence in all ensuing estimates.

\medskip
\noindent\textbf{Logical dependence of the proof.}
To clarify the main structure of this paper, we spell out the dependence structure of the main
arguments. The proof of Theorem \ref{maincircularlaw1.1} is reduced by the replacement principle to the
convergence of
\[
        \frac1{n\ell}\log |\det(T-zI)|
\]
for almost every fixed \(z\in\mathbb C\). Corollary \ref{specialcaselemma1.3} expresses this determinant
as the product of the determinants of the off-diagonal blocks, and the determinant of the output of a transfer matrix iteration. The block determinant contribution is standard, see Proposition \ref{proposition1.999}. The main work is thus to show that the transfer-matrix contribution converges to a specified deterministic value. This is achieved by decomposing the long transfer product
into shorter products of length \(n_0\simeq \ell^d\) and applying the finite-scale
log-determinant convergence theorem, Theorem \ref{theorem1.115}, to each piece. This yields
Corollary \ref{corollary2.17}, which is then used at the end of Section \ref{section22}  to conclude the proof of
Theorem \ref{maincircularlaw1.1}.

Theorem \ref{theorem1.115} is the central finite-scale input. Its proof is completed in Section \ref{section444},
relying on two technical estimates proved elsewhere. The first one is the least
singular value estimate, Theorem \ref{theorem1.112}, proved in Section \ref{leastsingularvalueplugin}. This theorem gives a quantitative lower bound for the smallest singular value. The second one is the
small-singular-value rigidity estimate, Proposition \ref{propositionrigidity}, whose proof is deferred to
Section \ref{959rigidity}. This result shows that there are not too many singular values below a fixed polynomial threshold. Meanwhile,
Corollary \ref{proofcorollary4.5} controls the bulk singular values, which is later used to show that the log summations of singular values converge to the specified limit via a truncation argument. These three estimates prove  Theorem \ref{theorem1.115}. The convergence in expectation follows from uniform integrability in Section \ref{uniformintegrability}, using the upgraded tail estimate
Theorem \ref{updatedtails}.

The proof of
Theorem \ref{theorem1.112} occupies Section \ref{leastsingularvalueplugin}. First, Lemma \ref{lemma3.2} gives a least singular value estimate for
a random rectangular matrix composed by a projection. On the high probability
event \(\mathcal E_K\) of Fact \ref{fact3,3}, Proposition \ref{proposition3.4} describes the geometry (i.e., non-trivial mass on each block) of normal
vectors for the interior rows. The boundary rows require a separate reduction in Fact \ref{fact3.575}
and Proposition \ref{proposition3.6}. Lemma \ref{lemma3.5}, the invertibility-via-distance reduction, then turns these normal-vector
estimates into a least singular value bound for an auxiliary matrix $\mathcal{T}_z^{\text{res,sq}}$ in Proposition \ref{proposition3.85}. Then in Proposition \ref{proposition3333} we use this information to complete our understanding of the geometry of the normal vectors at boundary rows. Combining all these, we prove Theorem \ref{theorem1.112}. The stronger tail estimate Theorem \ref{updatedtails} is obtained by applying the same argument on a stronger high-probability event.

The proof of Proposition \ref{propositionrigidity}  is also a bit involved so we outline its structure. The first part of the proof reduces the problem to a matrix with a simpler form and the second part uses an MDE argument. Lemma \ref{lemma5.1} first removes the general
boundary by a unitary change of variables and reduces the problem to a matrix \(Y_z^R\).
Lemma \ref{lemmas1231} and Lemma \ref{lemma5.2} then show that we can remove two random blocks of $Y_z^R$ so as to reduce \(Y_z^R\) to a block-diagonal form via a probabilistic argument. The large lower block is a canonical block
tridiagonal matrix. For this canonical matrix, Section \ref{section1463} derives the required rigidity
estimate from a local-law/MDE argument: Proposition \ref{proposition4.1} gives the resolvent comparison,
Lemmas \ref{proofoflemma4.2}--\ref{lem:pseudovariance-perturbation} and Corollary \ref{lem:block-tridiagonal-MDE-bound} identify and bound the deterministic MDE solution, and Corollary \ref{rigiditylargeboundcor}
converts this into a bound on the number of small singular values. Combining these
reductions proves Proposition \ref{propositionrigidity}.

\section{Proof of circular law: main structure of the argument}\label{section22}
This section has three components. First, we present the linear algebra result that reduces the determinant computation of block matrices to the transfer matrix problem and set up an inverse problem that, for a given initial value, produces a block tridiagonal matrix. Then we prove moment estimates for the transfer matrix, and apply them in a block martingale argument in Theorem \ref{theorem1.115}.

\subsection{Linear algebra results}
We first introduce a convenient notion $I_{mid}$ for certain block diagonal matrices:
\begin{notation}
    For given integers $n,\ell$, we let $I_{mid}$ be the following $(n+2)\ell\times(n+2)\ell$ diagonal matrix: it has entries 0 on the first $\ell$ rows and the last $\ell$ rows, and it has entries $1$ on the middle $n\ell$ rows.
\end{notation}
We shall need the following computation to handle general boundary conditions:

\begin{Proposition}\label{proposition168}
    Let $A_1,\cdots,A_n,B_1,\cdots,B_n,C_1,\cdots,C_n$ be $\ell\times\ell$ matrices such that $B_i$ are invertible for all $1\leq i\leq n$. Let $\Pi\in\mathbb{C}^{\ell\times 2\ell}$ and $\Xi\in\mathbb{C}^{2\ell\times\ell}$ be fixed matrices such that 
    $$
\det(\Pi\Pi^*)=\det(\Xi^*\Xi)=1,
    $$ where throughout the paper we use $\Xi^*$ to denote the complex conjugate transpose of $\Xi$. Then we can find a block tridiagonal matrix $\mathcal{T}_{n+2}$ of size $(n+2)\ell\times (n+2)\ell$  such that 
    $$
\det\left(\mathcal{T}_{n+2}-zI_{mid}\right)=(-1)^{(n+1)\ell^2}(\prod_{k=1}^n\det B_k)\det(\Pi(\prod_{k=n}^1 M_k^{(B)}(z))\Xi)
    $$ for any $z\in\mathbb{C}$, and where for each $k,$
    $$
M_k^{(B)}(z)=\begin{bmatrix}
    -B_k^{-1}(A_k-zI_\ell)&-B_k^{-1}C_k\\I_\ell&0
\end{bmatrix}, 
    $$ and $\mathcal{T}_{n+2}$ has the form
    \begin{equation}\label{equationnplus2}
\mathcal{T}_{n+2}=\begin{bmatrix}
    V&U&&&&\\C_1&A_1&B_1&&&&\\&C_2&A_2&B_2&&\\&&\ddots&\ddots&\ddots&\\&&&C_n&A_n&B_n\\&&&&S_{+}^{bd}&C_{+}
\end{bmatrix}     \in\mathbb{C}^{(n+2)\ell\times(n+2)\ell},   
    \end{equation} where $\begin{bmatrix}V,U\end{bmatrix},\begin{bmatrix}S_{+}^{bd},C_{+}\end{bmatrix}\in\mathbb{C}^{\ell\times 2\ell}$ are determined only by $\Pi$ and $\Xi$ that satisfy $$\begin{bmatrix}V,U\end{bmatrix}\begin{bmatrix}V,U\end{bmatrix}^*=\begin{bmatrix}S_{+}^{bd},C_{+}\end{bmatrix}\begin{bmatrix}S_{+}^{bd},C_{+}\end{bmatrix}^*=I_\ell.$$
\end{Proposition}
\begin{notation}
    We use the non-standard notation $\prod_{k=n}^1 M_k$ to denote the matrix product $M_nM_{n-1}\cdots M_2M_1$, since matrix multiplication is not commutative. 
\end{notation}

The symbols $\Xi$ and $\Pi$ denote the boundary conditions of this tridiagonal matrix iteration. We can specialize to block tridiagonal matrices as follows:

\begin{corollary}\label{specialcaselemma1.3}
    With the same assumption as Proposition \ref{proposition168}, define the following matrix
    \begin{equation}
\mathcal{T}_n^{tri}=  \begin{bmatrix}
A_1 &B_1&&&\\C_2&A_2&B_2&&\\&\ddots&\ddots&\ddots&\\&&C_{n-1}&A_{n-1}&B_{n-1}\\&&&C_n&A_n
\end{bmatrix}     
    \end{equation}
    then we have for any $z\in\mathbb{C}$, 
    $$
\det(\mathcal{T}_n^{tri}-zI_{n\ell})=(-1)^{n\ell^2}(\prod_{k=1}^n\det B_k)\det\left(\begin{bmatrix}
    I_\ell&0
\end{bmatrix}\prod_{k=n}^1M_k^{(B)}(z)\begin{bmatrix}
    I_\ell\\0
\end{bmatrix}\right).
    $$
\end{corollary}

\begin{proof}[\proofname\ of Proposition \ref{proposition168}]

\textbf{Step 1: forming unitary boundary frames}. We define
$$
Q_{-}:=\Xi(\Xi^*\Xi)^{-1/2}\in\mathbb{C}^{2\ell\times\ell},\quad \widetilde{Q}_+:=(\Pi\Pi^*)^{-1/2}\Pi\in\mathbb{C}^{\ell\times 2\ell}, 
$$so that we necessarily have
$$
Q_{-}^*Q_{-}=I_\ell,\quad \widetilde{Q}_+\widetilde{Q}_{+}^*=I_\ell.
$$Then for any $2\ell\times2\ell$ matrix $K$, we use the following elementary identity to simplify:
\begin{equation}\label{344equation}
    \det(\widetilde{Q}_+KQ_{-})=\frac{\det(\Pi K\Xi)}{\sqrt{\det(\Pi\Pi^*)\det(\Xi^*\Xi)}}=\det(\Pi K\Xi),
\end{equation}using the assumption $\det(\Pi\Pi^*)=\det(\Xi^*\Xi)=1.$

From $Q_{-}$ we choose a unitary complement $Q_{-}^\perp\in\mathbb{C}^{2\ell\times\ell}$ (and thus $(Q_{-}^\perp)^*Q_{-}^\perp=I_\ell$) so that we have $$\begin{bmatrix}Q_{-}^\perp,Q_{-}\end{bmatrix}\in SU(2\ell),$$ the group of unitary matrices of size $2\ell$ and determinant 1 (This can always be achieved via multiplying an arbitrary unitary complement by a unitary $D\in U(\ell)$ with a specified determinant.) We then partition $(Q_{-}^\perp)^*$ into the following form $$(Q_{-}^\perp)^*=\begin{bmatrix}
    U\quad V
\end{bmatrix}\in\mathbb{C}^{\ell\times 2\ell},$$ so that $VV^*+UU^*=I_\ell$ and thus $\|V\|,\|U\|\leq 1 $.
In this proof, we use column labels $0,1,\cdots,n+1$ to index the columns of $\mathcal{T}_{n+2}$. 

We also define the following decomposition 
\begin{equation}
\widetilde{Q}_+=\begin{bmatrix}
    C_{+},S_{+}^{bd}
\end{bmatrix}.
\end{equation}

\textbf{Step 2: Elementary reductions of $\mathcal{T}_{n+2}$ that preserve the determinant.} We introduce three operations that can help us to compute the determinant of $\mathcal{T}_{n+2}-zI_{\text{mid}}$, using $0,1,\cdots,n+1$ to index the column blocks and the row blocks of the matrix $\mathcal{T}_{n+2}-zI_{\text{mid}}$. The following column pairs $(k,k+1)$ all refer to this labeling of the $k$-th and the $k+1$-th block.

 \textbf{2.1: Right multiplication removing $A_k$ and $B_k$.} For each $k=1,\cdots,n,$ we act on the column pair $(k,k+1)$ by the following right multiplier
$$
G_k:=\begin{bmatrix}
    I_\ell&0\\-B_k^{-1}(A_k-zI_\ell)&B_k^{-1}
\end{bmatrix}. 
$$Then we extend $G_k$ to the whole matrix via placing it at the $(k,k+1)$-block columns and block rows, while we preserve the identity elsewhere. Let $\mathcal{G}_k$ denote this constructed matrix, then we have $$\det\mathcal{G}_k=\det G_k=\det(B_k)^{-1}.$$
Note that our choice of $G_k$ always satisfies
$$
\begin{bmatrix}A_k-zI_\ell,B_k\end{bmatrix}G_k=\begin{bmatrix}0,I_\ell\end{bmatrix}.
$$
Then the entry in column $k+1$ is exactly $I_\ell$. After the product $\mathcal{P}$ is applied, the rows $1,\cdots,n$ each contain a unique identity block in distinct columns. 

\textbf{2.2: Right multiplication removing $C_k$.} For each $k=1,2,\cdots,n$, we perform the following elementary column update
$$
\operatorname{Col}_{k-1}\leftarrow \operatorname{Col}_{k-1}-\operatorname{Col}_{k+1}C_k,
$$this is obtained by right multiplication of $\mathcal{T}_{n+2}-zI_{\text{mid}}$ by an elementary block matrix $\mathcal{E}_k$ of determinant 1.

\textbf{2.3: Clearing up the identity via left multiplication.}
For each $1\leq k\leq n$ define the following row operation
$$
\mathcal{L}_k:\text{ Row}_i\leftarrow\operatorname{ Row}_i-(\operatorname{Row}_i\text{ at column  }k+1)\cdot\operatorname{Row}_k,
$$and this is applied for all indices $i\neq k$ having a nonzero entry in the column $k+1$. In our sequential application (see below), each time we apply $\mathcal{L}_k$, the only such row to be modified by $\mathcal{L}_k$ is the bottom row $i=n+1$. Again, we have $\det\mathcal{L}_k=1$.

\textbf{Sequential cleaning}. For each $k=n,\cdots,1$ in decreasing order, we apply, in this order, the following operations on the matrix $\mathcal{T}_{n+2}$:
\begin{enumerate}
    \item  Right pivot on $B_k$: we multiply from the right the matrix $\mathcal{G}_k$, $\det\mathcal{G}_k=\det(B_k)^{-1}$.

\item Column clean-up: we multiply from the right the matrix $\mathcal{E}_k$ with $\det \mathcal{E}_k=1$.

\item  Row clean-up: we multiply from the left the matrix $\mathcal{L}_k$ with $\det \mathcal{L}_k=1$.\end{enumerate}

Set $\mathcal{L}=\prod_{k=1}^n \mathcal{L}_k,\quad \mathcal{P}=\prod_{k=n}^1 (\mathcal{G}_k\mathcal{E}_k)=(\mathcal{G}_n\mathcal{E}_n)\cdots (\mathcal{G}_1\mathcal{E}_1)$ and we denote 
$$
\mathcal{T}^{red}:=\mathcal{L}((\mathcal{T}_{n+2}-zI_{\text{mid}})\mathcal{P}).
$$

To compute the determinant $\det\mathcal{T}^{red}$, we first check that a number of properties are preserved under this sequential cleaning procedure. In the following we let $\rho_0,\rho_1,\cdots,\rho_{n+1}$ denote the entries of the last block row of $\mathcal{T}_{n+2}-zI_{\text{mid}}$, which changes after every iteration.

We claim that after the processing steps labeled $n,n-1,\cdots,k$ (i.e., after we apply $\mathcal{G}_n,\mathcal{E}_n,\mathcal{L}_n,\cdots,\mathcal{G}_k,\mathcal{E}_k,\mathcal{L}_k$), the following induction hypotheses hold for this value $k$:
\begin{itemize}
\item $\operatorname{Ind}(k)$ for interior levels: For each $n\geq j\geq k$, the row $j$ is a pure pivot in that 
$$
\text{row }j=\text{(only) }I_\ell\text{ at column }j+1,\text{ zeros elsewhere}. 
$$

\item $
\operatorname{Ind}_{top}
$: The top row of $\mathcal{T}_{n+2}$ is unchanged and remains $[V,U]$ supported on the columns $(0,1)$.

\item $\operatorname{Ind}_{bot}(k)$, the bottom row sliding pair: We assume that the row $n+1$ is supported only on columns $(k-1,k)$. We record the two blocks in the transfer order, that is column $k$ followed by column $k-1$. Then in this order,
$$
[\rho_k,\rho_{k-1}]=\begin{bmatrix}
    C_{+},S_{+}^{bd}
\end{bmatrix}(\prod_{m=n}^kM_m^{(B)}(z)).
$$
Equivalently, in the physical column order $(k-1,k)$, the two entries are $(\rho_{k-1},\rho_k)$. 
\end{itemize}

\textbf{Verification of induction hypotheses.}  For the initial value $k=n+1$, we have $\operatorname{Ind}(n+1)$ is vacuous, $\operatorname{Ind}_{top}$ is trivially true and $\operatorname{Ind}_{bot}(n+1)$ follows from definition. Then we prove that all these properties hold for all $k$ via taking a backward induction on $k$. 
In each inductive step $k\to k-1$, we assume the induction hypothesis $\operatorname{Ind}(k),\operatorname{Ind}_{top},\operatorname{Ind}_{bot}(k)$.

The application of $\mathcal{G}_{k-1}$ changes only columns $(k-1,k)$,  and then the interior $k-1$-th row becomes $[C_{k-1},0,I_\ell]$ on columns labeled $(k-2,k-1,k)$. The bottom pair $[\rho_{k-1},\rho_k]$ now becomes
$$\begin{bmatrix}\rho_{k-1}-\rho_kB_{k-1}^{-1}(A_{k-1}-zI_\ell),\rho_kB_{k-1}^{-1}\end{bmatrix}
$$ on entries $(k-1,k)$. Then applying $\mathcal{E}_{k-1}$, this zeros out $C_{k-1}$ in row $k-1$ and creates the new entry $-\rho_kB_{k-1}^{-1}C_{k-1}$ into column $k-2$ of the bottom row. The entry $\rho_kB_{k-1}^{-1}$ in column $k$ is still present. Finally, we apply $\mathcal{L}_{k-1}$, which uses the pivot $I_{\ell}$ in row $k-1$, to clear this column-$k$ entry from the bottom row.

Then we conclude that 
$$
[\rho_{k-1},\rho_{k-2}]=[\rho_{k},\rho_{k-1}]M_{k-1}^{(B)}=[C_{+},S_{+}^{bd}]\prod_{m=n}^kM_m^{(B)}(z)\cdot M_{k-1}^{(B)}(z),
$$thus verifying $\operatorname{Ind}(k-1)$ and $\operatorname{Ind}_{bot}(k-1)$. The top row never moves under these maps, so we verify $\operatorname{Ind}_{top}$. This verifies all the induction hypotheses for each $k=1,\cdots,n+1$.

From this inductive proof, we can see that $\mathcal{T}^{red}$ has the following form
$$
\mathcal{T}^{red}=\begin{pmatrix}
    \begin{bmatrix}V,U\end{bmatrix},&&& \\&&I_\ell&&&
\\&&&\ddots&&\\&&&&I_\ell\\\begin{bmatrix}\rho_0,\rho_1\end{bmatrix}
\end{pmatrix},
$$and, in the transfer order, the last row is
$$
[\rho_1,\rho_0]=[C_{+},S_{+}^{bd}]\prod_{k=n}^1M_k^{(B)}(z).
$$

Then let us denote by 
$$
\mathfrak{B}:=\begin{bmatrix}
    U& V\\\rho_1& \rho_0
\end{bmatrix}=\begin{bmatrix}
    U,\quad V\\\begin{bmatrix}C_+,S_{+}^{bd}\end{bmatrix}(\prod_{k=n}^1M_k^{(B)}(z))
\end{bmatrix} .
$$We have $\det\mathcal{T}^{red}=(-1)^{(n+1)\ell^2}\det\mathfrak{B}$, so that we conclude
$$
\det(\mathcal{T}_{n+2}-zI_{\text{mid}})=(-1)^{(n+1)\ell^2}(\prod_{k=1}^n\det B_k)\det\mathfrak{B}.
$$

\textbf{Step 3: computing $\det\mathfrak{B}$.}
By the following Fact \ref{fact3.45}, we have $$\det\mathfrak{B}=\det\begin{bmatrix}\begin{bmatrix}C_+,S_{+}^{bd}\end{bmatrix}\prod_{k=n}^1 M_k^{(B)}(z)Q_{-}\end{bmatrix}=\det\begin{bmatrix}\widetilde{Q}_+(\prod_{k=n}^1 M_k^{(B)}(z))Q_{-}\end{bmatrix},$$ which completes the proof by equation \eqref{344equation}. \end{proof}

\begin{fact}\label{fact3.45}
    Let $R\in\mathbb{C}^{\ell\times 2\ell}$ have orthonormal rows $RR^*=I_\ell$. Let $Q_{-}\in\mathbb{C}^{2\ell\times\ell}$ have unitary columns and $RQ_{-}=0$. Then we have, for any $X\in\mathbb{C}^{\ell\times 2\ell}$,
    $$
\det\begin{bmatrix}
    R\\X
\end{bmatrix}=\det(XQ_{-}),
    $$whenever the $2\ell\times 2\ell$ matrix block $\begin{bmatrix}
        R\\Q_{-}^*
    \end{bmatrix}$ is in $SU(2\ell)$.
\end{fact}
\begin{proof}
    We shall complete $R$ into the following unitary matrix 
    $$
\mathcal{Q}:=\begin{bmatrix}
    R\\R_\perp
\end{bmatrix}\in SU(2\ell),\quad R_\perp^*=Q_{-}.
    $$ Since $\det \mathcal{Q}=1$, and by left multiplication we have
    $$
\begin{bmatrix}
    R\\X
\end{bmatrix}\mathcal{Q}^*=\begin{bmatrix}
    RR^*&RR^*_{\perp}\\XR^*&XR^*_{\perp}
\end{bmatrix}=\begin{bmatrix}
    I_\ell&0\\XR^*&XQ_{-}
\end{bmatrix}.
    $$Thus we have
    $$
\det\begin{bmatrix}
    R\\X
\end{bmatrix}=\det (XQ_{-}),
    $$which completes the proof.
\end{proof}

Now we specialize these computations to the special case in Corollary \ref{specialcaselemma1.3}.
\begin{proof}[\proofname\ of Corollary \ref{specialcaselemma1.3}] For the determinant of $\mathcal{T}_n^{tri}$, one can check that with $\Pi=\Xi^*=[I_\ell,0]$, we have $V=-I_\ell,U=0,C_{+}=I_\ell,S_{+}^{bd}=0$. Taking these special values into definition of $\mathcal{T}_{n+2}$, we have $\det(\mathcal{T}_{n+2}-zI_{mid})=(-1)^{\ell^2}\det(\mathcal{T}_n^{tri}-zI_{n\ell})$, completing the proof.
\end{proof}

\subsection{Least singular value and operator norm bounds}We collect here several standard results on the least singular value and operator norm of a random matrix.

For entry distribution of bounded density, we use the following result of  \cite{tikhomirov2020invertibility}:
\begin{lemma}\label{prop1.5}(\cite{tikhomirov2020invertibility}, Corollary 1.2)
    Let $A$ be an $n\times n$ matrix with i.i.d. real entries of mean 0, variance 1 and having distributional density bounded by $L>0$. Then we can find $C=C(L)>0$ depending only on L such that
    \begin{equation}
        \mathbb{P}(s_{min}(A)\leq \epsilon n^{-1/2})\leq C\epsilon,\quad \epsilon>0.
    \end{equation}
\end{lemma}

The main advantage of Lemma \ref{prop1.5} is that it does not have the exponential error term $e^{-cn}$ which appears when no density assumption is made. For Gaussian matrices, these estimates without exponential error were known much earlier \cite{sankar2006smoothed}. 

For the complex case, we can dispense with the following weaker estimate:

\begin{lemma}\label{lemma2.6}Let $A$ be an $n\times n$ matrix with i.i.d. entries of law $\zeta$, where $\zeta$ is a mean 0, variance 1 random variable satisfying the density assumptions in Theorem \ref{maincircularlaw1.1}. Then we can find $C=C(L)>0$ such that 
 \begin{equation}
        \mathbb{P}(s_{min}(A)\leq \epsilon n^{-3/2})\leq C\epsilon,\quad \epsilon>0.
    \end{equation}
    
\end{lemma}
Lemma \ref{lemma2.6} can be proven in a way similar to Lemma \ref{lemma3.2} and is fairly standard by now, so we omit its proof since it is a straightforward modification of Lemma \ref{lemma3.2}.

We shall need a quantitative operator norm estimate, which follows from taking a truncation to \cite{bai2010spectral}, Theorem 5.9:

\begin{lemma}\label{astronger}  Let $A$ be an $n\times n$ matrix with i.i.d. complex-valued entries $\zeta$ having mean 0, variance 1 and all $p$-th moment finite. Then we can find $C_\zeta>0$ and $C_k>0$ for each $k\in\mathbb{N}_+$, depending only on the moments of $\zeta$ such that the following two estimates hold:
$$
\mathbb{E}\|A\|_{op}\leq C_\zeta\sqrt{n},
   \quad 
\mathbb{P}(\|A\|\geq C_\zeta t\sqrt{n})\leq C_k(tn)^{-k}\quad \forall k\in\mathbb{N}_+,\forall t\geq 1.
$$
\end{lemma}

Finally, we need a result on convergence of log determinants of random matrices. We use the normalization $(\frac{1}{3n})^{1/2}$ since the entries in $T$ are normalized by $(\frac{1}{3\ell})^{1/2}$.
\begin{Proposition}\label{proposition1.999}
    Let $A$ be an $n\times n$ matrix with i.i.d. complex-valued entries $\zeta$ having mean 0 and variance 1. Then 
    $$
\frac{1}{n}\log\left|\det\left((3n)^{-1/2}A\right)\right|\to_{n\to\infty} -\frac{1}{2}\log 3-\frac{1}{2}
    $$where the convergence holds both in probability and almost surely. Assuming moreover that $\zeta$ satisfies the assumptions in Theorem \ref{maincircularlaw1.1}, we also have the convergence in expectation
     $$
\frac{1}{n}\mathbb{E}\log\left|\det\left((3n)^{-1/2}A\right)\right|\to_{n\to\infty} -\frac{1}{2}\log 3-\frac{1}{2}.
    $$
\end{Proposition}
The statement for almost sure convergence follows from the circular law proof in \cite{WOS:000281425000010}, but we will actually use the version of convergence in expectation. We will justify this by proving uniform integrability of the log determinant, so that stronger assumptions (such as a bounded density) are needed. The proof of Proposition \ref{proposition1.999} is deferred to Section \ref{uniformintegrability}.

\subsection{Identification with a transfer operator iteration}

We will essentially be studying the quantity $\det(\Pi(\prod_{k=n}^1 M_k^{(B)}(z))\Xi).$

It is convenient to re-express this quantity in the wedge space $\wedge^\ell\mathbb{C}^{2\ell}$, and write it in the form of an ergodic dynamical system acting on $\wedge^\ell\mathbb{C}^{2\ell}$.

Denote by $\mathcal{W}:=\wedge^\ell\mathbb{C}^{2\ell}$. Then every linear map $g:\mathbb{C}^{2\ell}\to\mathbb{C}^{2\ell}$ lifts to the following linear map 
$$\wedge^\ell g:\mathcal{W}\to \mathcal{W},\quad \wedge^\ell g(v_1\wedge\cdots\wedge v_\ell)=(gv_1)\wedge(gv_2)\wedge\cdots\wedge(gv_\ell),
$$and similarly for any linear map $g:V_1\to V_2$ between two linear spaces $V_1,V_2$, the map $\wedge^\ell g$ is the lift: $\wedge^\ell g:\wedge^\ell V_1\mapsto\wedge^
\ell V_2$. We identify $\wedge^\ell\mathbb{C}^\ell$ with $\mathbb{C}$.

For the given boundary value $\Pi$ and $\Xi$, we let $\hat{f}:=\wedge^\ell \Pi$ and let $\hat{g}:=\wedge^\ell\Xi$, where we can identify $\hat{g}\in\mathcal{W}$ and $\hat{f}\in\mathcal{W}^*$, where $\mathcal{W}^*$ is the dual space of $\mathcal{W}$. Then we have the identity 
\begin{equation}\label{identity471559}
\det(\Pi(\prod_{k=n}^1 M_k^{(B)}(z))\Xi)=\langle\hat{f},\prod_{k=n}^1\wedge^\ell M_k^{(B)}(z)\hat{g}\rangle.
\end{equation}This equality follows from writing the Cauchy-Binet formula in the language of exterior algebra.

The following property will be frequently used:
\begin{fact}\label{fact2.100}
    Let $g:\mathbb{C}^{2\ell}\to \mathbb{C}^{2\ell}$ be a linear map with singular values $\sigma_1\geq\sigma_2\geq\cdots\geq\sigma_{2\ell}$. Then the singular values of $\wedge^\ell g$ are given by the following collection 
    $$
\{\sigma_{i_1}\sigma_{i_2}\cdots\sigma_{i_\ell}:\quad 1\leq i_1<i_2<\cdots<i_\ell\leq 2\ell\}.
    $$
\end{fact}

Then we introduce two frequently used subspaces of $\mathcal{W}$. Let $\operatorname{Gr}:=\operatorname{Gr}(\ell,2\ell)$ be the Grassmannian of $\ell$-planes in $\mathbb{C}^{2\ell}$, then $\operatorname{Gr}(\ell,2\ell)$ is naturally identified as a subspace of $\mathbb{P}(\mathcal{W})$, the projective space of $\mathcal{W}$, via Plücker embedding. We also define the Stiefel manifold
$$
\operatorname{St}(\ell,2\ell):=\{\Pi\in\mathbb{C}^{\ell\times 2\ell}:\Pi\Pi^*=I_\ell\}.
$$ Then we have the natural identification, for $U(\ell)$ the unitary group, 
$$
\operatorname{Gr}(\ell,2\ell)=\operatorname{St}(\ell,2\ell)/U(\ell).
$$
We also introduce the decomposable cone of $\mathcal{W}$, which spans $\mathcal{W}$:
$$
\mathcal{C}_{dec}:=\{v_1\wedge\cdots\wedge v_\ell\in \mathcal{W}\setminus \{0\}\}.
$$There is a natural identification of its projectivization with the Grassmannian: 
$$
\mathbb{P}\mathcal{C}_{dec}\simeq \operatorname{Gr}(\ell,2\ell).
$$

We start from $Q_0:=\hat{g}\in \mathcal{C}_{dec}\subseteq \mathcal{W}$ and proceed inductively.
Given a current $\ell$-frame $Q_{k-1}\in \mathcal{C}_{dec}\subseteq \mathcal{W}$, we take one further matrix multiplication by $\wedge^\ell M_{k}^{(B)}$ and set
$$
Q_k:=\wedge^\ell M_{k}^{(B)}Q_{k-1}\in \mathcal{C}_{dec}\subseteq \mathcal{W}. 
$$Let $X_k=[Q_k]$ be the projective chain on the projective space $\mathbb{P}\mathcal{C}_{dec}\simeq \operatorname{Gr}(\ell,2\ell)$. We define the following function 
$$g_\ell(M_k^{(B)}(z),\hat{g}):=\log\frac{\|\wedge^\ell M_k^{(B)}(z)\hat{g}\|}{\|\hat{g}\|}$$ for any $\hat{g}\in\mathcal{C}_{dec}.$
Then we may decompose the additive cocycle as 
\begin{equation}\label{eq535}
\log\left|\det(\Pi(\prod_{k=n}^1 M_k^{(B)}(z))\Xi)\right|=\sum_{k=1}^{n-1} g_\ell(M_{k}^{(B)},Q_{k-1})+\log\left|\frac{\langle\hat{f},\wedge^\ell M_n^{(B)}(z)Q_{n-1}\rangle}{\|Q_{n-1}\|}\right|.
\end{equation}
Here $g_\ell(M,Q)$ depends only on $[Q]$, but for formulas involving the norms of inner products we always use a representative $Q$ for $[Q]$.

\subsection{Exponential moments for intermediate increment steps}\label{section2.11110}

\begin{lemma}\label{lemma1.110} Assume $\ell\geq2$.
    Consider the function $$g_\ell(M_k^{(B)}(z),\hat{g}):=\log\frac{\|\wedge^\ell M_k^{(B)}(z)\hat{g}\|}{\|\hat{g}\|}$$ for any $\hat{g}\in\mathcal{C}_{dec}$. Assume that
    the random blocks $A_k,B_k,C_k$ in $M_k^{(B)}(z)$ satisfy the same assumption as in Theorem \ref{maincircularlaw1.1}. Then we can find two constants $c_0$ and $C_0$ (depending only on $|z|$ and the entry law of $M_k^{(B)}(z)$) such that, uniformly for any $\hat{g}\in\mathcal{C}_{dec},$ the following estimate holds for all $-\frac{c_0}{\ell}\leq c\leq \frac{c_0}{\ell}$:
    $$   \mathbb{E}\exp \left(cg_\ell(M_k^{(B)}(z),\hat{g})\right)\leq C_0\ell^{2}.
    $$We also have the following moment computations:
    $$
\quad\mathbb{E}[|g_\ell(M_k^{(B)}(z),\hat{g})|^2]\leq C_0(\ell\log\ell)^2,\quad \quad\mathbb{E}[|g_\ell(M_k^{(B)}(z),\hat{g})|^6]\leq C_0(\ell\log\ell)^6.
    $$
\end{lemma}

\begin{proof}
By Fact \ref{fact2.100}, we have
$$
s_{min}(M_k^{(B)}(z))^\ell\leq\|\wedge^\ell M_k^{(B)}(z)\|\leq s_{max}(M_k^{(B)}(z))^\ell, 
$$and similarly we have
$$\|(\wedge^\ell M_k^{(B)}(z))^{-1}\|
\leq s_{min}(M_k^{(B)}(z))^{-\ell}=s_{\max}\big((M_k^{(B)}(z))^{-1}\big)^\ell.
$$
We can bound 
$$
\|M_k^{(B)}(z)\|\leq \|B_k^{-1}\|(\|B_k\|+\|A_k\|+|z|+\|C_k\|). 
$$Then by Lemma \ref{prop1.5}, \ref{lemma2.6} and \ref{astronger} and applying Cauchy-Schwarz inequality, we can find constants $c_0>0,C_0>0$ depending only on $z$ and $\zeta$ such that for any $c\in(0,c_0)$, we have
$$
\mathbb{E}\|M_k^{(B)}(z)\|^c \leq C_0\ell^{2}.
$$Since the entries $C_k$ have bounded density, $C_k$ is invertible almost surely.
We then compute $$(M_k^{(B)}(z))^{-1}=\begin{bmatrix}
    0&I_\ell\\-C_k^{-1}B_k&-C_k^{-1}(A_k-zI_\ell)
\end{bmatrix},$$so that for the same choice of $c_0,C_0$ we have for all $c\in(0,c_0)$,
$$
\mathbb{E}\|(M_k^{(B)}(z))^{-1}\|^c\leq C_0\ell^{2}.
$$ Combining these two bounds yields the first estimate. For the estimate on variance, by these moment computations we verify that $\mathbb{E}\left[\left|\log\|\wedge^\ell M_k^{(B)}(z)\|\right|^2\right]\leq C_0\ell^2\log^2\ell$ for a sufficiently large $C_0$. The computation for the sixth moment is exactly the same.
\end{proof}
Via a more involved argument, we can show that 
\begin{lemma}\label{exactlythesame} Assume $\ell\geq2$. We take $\hat{f}=\wedge^\ell[I_\ell,0]$.
    For the two constants $c_0,C_0$ in Lemma \ref{lemma1.110}, we have uniformly for any $\hat{g}\in \mathcal{C}_{dec}$ with $\|\hat{g}\|=1$, for all $-\frac{c_0}{\ell}\leq c\leq \frac{c_0}{\ell}$:
    $$
\mathbb{E}\exp\left(c\log|\langle\hat{f},\wedge^\ell M_n^{(B)}(z)\hat{g}\rangle| \right)\leq C_0\ell^2.
    $$We also have the following moment computations:
    $$
\mathbb{E}[|\log|\langle\hat{f},\wedge^\ell M_n^{(B)}(z)\hat{g}\rangle||^2]\leq C_0(\ell\log\ell)^2,\quad \mathbb{E}[|\log|\langle\hat{f},\wedge^\ell M_n^{(B)}(z)\hat{g}\rangle||^6]\leq C_0(\ell\log\ell)^6.
    $$
\end{lemma}

\begin{proof}
We can find some $\Xi\in\mathbb{C}^{2\ell\times\ell}$, $\det(\Xi^*\Xi)=1$, such that $\hat{g}=\wedge^\ell\Xi$ since $\hat{g}$ is decomposable with unit norm. Then by \eqref{identity471559}, we only need to compute the moments of $\det(\Pi M_n^{(B)}(z)\Xi)$, with $\Pi=[I_\ell,0]$. Then by Proposition \ref{proposition168} with $n=1$ and $\Pi=[I_\ell,0]$, so that Proposition \ref{proposition168} yields  $[S_{+}^{bd},C_{+}]=[0,I_\ell]$ and $[V,U]$ depending on $\Xi$, we have
$$
\log|\det (\Pi M_n^{(B)}(z)\Xi)|=-\log|\det B_n|+\log\left|\det \begin{bmatrix}
V&U&0\\C_n&A_n-zI_\ell&B_n\\0&0&I_\ell
\end{bmatrix}\right|.
$$ We simply use \begin{equation}\label{imple}|\log|\det B_n||\leq \ell|\log s_{max}(B_n)|+\ell|\log s_{min}(B_n)|
\end{equation}for the first term, and it suffices to bound $|\log|\det E_n||$ where $$E_n=\begin{bmatrix} V&U\\
    C_n&A_n-zI_\ell\\ 
\end{bmatrix}.$$ By \eqref{imple} with $B_n$ replaced by $E_n$, we only need a lower tail estimate for $s_{min}(E_n)$ as the upper tail estimate for $s_{max}(E_n)$ follows from Lemma \ref{astronger}. We outline a direct proof of this lower-tail estimate here. The idea is as follows: put $R=[V,U]$, $W_1=\operatorname{ran}R^*$ and $W_2=\ker R=W_1^\perp$; since $RR^*=I_\ell$, $R$ is an isometry on $W_1$, and we restrict $[C_n,A_n-zI_\ell]$ to $W_2$. By Lemma \ref{lemma3.2}, the restricted mapping has least singular value bounded by \eqref{lemma3.2estimates}, so that we have $\mathbb{P}(s_{min}([C_n,A_n-zI_\ell]\mid_{W_2})\leq\epsilon \ell^{-2})\leq C_L\epsilon$ for all $\epsilon>0$ (we switch from $\ell^{-1.5}$ there to $\ell^{-2}$ here since the entries are normalized by $(3\ell)^{-1/2}$.) We also need an operator norm estimate: by Lemma \ref{astronger} we can assume that $\mathbb{P}(\|E_n\|\geq 4K_z\epsilon^{-1})\leq\epsilon$ for all $\epsilon\leq 1$, where $K_z>1$ is a constant depending only on $|z|$ and $\zeta$. Now for any unit vector $w\in\mathbb{C}^{2\ell}$ with orthonormal decomposition $w=w_1+w_2,w_1\in W_1,w_2\in W_2$, if $\|w_1\|\leq \epsilon^2\ell^{-2}(16K_z)^{-1}$, then on the event where $s_{min}$ is not too small and $\|E_n\|$ is bounded as above, we have that $$\|E_n w\|\geq s_{min}([C_n,A_n-zI_\ell]\mid _{W_2})\|w_2\|-\|[C_n,A_n-zI_\ell]\|\|w_1\|\geq \epsilon\ell^{-2}(16K_z)^
{-1}.$$ If $\|w_1\|\geq \epsilon^2\ell^{-2}(16K_z)^{-1}$ we also trivially have $\|E_nw\|\geq \epsilon^2\ell^{-2}(16K_z)^{-1}$. This completes the proof that $\mathbb{P}(s_{min}(E_n)\leq\epsilon^2\ell^{-2}(16K_z)^{-1})\leq2C_L\epsilon$ for any $\epsilon>0$. 

Combining all these estimates leads us to the result of Lemma \ref{exactlythesame}. Finally, we can slightly modify the constants $c_0,C_0$ in Lemma \ref{lemma1.110}, so they also work here.
\end{proof}

 Then we have the following estimate:

\begin{corollary}\label{corollary1.133}Under the same assumptions in Lemma \ref{lemma1.110} and Lemma \ref{exactlythesame}, denote by  \begin{equation}\operatorname{Pro}_{n,\ell}^z(\Pi,\Xi):=\log|\det(\Pi(\prod_{k=n}^1 M_k^{(B)}(z))\Xi)|,\quad \operatorname{Pro}_{n,\ell}^z(\Xi):=\log\|(\prod_{k=n}^1 \wedge^\ell M_k^{(B)}(z)\cdot \wedge^\ell\Xi)\|,\end{equation} where $\Pi=[I_\ell,0]$ and $\Xi^*$ is taken from the Stiefel manifold $\operatorname{St}(\ell,2\ell)$. Then
\begin{enumerate}\item
The values of these maps do not depend on the representative of $[\Xi^*]\in\operatorname{Gr}(\ell,2\ell)$ under the identification $\operatorname{Gr}(\ell,2\ell)=\operatorname{St}(\ell,2\ell)/U(\ell)$.  
\item Moreover, the following sixth-moment bounds hold. Let \(m\ge 1\), and let
\(\mathcal F_0\) be a sigma-field independent of the transfer matrices
\(M_1^{(B)}(z),\ldots,M_m^{(B)}(z)\). Let \(\Xi\) be an
\(\mathcal F_0\)-measurable boundary frame satisfying \(\det(\Xi^*\Xi)=1\).
Then, uniformly in \(\Xi\),
\[
\mathbb E\left[
\left|\operatorname{Pro}_{m,\ell}^z(\Xi)\right|^6
\mid \mathcal F_0
\right]
\le
C(m\ell\log\ell)^6,
\]
and
\[
\mathbb E\left[
\left|\operatorname{Pro}_{m,\ell}^z(\Pi,\Xi)\right|^6
\mid \mathcal F_0
\right]
\le
C(m\ell\log\ell)^6.
\]
Consequently, we have the centered moment estimates
\[
\mathbb E\left[
\left|
\operatorname{Pro}_{m,\ell}^z(\Xi)
-
\mathbb E[\operatorname{Pro}_{m,\ell}^z(\Xi)\mid \mathcal F_0]
\right|^6
\mid \mathcal F_0
\right]
\le
C(m\ell\log\ell)^6,
\]
and
\[
\mathbb E\left[
\left|
\operatorname{Pro}_{m,\ell}^z(\Pi,\Xi)
-
\mathbb E[\operatorname{Pro}_{m,\ell}^z(\Pi,\Xi)\mid \mathcal F_0]
\right|^6
\mid \mathcal F_0
\right]
\le
C(m\ell\log\ell)^6.
\]
Here the constant \(C\) depends only on \(z\) and on the law $\zeta$ of the entries.

\end{enumerate}
\end{corollary}
\begin{proof}The independence of values on the representation follows from Cauchy–Binet formula and the fact that for any $R\in\mathbb{C}^{\ell\times\ell}$, we have that $\wedge^\ell (\Xi R)=(\det R)\wedge^\ell\Xi$.

To prove the sixth-moment bound, we condition on \(\mathcal F_0\). Then the initial frame
\(\Xi\) is fixed, while the transfer matrices in the block are independent of \(\mathcal F_0\).

First consider \(\operatorname{Pro}_{m,\ell}^z(\Xi)\). We write each phase as
\[
Q_0=\wedge^\ell\Xi,
\qquad
Q_k=
\left(\prod_{s=k}^{1}\wedge^\ell M_s^{(B)}(z)\right)\wedge^\ell\Xi
,\qquad X_k:=[Q_k],
\]
we have the additive decomposition
\[
\operatorname{Pro}_{m,\ell}^z(\Xi)
=
\sum_{k=1}^{m}
g_\ell(M_k^{(B)}(z),Q_{k-1}).
\]
For each \(k\), \(Q_{k-1}\) is measurable with respect to
\[
\mathcal F_{k-1}:=
\mathcal F_0\vee \sigma(M_1^{(B)}(z),\ldots,M_{k-1}^{(B)}(z)),
\]
and \(M_k^{(B)}(z)\) is independent of \(\mathcal F_{k-1}\). Therefore Lemma \ref{lemma1.110} gives
\[
\mathbb E\left[
\left|
g_\ell(M_k^{(B)}(z),Q_{k-1})
\right|^6
\mid \mathcal F_{k-1}
\right]
\le
C_0(\ell\log\ell)^6.
\]
Taking conditional expectation once more gives
\[
\mathbb E\left[
\left|
g_\ell(M_k^{(B)}(z),Q_{k-1})
\right|^6
\mid \mathcal F_0
\right]
\le
C_0(\ell\log\ell)^6.
\]
By conditional Minkowski,
\[
\begin{aligned}
\left(
\mathbb E\left[
\left|\operatorname{Pro}_{m,\ell}^z(\Xi)\right|^6
\mid \mathcal F_0
\right]
\right)^{1/6}
&\le
\sum_{k=1}^{m}
\left(
\mathbb E\left[
\left|
g_\ell(M_k^{(B)}(z),Q_{k-1})
\right|^6
\mid \mathcal F_0
\right]
\right)^{1/6}  \\
&\le
C m\ell\log\ell.
\end{aligned}
\]
This proves
\[
\mathbb E\left[
\left|\operatorname{Pro}_{m,\ell}^z(\Xi)\right|^6
\mid \mathcal F_0
\right]
\le
C(m\ell\log\ell)^6.
\]

The proof for \(\operatorname{Pro}_{m,\ell}^z(\Pi,\Xi)\) is the same, using the decomposition
\[
\operatorname{Pro}_{m,\ell}^z(\Pi,\Xi)
=
\sum_{k=1}^{m-1}
g_\ell(M_k^{(B)}(z),Q_{k-1})
+
\log
\frac{
|\langle \hat f,\wedge^\ell M_m^{(B)}(z)Q_{m-1}\rangle|
}{\|Q_{m-1}\|}.
\]
The first \(m-1\) terms are controlled by Lemma \ref{lemma1.110}. For the final term, after replacing
\(Q_{m-1}\) by \(Q_{m-1}/\|Q_{m-1}\|\), Lemma \ref{exactlythesame} gives
\[
\mathbb E\left[
\left|
\log
\frac{
|\langle \hat f,\wedge^\ell M_m^{(B)}(z)Q_{m-1}\rangle|
}{\|Q_{m-1}\|}
\right|^6
\mid \mathcal F_{m-1}
\right]
\le
C(\ell\log\ell)^6.
\]
Another application of conditional Minkowski yields
\[
\mathbb E\left[
\left|\operatorname{Pro}_{m,\ell}^z(\Pi,\Xi)\right|^6
\mid \mathcal F_0
\right]
\le
C(m\ell\log\ell)^6.
\]

Finally, the centered bounds follow from conditional Jensen and the triangle inequality:
\[
\begin{aligned}
&\left(
\mathbb E\left[
\left|
Z-\mathbb E[Z\mid\mathcal F_0]
\right|^6
\mid \mathcal F_0
\right]
\right)^{1/6} \\
&\qquad\le
\left(
\mathbb E\left[
|Z|^6
\mid \mathcal F_0
\right]
\right)^{1/6}
+
|\mathbb E[Z\mid\mathcal F_0]| \\
&\qquad\le
2
\left(
\mathbb E\left[
|Z|^6
\mid \mathcal F_0
\right]
\right)^{1/6}.
\end{aligned}
\]
Applying this with
\[
Z=\operatorname{Pro}_{m,\ell}^z(\Xi)
\quad\text{or}\quad
Z=\operatorname{Pro}_{m,\ell}^z(\Pi,\Xi)
\]
gives the two centered estimates.
\end{proof}

\begin{remark}\label{wecananalogously}
   By the second claim of Proposition \ref{proposition1.999},  $$\frac{1}{n\ell}\sum_{i=1}^n\log|\det B_i|\to -\frac{1}{2}\log 3-\frac{1}{2},\quad a.s.,\quad  n\to\infty,$$and the same holds with $C_i$ in place of $B_i$. Indeed, the expectation converges by Proposition \ref{proposition1.999}. For the centered part, the variables
\(\log|\det B_i|-\mathbb E\log|\det B_i|\) are independent, and the same proof as in
Lemma \ref{exactlythesame} offers a sixth-moment bound \(C(\ell\log\ell)^6\). Hence,
\[
\mathbb E\left|\sum_{i=1}^n
(\log|\det B_i|-\mathbb E\log|\det B_i|)
\right|^6
\le C n^3(\ell\log\ell)^6.
\]
After division by \(n\ell\), Markov's inequality yields an error probability $C_\eta n^{-3}(\log\ell)^6$. Since $\ell=O(\operatorname{Poly}(n))$, the error is summable and Borel--Cantelli leads to the almost sure convergence.
\end{remark}

\subsection{Decomposition of the ergodic sum}
Corollary \ref{specialcaselemma1.3} reduces the problem to studying the additive cocycle $\operatorname{Pro}_{n,\ell}^z$. We expect that the value of $\frac{1}{n\ell}\mathbb{E}\operatorname{Pro}_{n,\ell}^z(\Pi,\Xi)$ will be asymptotically independent of the precise speed for $n$ and $\ell$ tending to infinity, and asymptotically independent of the initial frame $\Xi$. The same asymptotics should also hold for $\operatorname{Pro}_{n,\ell}^z(\Xi)$. We now rigorously justify this heuristic by identifying the asymptotic value of these short block systems, and eventually use it to prove the circular law theorem for slowly growing $\ell$. The following theorem is the most essential technical step in verifying these heuristics:
\begin{theorem}\label{theorem1.115}
    Take $\Pi=[I_\ell,0]$ and consider any frame $\Xi\in \mathbb{C}^{2\ell\times\ell}$ satisfying $\det(\Xi^*\Xi)=1$. Consider the block tridiagonal matrix $\mathcal{T}_{n+2}$ defined in equation \eqref{equationnplus2}, where all assumptions of Theorem \ref{maincircularlaw1.1} are verified. We assume that $\ell^{1/12}\geq n\geq \ell^d$ for some fixed $d\in(0,\frac{1}{12})$. Then for any $z\in\mathbb{C}$, the following limit holds in probability as $n,\ell\to\infty$:
    $$
\frac{1}{n\ell}\log|\det(\mathcal{T}_{n+2}-zI_{mid})|\to\mathcal{U}(z)=\begin{cases}
\frac{|z|^2-1}{2}    ,\quad |z|\leq 1,\\\log|z|,\quad |z|>1.
\end{cases}
    $$We also have the convergence in expectation as $n,\ell\to\infty$:
    $$
\frac{1}{n\ell}\mathbb{E}\log|\det(\mathcal{T}_{n+2}-zI_{mid})|\to\mathcal{U}(z).
    $$
Moreover, the convergence rate in probability and in expectation can be quantified to be uniform over the choice of frame $\Xi$, and the rate depends only on $\zeta$, $|z|$ and $d>0$.
\end{theorem}
The function $\mathcal{U}(z)$ is the well-known Ginibre potential and Theorem \ref{theorem1.115} is essentially stating the circular law for $\mathcal{T}_{n+2}$. What is important here is the relative magnitude of $\ell$ and $n$: it allows $n$ to be polynomial growth in $\ell$ but still require the block size $\ell$ to be much larger than the number of blocks $n$, so that we are close to the delocalization regime and should have mean-field type behavior for the matrix (although this is not yet rigorously proven). We have stated the theorem in the special case $\Pi=[I_\ell,0]$ to simplify its proof as this is the only case we will actually use, but it is not hard to verify that the proof works for general $\Pi$ via the same argument.

For a large $n$ and a given $\ell$, we use the following fact to decompose $[n]$ into smaller components so that Theorem \ref{theorem1.115} can be applied separately to each component. 
\begin{fact}\label{fact2.16}
    Let $n\geq 10\ell^{d}$ for some $d>0$, and set $n_0=\lceil2\ell^d\rceil$ and $q=\lfloor n/n_0\rfloor-1$. Then $2\ell^d\leq n_0\leq3\ell^d$ and the final interval has length $r=n-qn_0$ satisfying $n_0\leq r<2n_0$.
\end{fact}
\begin{proof}
Write $n=(q+1)n_0+s$ with $0\leq s<n_0$. Then $r=n_0+s$.
\end{proof}
\begin{corollary}\label{corollary2.17} Assume that $\ell^{1/12}\geq n\geq \ell^d$ for some $d>0$, and take $\Pi=(I_\ell,0)$. Then for any $\Xi\in\mathbb{C}^{2\ell\times\ell}$ with $\det(\Xi^*\Xi)=1$, we have for any fixed $z\in\mathbb{C}$, 
    $$\mathbb{E}
\frac{1}{n\ell}\operatorname{Pro}_{n,\ell}^z(\Pi,\Xi)\to \mathcal{U}(z)+\frac{1}{2}\log 3+\frac{1}{2},\quad \mathbb{E}
\frac{1}{n\ell}\operatorname{Pro}_{n,\ell}^z(\Xi)\to \mathcal{U}(z)+\frac{1}{2}\log 3+\frac{1}{2}
    $$ and the convergence rate is uniform over the choice of the frame $\Xi$.
\end{corollary}
\begin{proof}
   For the convergence regarding $\operatorname{Pro}_{n,\ell}^z(\Pi,\Xi)$, we simply take this $\Pi$ and $\Xi$ and apply Theorem \ref{theorem1.115}. Indeed, we write 
   $$
\log|\det(\mathcal{T}_{n+2}-zI_{\text{mid}})|=\sum_{k=1}^n\log|\det B_k|+\operatorname{Pro}_{n,\ell}^z(\Pi,\Xi).
   $$We take the expectation and divide by $n\ell$. Theorem \ref{theorem1.115} gives the limit in expectation of the log determinant, and Proposition \ref{proposition1.999} gives the expectation limit for $\sum_k\log|\det B_k|$, so the claim follows. For the convergence regarding \(\operatorname{Pro}_{n,\ell}^z(\Xi)\), write
\[
Q_{n-1}:=\left(\prod_{k=n-1}^{1}\wedge^\ell M_k^{(B)}(z)\right)\wedge^\ell\Xi.
\]
Then
\[
\operatorname{Pro}_{n,\ell}^z(\Pi,\Xi)
=
\operatorname{Pro}_{n-1,\ell}^z(\Xi)
+
\log\frac{
|\langle \hat f,\wedge^\ell M_n^{(B)}(z)Q_{n-1}\rangle|
}{\|Q_{n-1}\|}.
\]
Taking conditional expectation over \(Q_{n-1}\), Lemma \ref{exactlythesame} gives the upper bound
\[
\left|
\mathbb E
\log\frac{
|\langle \hat f,\wedge^\ell M_n^{(B)}(z)Q_{n-1}\rangle|
}{\|Q_{n-1}\|}
\right|
\le C\ell\log\ell .
\]
Therefore the following difference is small:
\[
\frac1{n\ell}
\left|
\mathbb E\operatorname{Pro}_{n,\ell}^z(\Pi,\Xi)
-
\mathbb E\operatorname{Pro}_{n-1,\ell}^z(\Xi)
\right|
\le
\frac{C\log\ell}{n}=o(1).
\]
The difference between
\(\operatorname{Pro}_{n-1,\ell}^z(\Xi)\) and
\(\operatorname{Pro}_{n,\ell}^z(\Xi)\) is upper bounded in the same way by Lemma \ref{lemma1.110}:
\[
\frac1{n\ell}
\left|
\mathbb E\operatorname{Pro}_{n,\ell}^z(\Xi)
-
\mathbb E\operatorname{Pro}_{n-1,\ell}^z(\Xi)
\right|
\le
\frac{C\log\ell}{n}=o(1).
\]
This proves the second convergence, uniformly in \(\Xi\).
   \end{proof}

With these technical results, we can quickly conclude the proof of Theorem \ref{maincircularlaw1.1}.

\begin{proof}[\proofname\ of Theorem \ref{maincircularlaw1.1}] 
The eighth moment assumption gives $\mathbb{E}|\|T\|_F^2/N-\mathbb{E}(\|T\|_F^2/N)|^4\leq Cn^{-2}$; thus Borel--Cantelli gives $\frac{1}{N^2}\|\sqrt{N}T\|_F^2=O(1)$ almost surely, justifying criterion (1) of Theorem \ref{replacement}.

We apply the replacement principle (Theorem \ref{replacement}) between the matrix $\sqrt{N}T$ and an independent complex Ginibre matrix. By the determinantal expansion in Corollary \ref{specialcaselemma1.3} and the almost sure convergence of $\frac{1}{n\ell}\sum_{k=1}^n\log|\det B_k|$ term (which follows from Proposition \ref{proposition1.999} and Remark \ref{wecananalogously}), we only need to show the a.s. convergence of 
\begin{equation}\label{line65000}\frac{1}{n\ell}\log|\langle \wedge^\ell\Pi,\prod_{k=n}^1\wedge^\ell M_k^{(B)}(z)\cdot\wedge^\ell\Xi\rangle | 
 \end{equation} to the deterministic limit $\mathcal{U}(z)+\frac{1}{2}\log 3+\frac{1}{2}$ for $\Pi=\begin{bmatrix}
     I_\ell,0
 \end{bmatrix}$ and $\Xi=\begin{bmatrix}
     I_\ell\\0
 \end{bmatrix}.$

 We choose a smaller value $d_0\in(0,\min(\frac{2d}{3},\frac{1}{12}))$.
 By Fact \ref{fact2.16}, set $n_0=\lceil2\ell^{d_0}\rceil$ and $\underline{n}=\lfloor n/n_0\rfloor-1$. Then $\underline{n}\geq n/n_0-2\geq\frac14\ell^{d-d_0}-2\to\infty$, and $\underline{n}\geq\frac14 n^{1-d_0/d}-2$. We decompose the integer interval $[1,n]$
as the union $\{[(k-1)n_0+1,kn_0]\}_{1\leq k\leq\underline{n}}\cup[n_0\underline{n}+1,n]$. The last interval has length $r=n-n_0\underline{n}\in[n_0,2n_0)$. Since $d_0<1/12$, all these lengths lie in $[\ell^{d_0},\ell^{1/12}]$ for sufficiently large $\ell$. Recall that we denote by $Q_k:=\prod_{m=k}^1\wedge^\ell M_m^{(B)}(z)\cdot\wedge^\ell\Xi$ and $X_k=[Q_k]$ its projection on the Grassmannian $\operatorname{Gr}(
\ell,2\ell)$. 

Denote by $$
Z_j:=\operatorname{Pro}^z_{n_0,\ell}(X_{(j-1)n_0}),\quad m_j:=\mathbb{E}[Z_j\mid\mathcal{G}_{j-1}],
$$
where $\mathcal{G}_j=\sigma(M_1^{(B)}(z),\ldots,M_{jn_0}^{(B)}(z))$ and $\mathcal{G}_0$ is trivial. Here $Z_j$ is computed using the transfer matrices with indices $(j-1)n_0+1,\cdots,jn_0$ with incoming boundary frame $X_{(j-1)n_0}$. Thus $Z_j-m_j$ forms a martingale difference sequence. For the final interval we use 
$$
Z_{\text{rem}}:=\operatorname{Pro}^z_{r,\ell}(\Pi,X_{\underline{n}n_0}),\quad m_{\text{rem}}:=\mathbb{E}[Z_{\text{rem}}\mid \mathcal{G}_{\underline{n}}].
$$
Then 
$$
\operatorname{Pro}_{n,\ell}^z(\Pi,\Xi)=\sum_{j=1}^{\underline{n}}m_j+m_{\text{rem}}+\sum_{j=1}^{\underline{n}}(Z_j-m_j)+(Z_{\text{rem}}-m_{\text{rem}}).
$$
The first summation $\sum_{j=1}^{\underline{n}}m_j+m_{\text{rem}}$, divided by $n\ell$, converges to the specified limit by the uniform-over-$\Xi$ statement of Corollary \ref{corollary2.17}. Indeed, it applies to both lengths $n_0,r$, and the conditional mean errors are weighted by the lengths, whose sum is $\underline{n}n_0+r=n$. For the second part, the summation $\sum_j(Z_j-m_j)$ is a martingale so we use the sixth moment bound and Burkholder/Rosenthal in item (2) of Corollary \ref{corollary1.133} (which implies $\mathbb{E}[|Z_j-m_j|^6\mid\mathcal{G}_{j-1}]\leq C(n_0\ell\log\ell)^6$ uniformly in the incoming frame) to bound that 
$$
\mathbb{E}\left|\sum_{j=1}^{\underline{n}}(Z_j-m_j)\right|^6\leq C\underline{n}^3(n_0\ell\log\ell)^6.
$$
Then by Markov's inequality, for any $\eta>0$, 
$$
\mathbb{P}(|\sum_{j=1}^{\underline{n}}(Z_j-m_j)|\geq\eta n\ell)
\leq C_\eta \underline{n}^{-3}(\log\ell)^6,
$$so that the second summation converges almost surely to 0 after dividing by $n\ell$ (the error term $\underline{n}^{-3}(\log\ell)^6$ is summable in $n$ since $\underline{n}\geq\frac{1}{4}n^{1-d_0/d}-2,d_0/d<\frac{2}{3}$, and apply Borel-Cantelli). For the third part, $r<2n_0$ gives the analogous sixth-moment bound
$$
\mathbb{P}(|Z_{\text{rem}}-m_{\text{rem}}|\geq\eta n\ell)\leq C_\eta \underline{n}^{-6}(\log\ell)^6,
$$
so that $\frac{Z_{\text{rem}}-m_{\text{rem}}}{n\ell}\to 0$, and the convergence is almost sure by Borel-Cantelli. Since the convergence \eqref{line65000} is proven for each $z\in\mathbb{C}$, this completes the proof of circular law. \end{proof}

\section{Least singular value for the auxiliary matrix}\label{leastsingularvalueplugin}
In this section, we derive a least singular value estimate in Theorem \ref{theorem1.112} as the first step to the proof of Theorem \ref{theorem1.115}. Throughout this section, assume $2\leq n\leq\ell$.

We first prove the following high probability least singular value bound for $\mathcal{T}_{n+2}-zI_{mid}$, whose matrix structure we recall here:
\begin{equation}
\mathcal{T}_{n+2}-zI_{mid}=\begin{bmatrix}
    V&U&&&&\\C_1&A_1-zI_\ell&B_1&&&&\\&C_2&A_2-zI_\ell&B_2&&\\&&\ddots&\ddots&\ddots&\\&&&C_n&A_n-zI_\ell&B_n\\&&&&S_{+}^{bd}&C_{+}
\end{bmatrix}     \in\mathbb{C}^{(n+2)\ell\times(n+2)\ell},   
    \end{equation}
\begin{theorem}\label{theorem1.112}
    Take the specialization $\Pi=[I_\ell,0]$ so that $[S_{+}^{bd},C_{+}]=[0,I_\ell]$. Let $\zeta$ satisfy the same conditions as in Theorem \ref{maincircularlaw1.1}. Then we can find $C>0$ depending only on $|z|$ and the law of $\zeta$ such that, whenever $\ell$ is sufficiently large, 
$$
\mathbb{P}(s_{min}(\mathcal{T}_{n+2}-zI_{mid})\leq \ell^{-10n}(\ell^2n)^{-10}t)\leq Ct+\ell^{-5},\quad t>0.
$$\end{theorem}

The lower bound in Theorem \ref{theorem1.112} is exponentially small in the dimension, and we will only use this estimate when $\ell$ is much larger than $n$, say when $\ell\geq n^{12}$.

The proof of Theorem \ref{theorem1.112} is similar to \cite{jain2021circular}, Section 2, and uses the invertibility via distance approach pioneered in \cite{rudelson2008littlewood}. However, the first block row and last block row of $\mathcal{T}_{n+2}$ are purely deterministic and $[V,U]$ is arbitrary: this adds significant technical difficulty and is new in the literature. 

To handle these generic boundary conditions, we first prove the following estimate: 
\begin{lemma}\label{lemma3.2} Fix $L>0$.
    Let $\zeta$ be a random variable which is either real with distributional density on $\mathbb{R}$ bounded by $L$, or has independent real and imaginary parts $\Re\zeta,\Im\zeta$ such that at least one of them has distributional density bounded by $L$. Let $F$ be a $\ell\times 2\ell$ matrix with i.i.d. entries of law $\zeta$, and $S\in\mathbb{C}^{2\ell\times\ell}$ be a fixed (complex-valued) matrix with unitary columns. Let $Z\in\mathbb{C}^{\ell\times\ell}$ be an arbitrary fixed matrix. Then we can find $C_L>0$ depending only on $L$ such that for any $\epsilon>0$,
    \begin{equation}\label{lemma3.2estimates}
        \mathbb{P}(s_{min}(FS-Z)\leq \epsilon\ell^{-1.5})\leq C_L\epsilon.
    \end{equation}
\end{lemma}

The proof of Lemma \ref{lemma3.2} is given at the end of the section.

For the upper boundary $\begin{bmatrix}
    V&U
\end{bmatrix}$, we denote by $P_{V,U}^\perp$ its unitary complement in $\mathbb{C}^{2\ell}$:
\begin{equation}
P_{V,U}^\perp=\{v\in\mathbb{C}^{2\ell}:\begin{bmatrix}
    V\quad U
\end{bmatrix}\cdot v=0\}.
\end{equation} We denote by $\begin{bmatrix}
    C_1,A_1-zI_\ell
\end{bmatrix}\mid_{P_{V,U}^\perp}$ the restriction of this $\ell\times 2\ell$ matrix to the subspace $P_{V,U}^\perp$.

We first introduce a family of high probability events on the matrix blocks. For $K>0$,
\begin{equation}\begin{aligned}\label{EK}
\mathcal{E}_K:=\{&\forall i\in[n]:\|B_i\|,\|C_i\|,\|A_i\|\leq K;s_{min}(B_i)\geq \ell^{-9},s_{min}(C_i)\geq\ell^{-9},\\& s_{min}(\begin{bmatrix}
    C_1,A_1-zI_\ell
\end{bmatrix}\mid_{P_{V,U}^\perp})\geq\ell^{-9} \}.\end{aligned}
\end{equation}
\begin{fact}\label{fact3,3}
    For each $z\in\mathbb{C}$ we can find $K>0$ such that 
$\mathbb{P}(\mathcal{T}_{n+2}\in\mathcal{E}_K)\geq 1-K\ell^{-5}$.
\end{fact}

\begin{proof}
After normalization by $(3\ell)^{-1/2}$, Lemmas \ref{lemma2.6} and \ref{lemma3.2} give failure probability $C\ell^{-7}$ for each least singular value condition. Lemma \ref{astronger} gives the same bound for each norm condition when $K$ is sufficiently large. A union bound yields $\mathbb P(\mathcal E_K^c)\leq C(n+1)\ell^{-7}$, which implies the claim for $n\leq\ell$ and sufficiently large $\ell$.  
\end{proof}

\subsection{Structure of normal vectors}
Let $T_1,\cdots,T_{(n+2)\ell}$ be the rows of $\mathcal{T}_{n+2}-zI_{mid}$ from top to bottom. For each $1\leq i\leq (n+2)\ell$, let $H_{i}$ be the linear span of all rows of $\mathcal{T}_{n+2}-zI_{mid}$ in $\mathbb{C}^{(n+2)\ell}$ except the $i$-th row. We prove the following structural theorem:

\begin{Proposition}\label{proposition3.4} For any unit vector $v\in\mathbb{C}^{(n+2)\ell}$ we denote by $v_{[1]},\cdots,v_{[n+2]}$ its restriction to the columns of the blocks labeled $1,2,\cdots,n+2$. For each value $1\leq k\leq (n+2)\ell$ we denote by $N_k=\lfloor \frac{k-1}{\ell}\rfloor+1$, so that the $k$-th row lies in the $N_k$-th block from top to bottom. 
Then on the event $\mathcal{E}_K$, 
 for each $\ell+1\leq k\leq (n+1)\ell$, let $v$ be a unit vector orthogonal to $H_k$.
Then for sufficiently large $\ell$, $$\max(
\|v_{[N_k-1]}\|_2,\|v_{[N_k]}\|_2,\|v_{[N_{k}+1]}\|_2)\geq\ell^{-10n}(\ell n)^{-1/2}.$$

\end{Proposition}

\begin{proof} The block equations below use the conjugate normal, denoted again by $v$; its block norms are unchanged. In this proof we write $A_k^z:=A_k-zI_\ell$ and we prove a slightly stronger result. Since $v$ is a unit vector, there must exist a  $j_0\in[n+2]$ such that $\|v_{[j_0]}\|\geq (2n)^{-1/2}$. Without loss of generality assume that $j_0<N_k-1$ (the opposite case $j_0>N_{k}+1$ is analogous, and the special cases $j_0=N_k,N_k-1,N_k+1$ already verify the claim).

Step 1: To the left of $j_0$, when $j_0\geq3$. The normal vector $v$ must satisfy the following equation 
$$
C_{j_0-2}v_{[j_0-2]}+A^z_{j_0-2}v_{[j_0-1]}+B_{j_0-2}v_{[j_0]}=0.
$$Then on the event $\mathcal{E}_K$ we get that
$$
\|B_{j_0-2}v_{[j_0]}\|\geq \ell^{-9}\|v_{[j_0]}\|\geq \ell^{-9}(2n)^{-1/2},
$$and that 
\begin{equation}\label{lines804}
\|C_{j_0-2}v_{[j_0-2]}+A_{j_0-2}^zv_{[j_0-1]}\|\leq (K+|z|)(\|v_{[j_0-2]}\|+\|v_{[j_0-1]}\|).
\end{equation}
(Note that the invertibility of $B_{j_0-2}$ is valid since $j_0\leq N_k$ and we only remove a row of $B_{N_k-1}$). Combining both sides, we see whenever $\ell$ is large enough relative to $K,z$, then 
\begin{equation}\label{equation7}
    \text{either } \|v_{[j_0-2]}\|\geq \ell^{-10}(2n)^{-1/2}\text{ or }\|v_{[j_0-1]}\|\geq \ell^{-10}(2n)^{-1/2}.
\end{equation}(Otherwise the right hand side of \eqref{lines804} is at most $2(K+|z|)\ell^{-10}(2n)^{-1/2}\ll \ell^{-9}(2n)^{-1/2}$ when $\ell$ is sufficiently large, a contradiction.)
Then we take $j_{-1}$ to be the smaller one of $j_0-1$ or $j_0-2$ satisfying \eqref{equation7}. If $j_{-1}\geq 3$ we iterate the argument with $j_{-1}$ and we can find $j_{-2}\in\{j_{-1}-1,j_{-1}-2\}$ with 
    $$
\|v_{[j_{-2}]}\|\geq \ell^{-20}(2n)^{-1/2}.
    $$
    Then we continue in this manner and find a sequence of indices $j_0,j_{-1},\cdots,j_{-x},x\leq n$ so that $j_{-x}\in\{1,2\}$ and that for all $i\in[x]$, we have
    $$
|j_{-i}-j_{-(i-1)}|\leq 2,\text{ and }\|v_{[j_{-i}]}\|\geq \ell^{-10i}(2n)^{-1/2}.
    $$(If $j_0=1$ or $j_0=2$ then this procedure is unnecessary).

    Step 2: To the right of $j_0$. Since $v$ solves the following equation 
    \begin{equation}
        C_{j_0}v_{[j_0]}+A^z_{j_0}v_{[j_0+1]}+B_{j_0}v_{[j_0+2]}=0, 
    \end{equation}so long as $j_0<N_k-1$ so that $C_{j_0}$ is invertible, we conduct exactly the same computation as in Step 1 to deduce that at least one of $\|v_{[j_0+1]}\|$ or $\|v_{[j_0+2]}\|$ has norm at least $\ell^{-10}(2n)^{-1/2}$. Iterating, we can find a sequence of indices $j_1,j_2,\cdots,j_y$, $y\leq n$, so that $j_y\in \{N_k-1,N_k\}$, and for all $i\in[y]$,
$$
  |j_i-j_{i-1}|\leq 2,\quad\text{and}\|v_{[j_i]}\|\geq \ell^{-10i}(2n)^{-1/2}.
$$ This completes the proof. \end{proof}

In Proposition \ref{proposition3.4}, we have excluded values of $k$ in the first and last block row. A special treatment is needed for them. We now introduce an auxiliary matrix:

$$
\mathcal{T}_z^{res}:=\begin{bmatrix}
    C_1&A_1-zI_\ell&B_1&&&\\&C_2&A_2-zI_\ell&B_2&\\&&\ddots&\ddots&\ddots\\&&&C_n&A_n-zI_\ell
\end{bmatrix}\in\mathbb{C}^{n\ell\times(n+1)\ell}.
$$
This matrix is obtained from $\mathcal{T}_{n+2}$ by removing the first row and the last row and column. From now on we take the specialization $[S_{+}^{bd},C_{+}]=[0,I_\ell]$.

We will evaluate $\mathcal{T}_z^{res}$ on the following subset of $\mathbb{C}^{(n+1)\ell}$:
$$
\operatorname{Vec}_{V,U}^\perp:=\{v\in\mathbb{C}^{(n+1)\ell}:(v_{[1]},v_{[2]})\in P_{V,U}^\perp\}.
$$ To allow for a more convenient linear algebra treatment, we take a projection on the first two components of $\mathbb{C}^{(n+1)\ell}$ and reduce $\mathcal{T}_z^{res}$ to the following square matrix:

\begin{fact}\label{fact3.575}There exists a matrix $C_2^*$  and a matrix $A_1^*:=[C_1,A_1-zI_\ell]\mid_{P_{V,U}^\perp }$ such that the following defined matrix 
$$
\mathcal{T}_z^{res,sq}:=\begin{bmatrix}
A_1^* & B_1 & & \\
C_2^* & A_2 - zI_\ell & B_2 & \\
     & \ddots & \ddots & \ddots \\
     &        & C_n    & A_n - zI_\ell \\
\end{bmatrix}\in\mathbb{C}^{n\ell\times n\ell} 
$$ satisfies
$$
s_{min}(\mathcal{T}_z^{res,sq})=s_{min}(\mathcal{T}_z^{res}\mid_{\operatorname{Vec}_{V,U}^\perp}),
$$
where $C_2^*$ is purely a function of $C_2$ defined by $C_2^*=[0,C_2]\mid_{P_{V,U}^\perp}$.
\end{fact}This fact follows from applying a projection onto the first two components of $\mathbb{C}^{(n+1)\ell}$. We then prove the following geometric property for $\mathcal{T}_z^{res,sq}$:
\begin{Proposition}\label{proposition3.6}
    Let $T_1^{res},\cdots,T_{n\ell}^{res}$ be the rows of $\mathcal{T}_z^{res,sq}$ from top to bottom. Also let $H_i^{res}$ be the linear span of all rows of $\mathcal{T}_{z}^{res,sq}$ in $\mathbb{C}^{n\ell}$ except the $i$-th row, for all $1\leq i\leq n\ell$. Then on the event $\mathcal{E}_K$, for each $1\leq k\leq n\ell$, let $v$ be a unit vector orthogonal to $H_k^{res}$, and that row $k$ lies in the $N_k$-th block from top to bottom. Then for sufficiently large $\ell$,
$$\max(
\|v_{[N_k-1]}\|_2,\|v_{[N_k]}\|_2,\|v_{[N_{k}+1]}\|_2)\geq\ell^{-10n}(\ell n)^{-1/2}.$$ (When we encounter $v_{[0]}$ or $v_{[n+1]}$, we simply remove them from the maximum bracket).

\end{Proposition}

\begin{proof} The block equations below use the conjugate normal, denoted again by $v$; its block norms are unchanged.

    If $2\leq N_k\leq n$, we can always find a $j_0\in[n]$ such that $\|v_{[j_0]}\|\geq (2n)^{-1/2}$. Stop if $j_0\in\{N_k-1,N_k,N_k+1\}$; if $j_0=1$, we apply the relation $$A_1^*v_{[1]}+B_1v_{[2]}=0$$ and deduce that on $\mathcal{E}_K$, $\|v_{[2]}\|\geq \ell^{-10}(2n)^{-1/2}$. If $N_k\leq3$, stop here; otherwise apply the relation $$C_3v_{[2]}+A_3^zv_{[3]}+B_3v_{[4]}=0$$ to deduce that $\max(\|v_{[3]}\|,\|v_{[4]}\|)\geq \ell^{-20}(2n)^{-1/2}$. Applying this iteratively completes the proof. Suppose that $2\leq j_0\leq N_k-2$, then we start from $C_{j_0+1}v_{[j_0]}+A^z_{j_0+1}v_{[j_0+1]}+B_{j_0+1}v_{[j_0+2]}=0$ and do the iteration. The other case $N_k+2\leq j_0\leq n$ is exactly analogous.

 If $N_k=1$, and that $\|v_{[1]}\|\leq (2n)^{-1/2}$, then we can find some $2\leq j_0\leq n$ such that $\|v_{[j_0]}\|\geq (2n)^{-1/2}$. If $j_0=2$, the claim is immediate; otherwise apply the relation
$$
\widetilde C_{j_0-1}v_{[j_0-2]}+A_{j_0-1}^zv_{[j_0-1]}+B_{j_0-1}v_{[j_0]}
=0$$ where $\widetilde C_2=C_2^*$ and $\widetilde C_j=C_j$ for $j\geq3$. This gives $\max(\|v_{[j_0-1]}\|,\|v_{[j_0-2]}\|)\geq \ell^{-10}(2n)^{-1/2}$. Then we iterate all the way to the left and deduce that $\max(\|v_{[1]}\|,\|v_{[2]}\|)\geq \ell^{-10n}(2n)^{-1/2}$, 
completing the proof. For $N_k=2$, the same iteration gives
\[\max(\|v_{[2]}\|,\|v_{[3]}\|)\geq\ell^{-10n}(\ell n)^{-1/2}.\]
Use the first-row equation if $j_0=1$, and stop at block $2$ or $3$ otherwise.
\end{proof}

Before proceeding to the next step, let us conclude that we have gained a good understanding of the geometry of the normal vector to the subspaces $H_k^{res}$ of $\mathcal{T}_z^{res,sq}$ for all $k$, and the geometry of the normal vector to the subspaces $H_k$ of $\mathcal{T}_{n+2}-zI_{mid}$ except those in the first and last block. To complete the picture for $\mathcal{T}_{n+2}-zI_{mid}$, we will first bound the least singular value of $\mathcal{T}_z^{res,sq}$ in the following section.

\subsection{Invertibility via distance}
We use the following simple geometric lemma to reduce singular value to distance estimates. The version we present is actually weaker than the version in \cite{rudelson2008littlewood} and \cite{jain2021circular}, but it relieves us of the task of studying compressible/incompressible vectors separately, and the loss of quantitative bound is negligible for our purpose.

\begin{lemma}\label{lemma3.5}
With the notations above, we have for any $t>0$,
$$\begin{aligned}&
\mathbb{P}(\mathcal{E}_K\cap\{s_{min}(\mathcal{T}_{n+2}-zI_{mid})\leq t(2n\ell)^{-1/2}\})\leq\sum_{k=1}^{\ell (n+2)}\mathbb{P}(\mathcal{E}_K\cap\{\operatorname{dist}(T_k,H_k)\leq t\}),\end{aligned}
$$
$$\begin{aligned}&
\mathbb{P}(\mathcal{E}_K\cap\{s_{min}(\mathcal{T}_{z}^{res,sq})\leq t(n\ell)^{-1/2}\})\leq\sum_{k=1}^{\ell n}\mathbb{P}(\mathcal{E}_K\cap\{\operatorname{dist}(T_k^{res},H_k^{res})\leq t\}).\end{aligned}
$$
\end{lemma}

\begin{proof}Both estimates follow from the following general principle: let $E\in\mathbb{C}^{n\times n}$ and $v\in\mathbb{C}^{n}$ with $\|v\|=1$ such that $\|Ev\|=s_{min}(E)$. There exists a coordinate $i\in[n]$ such that $\|v_i\|\geq n^{-1/2}$, and that we use $$\|Ev\|\geq\operatorname{dist}(E_iv_i,\mathcal{H}_i)
\geq n^{-1/2}\operatorname{dist}(E_i,\mathcal{H}_i)$$
    where $E_i$ is the $i$-th column and $\mathcal{H}_i$ the span of all other columns except the $i$-th. Then we take a union bound for the small ball events of $\operatorname{dist}(E_i,\mathcal{H}_i)$ over all $i$, and we use $s_{min}(E)=s_{min}(E^*)$ to switch the role of rows and columns.
\end{proof}

Before proceeding to bounding from below $s_{min}(\mathcal{T}_z^{res,sq})$, we would like to insert here the proof of Lemma \ref{lemma3.2} because we will reuse a key computation in the proof.

\begin{proof}[\proofname\ of Lemma \ref{lemma3.2}]
    Let $F_i$ denote the $i$-th row of $F$ and $Z_i$ the $i$-th row of $Z$, then $F_iS-Z_i$ is the $i$-th row of $FS-Z$. Let $H_i$ be the subspace of $\mathbb{C}^\ell$ spanned by other rows of $FS-Z$ except the $i$-th row. Applying the negative second moment identity (see \cite{WOS:000281425000010},Lemma A.4) to $FS-Z$,
\begin{equation}\label{negativesecondmoments}
\|(FS-Z)^{-1}\|_{HS}^2= \sum_{i=1}^\ell \operatorname{dist}(F_iS-Z_i,H_i)^{-2}.\end{equation} 
     Let $n_i\in\mathbb{C}^\ell$ be a unit normal to $H_i$, which is independent of $F_i$. Then \begin{equation}\label{stars1}\operatorname{dist}(F_iS-Z_i,H_i)\geq|\langle F_iS-Z_i,n_i\rangle|=|\langle F_i,\bar S n_i\rangle-q|\end{equation} where $q=\langle Z_i,n_i\rangle$. We have
     $\|\bar S n_i\|_2=1$ since $S$ has unitary columns.
 When $\zeta$ is real-valued, we can choose $\theta\in[0,2\pi]$ such that $\|\Re(e^{i\theta}\bar S n_i)\|\geq\frac{1}{2}$.
Then
\begin{equation}\label{stars2}
|\langle F_i,\bar S n_i\rangle-q|=|e^{-i\theta}\langle F_i,\bar S n_i\rangle-e^{-i\theta}q|\geq |\langle F_i,\Re (e^{i\theta}\bar S n_i)\rangle-\Re(e^{-i\theta}q)|.
\end{equation} By the main result of  \cite{rudelson2015small}, since $F_i$ has i.i.d. coordinates of bounded density $L$, its projection to any 1-dimensional subspace of $\mathbb{R}^{2\ell}$ has density bounded by $C'L$ for some $C'>0$. Thus \begin{equation}\label{stars3}\sup_{s\in\mathbb{R}}\mathbb{P}(|\langle F_i,\Re (e^{i\theta}\bar S n_i)\rangle-s|\leq t)\leq C''t\end{equation} for a $C''$ depending only on $L$. We take $t=\epsilon\ell^{-1}$ and sum up the estimate over all $i=1,\cdots,\ell$ by a union bound. Then we plug the estimate into \eqref{negativesecondmoments} and finally complete the proof.

In the complex case, we condition on one component $\Re \zeta$ or $\Im\zeta$ and apply the argument to  the other component with a bounded density to complete the proof.
\end{proof}

Now we state and prove the least singular value lower bound for $\mathcal{T}_z^{res,sq}$:
\begin{Proposition}\label{proposition3.85}
With the assumptions above, we have the following estimate for some $C>0$ depending on $z$ and $\zeta$:
\begin{equation}
    \mathbb{P}(s_{min}(\mathcal{T}_z^{res,sq})\leq t\ell^{-10n}\ell^{-2.5}n^{-2})\leq C(t+\ell^{-5}),\quad t>0.
\end{equation}

\end{Proposition}

\begin{proof}
    Choose orthonormal coordinates $S=\begin{pmatrix}S_1\\S_2\end{pmatrix}$ on $P_{V,U}^\perp$.
Let $u_k$ be the conjugate of a unit normal to $H_k^{res}$, measurable in the other rows.
The coefficients in the original entries are
\[
a_k=\begin{cases}
(Su_{k,[1]},u_{k,[2]}),&N_k=1,\\
(S_2u_{k,[1]},u_{k,[2]},u_{k,[3]}),&N_k=2,\\
(u_{k,[N_k-1]},u_{k,[N_k]},u_{k,[N_k+1]}),&N_k\geq3,
\end{cases}
\]
with nonexistent blocks omitted. On $\mathcal E_K$, $\|a_k\|\geq\ell^{-10n}(\ell n)^{-1/2}$ by Proposition \ref{proposition3.6} and its second-row refinement.
Condition on the other rows, not on $\mathcal E_K$, and apply Lemma \ref{lemma3.2}'s bounded-density argument; all $z$ terms are deterministic shifts.
The coefficient-size event is measurable in the other rows, hence
\begin{equation}\label{thepreviouscase}
\mathbb P(\mathcal E_K\cap\{\operatorname{dist}(T_k^{res},H_k^{res})\leq t\})
\leq Ct\ell^{10n}(\ell^2n)^{1/2}.
\end{equation}

    Combining all these estimates with Lemma \ref{lemma3.5}, where we take $
r=t\ell^{-10n}\ell^{-2}n^{-3/2},
$
and then using the union bound over \(n\ell\) rows gives
\[
\mathbb P\left(\mathcal E_K\cap\{s_{\min}(\mathcal{T}_z^{\rm res,sq})
\leq t\ell^{-10n}\ell^{-2.5}n^{-2}\}\right)
\leq Ct.
\]
Adding $\mathbb P(\mathcal E_K^c)$ proves the claim.
\end{proof}

\subsection{Back to the original matrix}
Now we return to $\mathcal{T}_{n+2}$. First, we complete the study of the geometric structure of its normal vectors:

\begin{Proposition}\label{proposition3333} We take the specialization $[S_+^{bd},C_+]=[0,I_\ell]$.
On the event $\mathcal{E}_K\cap\{s_{min}(\mathcal{T}_z^{res,sq})\geq \ell^{-10n}(\ell^2n)^{-4}\}$,\begin{enumerate}
    \item 
 For any $1\leq k\leq \ell$, let $v$ be a unit vector orthogonal to $H_k$, then $|\langle v,[V,U][k]\rangle |\geq \ell^{-10n}(\ell^2 n)^{-5}$, where $[V,U][k]$ is the $k$-th row of the matrix $\mathcal{T}_{n+2}$.

\item Meanwhile, for any $(n+1)\ell+1\leq k\leq (n+2)\ell,$ let $v$ be a unit vector orthogonal to $H_k$, then the $k$-th entry of $v$ has absolute value at least $\ell^{-10n}(\ell^2n)^{-5}$.\end{enumerate}
\end{Proposition}
\begin{proof}Put $u=\bar v$, so $T_z^{\setminus\{k\}}u=0$, where $T_z=\mathcal T_{n+2}-zI_{mid}$ and $K_z=C(K+|z|+1)$, so $\|T_z\|\leq K_z$ on $\mathcal E_K$. The proofs of the two claims are analogous. For claim (1), we must have $u_{[n+2]}=0$ for $u$ to be the normal vector, using this lower boundary specialization $[0,I_\ell]$. We can take an orthogonal decomposition $u=u_1+u_2,$ with $u_1\in P_{V,U}$ and $u_2\in \operatorname{Vec}_{V,U}^\perp$ since the last block entry vanishes. Using again $u$ is the normal vector of $H_k$, we see that $u_1$ must be colinear with the adjoint of the $k$-th row of $[V,U]$ for this relation to hold. By orthogonality, $\|u_1\|^2+\|u_2\|^2=1$ and $T_z$ acts on $u_2$ just as via $\mathcal{T}_{z}^{res}$. By Fact \ref{fact3.575}, the latter matrix has the same least singular value as $\mathcal{T}_z^{res,sq}$, which is at least $\ell^{-10n}(\ell^2n)^{-4}$. Meanwhile, $u_1$ is supported only on the first two coordinate blocks, so by the operator norm bound we have $$\|T_z^{\setminus \{k\}}\cdot u_1\|\leq K_z\|u_1\|,$$ where we denote by $T_z^{\setminus \{k\}}$ as $T_z$ with the $k$-th row removed. This combined with $$T_z^{\setminus\{k\}}\cdot(u_1+u_2)=0,\quad \|T_z^{\setminus\{k\}}\cdot u_2\|= \|T_z\cdot u_2\|\geq \ell^{-10n}(\ell^2n)^{-4}\|u_2\|$$ yield the desired lower bound for $\|u_1\|$. Finally, $|\langle v,[V,U][k]\rangle|=|[V,U][k]u|=\|u_1\|$.

For (2), decompose $u=u_ke_k+(u-u_ke_k)$. The second term has zero last block and first two blocks in $P_{V,U}^{\perp}$; the same estimate applies.
\end{proof}
Finally we complete the proof of the main result on lower bounding $\sigma_{min}(\mathcal{T}_{n+2}-zI_{mid})$:

\begin{proof}[\proofname\ of Theorem \ref{theorem1.112}]By the joint-event estimate in Proposition \ref{proposition3.85}, with $t=\ell^{-5.5}n^{-2}$ there, and $\mathbb P(\mathcal E_K^c)\leq C(n+1)\ell^{-7}$, the event
$\Omega=\mathcal E_K\cap\{s_{min}(\mathcal T_z^{res,sq})\geq\ell^{-10n}(\ell^2n)^{-4}\}$
satisfies $\mathbb P(\Omega^c)\leq C\ell^{-5.5}n^{-2}+C(n+1)\ell^{-7}\leq\ell^{-5}$.
Work on $\Omega$.
For interior rows with $\ell+1\leq k\leq (n+1)\ell$, Proposition \ref{proposition3.4} offers the geometric description for the normal vector. One then takes the inner product with the $k$-th row which has i.i.d. entries with bounded density on blocks $[N_k-1,N_k,N_{k}+1]$. This implies the same joint-event small ball probability estimate as in \eqref{thepreviouscase}. For boundaries where $k\leq \ell$ or $k\geq (n+1)\ell+1$, Proposition \ref{proposition3333} shows the normal vector $v$ already has an overlap with an inner product of absolute value at least $\ell^{-10n}(\ell^2n)^{-5}$ with the $k$-th row of $\mathcal{T}_{n+2}-zI_{mid}$. For $0<t\leq1$, put $r=\sqrt{2n\ell}\,\ell^{-10n}(\ell^2n)^{-10}t$.
As $r<\ell^{-10n}(\ell^2n)^{-5}$, boundary terms vanish on $\Omega$ in Lemma \ref{lemma3.5}; interior terms total $O(t)$. Add $\mathbb P(\Omega^c)$. For $t>1$ use $C\geq1$.
\end{proof}

\subsection{An upgraded tail estimate}
Currently, the least singular value bound in Theorem \ref{theorem1.112} has an additive error $\ell^{-5}$. When we later deduce uniform integrability of log determinant in Section \ref{uniformintegrability}, we will need a version of Theorem \ref{theorem1.112} with a much stronger tail estimate. We present here a strengthening of Theorem \ref{theorem1.112} with better tail control:

\begin{theorem}\label{updatedtails} In the same setting as Theorem \ref{theorem1.112}, there exist $C>0$ and $\ell_0$, depending only on $|z|$ and the law of $\zeta$, such that for every $\ell\geq\ell_0$ and every $i\in\mathbb N_+$,
    \begin{equation}\label{momentseverals}
\mathbb{P}(s_{min}(\mathcal{T}_{n+2}-zI_{mid})\leq \ell^{-10ni}(\ell^2n)^{-10}t)\leq Ct+\ell^{-5i},\quad t>0.
\end{equation}
Consequently, for each $\alpha\geq1$ we can find $C_\alpha>0$ depending only on $\alpha,|z|,\zeta$ such that 
$$
\mathbb{E}[|\log|s_{min}(\mathcal{T}_{n+2}-zI_{mid})||^\alpha]\leq C_\alpha (n(\log\ell+\log n))^\alpha.
$$
\end{theorem}

\begin{proof}
Write $T=\mathcal T_{n+2}-zI_{\rm mid}$ and
$M=\mathcal T_z^{\rm res,sq}$. Let $\Omega_i$ impose
$\|A_j\|,\|B_j\|,\|C_j\|\le K\ell^{i/4}$ and retain the lower bounds
$\ell^{-9i}$ on $s_{\min}(B_j),s_{\min}(C_j)$ and
$s_{\min}([C_1,A_1-zI_\ell]|_{P_{V,U}^{\perp}})$.
Using Lemma \ref{astronger} with the fixed order $24$, and the
normalized estimates in Lemmas \ref{lemma2.6}, \ref{lemma3.2}, gives
\[
 \mathbb P(\Omega_i^c)
 \le C\ell^{-23-6i}+C\ell^{3-9i}
 \le\tfrac14\ell^{-5i}
\]
for sufficiently large $\ell$, uniformly for all $i\ge1$.
Indeed, after division by $\ell^{-5i}$ the two terms are bounded by
$C\ell^{-24}$ and $C\ell^{-1}$, so the lower threshold for $\ell$
does not depend on $i$.

Choose each row normal as a measurable unit right-kernel vector of the
row-deleted matrix, with absent blocks set to zero. The propagation
arguments of Propositions \ref{proposition3.4}, \ref{proposition3.6}
now cost at most
$\ell^{-9i}/(2K_*\ell^{i/4})\ge\ell^{-19i/2}$ per step. For $T$'s
interior-row normals, the complete lower boundary forces the last
block to vanish, so propagation into the deleted row's neighbourhood
takes at most $n-1$ steps. The $n$ blocks of $M$ give the same bound.
Thus the
tested random row has an effective coefficient vector of norm at least
\[
 g_i=\ell^{-19(n-1)i/2}(n\ell)^{-1/2}.
\]
For $M$'s second row block, use blocks $2,3$: the complete first-row
relation $A_1^*v_{[1]}+B_1v_{[2]}=0$ transfers any mass in block $1$
to block $2$. No invertibility of $C_2^*$ is required. Conditioning
only on the other rows, the bounded-density estimate and row-distance
union bound give
\[
 \mathbb P(\Omega_i\cap\{s_{\min}(M)\le s\})
 \le Cs\ell^{19(n-1)i/2}\ell^{5/2}n^2.
\]
Here $\Omega_i$ is intersected with the event being estimated; the
tested row is never conditioned on a norm or inverse event.
Set $a_i=\ell^{-(19(n-1)/2+6)i}(\ell^2n)^{-4}$ and
$\Gamma_i=\Omega_i\cap\{s_{\min}(M)>a_i\}$. Consequently,
\[
 \mathbb P(\Gamma_i^c)
 \le\tfrac14\ell^{-5i}+C\ell^{-6i-11/2}n^{-2}
 \le\tfrac12\ell^{-5i}.
\]
The orthogonal decomposition in Proposition \ref{proposition3333},
using $\|T\|\le K_T\ell^{i/4}$, bounds every boundary-row distance on
$\Gamma_i$ below by
$a_i/(2K_T\ell^{i/4})\ge
h_i:=\ell^{-(19n/2-3)i}(\ell^2n)^{-4}$.
Indeed, if $v=v_1+v_2$ is that decomposition, then
$a_i\|v_2\|\le K_T\ell^{i/4}\|v_1\|$ and
$\|v_1\|^2+\|v_2\|^2=1$; the missing boundary coordinate has
absolute value $\|v_1\|$.
For $b_i=t\ell^{-10ni}(\ell^2n)^{-10}$ and $0<t\le1$,
$\sqrt{(n+2)\ell}\,b_i<h_i$. Only interior rows therefore contribute,
and their conditional small-ball estimates yield
\[
 \mathbb P(\Gamma_i\cap\{s_{\min}(T)\le b_i\})
 \le Cn\ell\sqrt\ell\sqrt{(n+2)\ell}\,b_i/g_i\le Ct.
\]
Adding the exceptional probability proves \eqref{momentseverals};
$t>1$ is immediate after choosing $C\ge1$.

Finally, put $X=(-\log s_{\min}(T))_+$ and
$L_*=32n(\log\ell+\log n)$. Taking $t=\ell^{-i}$ gives
$\mathbb P(X\ge L_*i)\le(C+1)\ell^{-i}$, uniformly in $i$.
The same uniformity first implies $s_{\min}(T)>0$ almost surely.
Integration over $[L_*i,L_*(i+1)]$ gives
\[
 \mathbb EX^\alpha\le L_*^\alpha
 \left[1+(C+1)\sum_{i\ge1}
 ((i+1)^\alpha-i^\alpha)\ell^{-i}\right]\le C_\alpha L_*^\alpha.
\]
The positive logarithm is bounded by $\log_+\|T\|$; the block norm
estimate gives $\mathbb E\|T\|\le C_z(1+n)$, and
$(\log x)_+^\alpha\le C_\alpha x$ completes the moment bound.
\end{proof}

\section{Convergence of the log determinant}\label{section444}
With the least singular value estimate in Theorem \ref{theorem1.112}, the proof of Theorem \ref{theorem1.115}, i.e. the convergence of log determinant, is reduced to the control of small-ish singular values. We will do this via a rigidity estimate as in \cite{jain2021circular}, \cite{han2025circular}. However, two new technical difficulties arise here: (1) the tridiagonal block matrix does not have a doubly stochastic variance profile, and (2) we need to treat the arbitrary upper boundary $[V,U]$.
\subsection{Convergence of Stieltjes transform} We shall derive a convergence rate for the Stieltjes transform to the circular law limit. Since $\mathcal{T}_{n+2}$ is not translationally invariant, we consider a translationally invariant substitute, the following periodic block matrix model :

\begin{equation}\label{eqperiodic}
\mathcal{T}_{n+2}^{per} =\begin{bmatrix}
A_0 &B_0&&&C_0\\C_1&A_1&B_1&&\\&\ddots&\ddots&\ddots&\\&&C_{n}&A_{n}&B_{n}\\B_{n+1}&&&C_{n+1}&A_{n+1}
\end{bmatrix}  \end{equation}
where $\mathcal{T}_{n+2}^{per}$ shares the interior random blocks of $\mathcal{T}_{n+2}$, and the six additional blocks $A_0,B_0,C_0$ and $C_{n+1},A_{n+1},B_{n+1}$ are mutually independent, independent of those interior blocks, and have the same normalized i.i.d. entry distribution. Then we clearly see that $\operatorname{rank}(\mathcal{T}_{n+2}^{per}-zI_{(n+2)\ell}-\mathcal{T}_{n+2}+zI_{mid})\leq 8\ell$, and for any $z\in\mathbb{C}$ we have
\begin{equation}\label{rankupdateformula}
\operatorname{rank}\left((\mathcal{T}_{n+2}^{per}-zI_{(n+2)\ell})(\mathcal{T}_{n+2}^{per}-zI_{(n+2)\ell})^*-(\mathcal{T}_{n+2}-zI_{mid})(\mathcal{T}_{n+2}-zI_{mid})^*\right)\leq 24\ell.
\end{equation}

The following convergence result is a variant of Proposition \ref{proposition4.1} and a similar result is also proven in \cite{jain2021circular}, Theorem 3.1:
\begin{Proposition}\label{stieltjesconvergence} Assume that $\ell\geq n$, and take any fixed $z\in\mathbb{C}$. Assume that the entry law $\zeta$ satisfies the moment assumptions in Theorem \ref{maincircularlaw1.1}. We denote by \begin{equation}\label{stieltjes}m_{n+2,z}^{per}(\xi):=\frac{1}{(n+2)\ell}\sum_{i=1}^{(n+2)\ell}[\lambda_i\left((\mathcal{T}_{n+2}^{per}-zI_{(n+2)\ell})(\mathcal{T}_{n+2}^{per}-zI_{(n+2)\ell})^*\right)-\xi]^{-1}\end{equation} the Stieltjes transform for the empirical measure of $(\mathcal{T}_{n+2}^{per}-zI_{(n+2)\ell})(\mathcal{T}_{n+2}^{per}-zI_{(n+2)\ell})^*$. Then we can find a deterministic probability measure $\nu_z$ on $[0,\infty)$ such that $m_z(\xi):=\int_\mathbb{R}\frac{d\nu_z(x)}{x-\xi}$ is the unique solution to the following equation 
$$
m_z(\xi)=\left[\frac{|z|^2}{1+m_z(\xi)}-(1+m_z(\xi))\xi\right]^{-1}
$$in the class of Stieltjes transforms of probability measures on $[0,\infty)$. Equivalently, the branch satisfies $\Im m_z(\xi)>0$, $\Im(\sqrt\xi\,m_z(\xi))>0$, with the principal square root, and $m_z(\xi)\sim-1/\xi$ at infinity (see \cite{jain2021circular}, Theorem 3.1). For any fixed $D>0$ and $c>0$, with probability at least $1-(\ell n)^{-100}$ the following estimate holds uniformly for $|\Re\xi|\leq D$ and $0<\Im\xi\leq1$:
$$
|m_{n+2,z}^{per}(\xi)-m_z(\xi)|\leq \frac{C}{|\Im\xi|^8}\frac{(\ell n)^c}{\ell^{1/2}},$$ where $C$ depends only on $\zeta,c,D$ and $z$.
\end{Proposition}

The proof of Proposition \ref{stieltjesconvergence} is deferred to the next section.
Let $m_{n+2,z}(\xi)$ be the Stieltjes transform of empirical measure of $(\mathcal{T}_{n+2}-zI_{mid})(\mathcal{T}_{n+2}-zI_{mid})^*$, defined similarly as in \eqref{stieltjes} (but subtract $I_{mid}$ instead of the identity matrix). We can derive the convergence of $m_{n+2,z}(\xi)$ via the following low rank update formula:
\begin{lemma}\label{lemma3.1111}
With the assumptions above, we have for any $\xi\in\mathbb{C}:\Im\xi>0$ and $z\in\mathbb{C}$,
$$
|m_{n+2,z}^{per}(\xi)-m_{n+2,z}(\xi)|\leq \frac{24\pi}{(n+2)|\Im\xi|}.
$$
\end{lemma}

\begin{proof}
    This follows from the following perturbation formula (see \cite{bai2010spectral}, Appendix A and B):
    $$
|\frac{1}{N}\operatorname{tr}(A-\xi)^{-1}-\frac{1}{N}\operatorname{tr}(B-\xi)^{-1}|\leq\frac{\pi}{|\Im\xi|}\frac{\operatorname{rank}(A-B)}{N}
    $$ for two Hermitian matrices $A,B$ of size $N$, applied to the rank difference in \eqref{rankupdateformula}.
\end{proof}

\subsection{Singular value rigidity estimate} When we combine the convergence in Proposition \ref{stieltjesconvergence} with the rank update in Lemma \ref{lemma3.1111}, a problem arises that the error caused by the rank perturbation, which has the order $1/n\Im\xi$, dominates the error in the convergence in Proposition \ref{stieltjesconvergence} when $\ell$ is much larger than $n$. This increased error is not an artifact: since the matrix $\mathcal{T}_{n+2}$ does not have a doubly stochastic variance profile, its Stieltjes transform differs from the circular law density by a correction having the scale $\ell/n\ell=1/n$. In the polynomial regime $\ell=\operatorname{Poly}(n)$, we still get a polynomial convergence rate of $m_{n+2,z}(\xi)$ to $m_z(\xi)$. However, since the least singular value estimate we get in Theorem \ref{theorem1.112} is exponentially small, a significantly weakened polynomial convergence rate will not permit us to conclude with the convergence of the log determinant, see the details in Section \ref{logdeterminant}.

We will eliminate this problem by proving a rigidity estimate for the small singular values of $\mathcal{T}_{n+2}-zI_{mid}$, which states that with high probability most of its singular values are at least $(n\ell)^{-120}$. Related coarse singular-value counting estimates for doubly stochastic profiles are proved in \cite{han2025circular}.
The reason why this rigidity holds for $\mathcal{T}_{n+2}$ is that the added upper and lower boundaries actually do not destroy the invertibility of the matrix near 0, and hence do not destroy singular value rigidity near 0. Also, the fact that we change from periodic to Dirichlet boundary condition (we do not have the upper right and lower left blocks) should also have negligible effect in small singular values. The precise rigidity result is the following:
\begin{Proposition}\label{propositionrigidity}
In the setting of Theorem \ref{theorem1.115}, for any $c>0$, with probability at least $1-C\ell^{-7}$, in the interval $[0,(n\ell)^{-120}]$, the random matrix $\mathcal{T}_{n+2}-zI_{mid}$ has at most $\frac{(n\ell)^{1+c}}{\ell^{0.1}}$ number of singular values. Here $C>0$ depends on $c>0$, the law of $\zeta$ and on $|z|$.
\end{Proposition}

The proof of Proposition \ref{propositionrigidity} is rather technical and deferred to Section \ref{959rigidity}. The value of 120 in the exponent is inessential and can be tuned so long as it is polynomially small. This is the only place throughout the paper where we use the fact that $\zeta$ has finite $p$-th moment for all $p$: we choose this strong moment condition to considerably simplify the proof.

Proposition \ref{propositionrigidity} is mainly used for very large $\ell$, say when $\ell=(n\ell)^{0.95}$. Although this estimate does not appear to be strong enough, as it excludes a very large number, say $\frac{(n\ell)^{1+c}}{\ell^{0.1}}$ of singular values, it guarantees that the other singular values are at least polynomially small in $\ell$, rather than exponentially small in $\ell$ as in Theorem \ref{theorem1.112}.

\subsection{Passing to the log determinant}\label{logdeterminant}
We note here an upper bound for $\|\mathcal{T}_{n+2}\|$:

\begin{lemma}\label{operatornormbound}We have the following deterministic operator norm upper bound of $\mathcal{T}_{n+2}$:
    $$
\|\mathcal{T}_{n+2}\|\leq 2+10(\max_{1\leq i\leq n}\|A_i\|+\max_{1\leq i\leq n}\|B_i\|+\max_{1\leq i\leq n}\|C_i\|).
    $$
\end{lemma}

This estimate can be verified by using the tridiagonal block structure of inner rows of $\mathcal{T}_{n+2}$, and that the top and bottom rows have operator norm bounded by 1.

Now we proceed to prove convergence of the log determinant to the Gaussian model.
For $z\in\mathbb{C}$, let $G_z:=G-zI_{(n+2)\ell}$ where $G$ is a complex Ginibre matrix of size $(n+2)\ell$ and having normalized variance. Denote by $X_z:=\mathcal{T}_{n+2}-zI_{mid}$ (we are not subtracting by identity matrix) and denote by $s_1(X_z)\geq s_2(X_z)\geq\cdots$ the singular values of $X_z$ in decreasing order. Define $s_i(G_z)$ analogously. We denote by 
$$
\nu_{X_z}(\cdot):=\frac{1}{(n+2)\ell}\sum_{i=1}^{(n+2)\ell}\delta_{s_i^2(X_z)}(\cdot),\quad \nu_{G_z}(\cdot):=\frac{1}{(n+2)\ell}\sum_{i=1}^{(n+2)\ell}\delta_{s_i^2(G_z)}(\cdot) 
$$ the empirical measure of squared singular values of $X_z,G_z$.  For two probability measures $\mu,\nu$ supported on $[0,\infty)$,
we define their Kolmogorov distance via 
$$
\|\mu-\nu\|_{[0,\infty)}:=\sup_{x\geq 0}|\mu([0,x])-\nu([0,x])|.
$$
We can obtain an upper bound for the Kolmogorov distance $\|\nu_{X_z}(\cdot)-\nu_{G_z}(\cdot)\|$ as follows:

\begin{corollary}\label{proofcorollary4.5}In the setting of Theorem \ref{theorem1.115}, where we assume that for some $d\in(0,\frac{1}{2})$ we have $ \ell^{1/2}\geq n\geq\ell^d$, we can find a constant $x_d\in(0,1)$ depending only on $d$ such that, with probability at least $1-(n\ell)^{-90}$, the following estimate holds for some constant $C$ depending only on $\zeta,z$ and $d$:
$$
\|\nu_{X_z}(\cdot)-\nu_z(\cdot)\|_{[0,\infty)}\leq C\ell^{-x_d}.
$$
    Exactly the same estimate also holds for $\|\nu_{G_z}(\cdot)-\nu_z(\cdot)\|_{[0,\infty)}$.
\end{corollary}
The proof of Corollary \ref{proofcorollary4.5} is deferred to the end of this section. We also use a convenient formula for the Kolmogorov distance:
\begin{fact}\label{fact4,690}
    For two probability measures $\mu,\nu$ on $[0,\infty)$, and $0<a<b$,
    $$
\left|\int_a^b\log x\,d\mu(x)-\int_a^b\log x\,d\nu(x)\right|\leq 2[|\log b|+|\log a|]\|\mu-\nu\|_{[0,\infty)}.
    $$
Moreover, for any $\beta>1$ we have
$$
\left|\int_a^b|\log x|^\beta\, d\mu(x)-\int_a^b|\log x|^\beta  \,d\nu(x)\right|\leq 2[|\log b|^\beta+|\log a|^\beta ]\|\mu-\nu\|_{[0,\infty)}.
    $$
    
\end{fact}

\begin{proof} The first claim can be found in \cite{jain2021circular}, Lemma 4.3. Write $F_\mu(t)=\mu([0,t])$ and $F_\nu(t)=\nu([0,t])$. For any $C^1$ function $\varphi$ on $[a,b]$, integration by parts gives
    $$
    \int_{[a,b]} \varphi\,d\mu=\varphi(b)F_\mu(b)-\varphi(a)F_\mu(a-)-\int_a^bF_\mu(t)\varphi'(t)\,dt.
    $$Subtracting the corresponding identity for $\nu$, we get
    \begin{equation}\label{lines1095}
\left|\int_a^b\varphi d\mu-\int_a^b \varphi d\nu\right|\leq \|\mu-\nu\|_{[0,\infty)}\left(|\varphi(a)|+|\varphi(b)|+\int_a^b|\varphi'(t)|dt\right).
    \end{equation}
    For $\varphi(x)=|\log x|^\beta$, $\beta>1$, the function is $C^1$ on $(0,\infty)$, with derivative zero at $x=1$, and its total variation satisfies
    $$\int_a^b|\varphi'(t)|dt=\beta(\int_{[a,1]}+\int_{[1,b]})\frac{|\log t|^{\beta-1}}{t}dt\leq |\log a|^\beta+|\log b|^\beta
    $$when $a\leq1\leq b$; if the interval lies on one side of $1$, the same bound follows from monotonicity. Combining this with \eqref{lines1095} gives the stated factor $2(|\log a|^\beta+|\log b|^\beta)$.
\end{proof}

Now we can prove Theorem \ref{theorem1.115}, the main result of this section:

\begin{proof}[\proofname\ of Theorem \ref{theorem1.115}:  convergence in probability]Recall that we assume, for some very small $d>0$, $\ell^{1/12}\geq n\geq\ell^d$.  With probability at least $1-C\ell^{-5}$, we can assume the following four estimates hold simultaneously: (The following claims for $X_z$
follow from this paper's proof; while the claims for $G_z$, which is a shifted complex Ginibre matrix, are standard and can be found from other papers.)
\begin{enumerate}
    \item The least singular values of both $X_z,G_z$ are at least $\ell^{-11n}$, by Theorem \ref{theorem1.112};
    \item  $\|G_z\|,\|X_z\|\leq K<\infty$ for some constant $K>0$, by Lemma \ref{operatornormbound} and Lemma \ref{astronger};
    \item The convergence rate in Corollary \ref{proofcorollary4.5} holds; 
    \item The two matrices $X_z,G_z$ have at most $\frac{(n\ell)^{1+c}}{\ell^{0.1}}$ number of singular values in the interval $[0,(n\ell)^{-120}]$ for any $c>0$, by Proposition \ref{propositionrigidity}. (The claim for $G_z$ is standard Ginibre estimate without using Proposition \ref{propositionrigidity}.)
\end{enumerate}
Then on this event we decompose the log integration as follows:
\begin{equation}\label{differencelogs}
    \int_0^\infty\log x\nu_{X_z}(dx)- \int_0^\infty\log x\nu_{G_z}(dx)=\int_{I_1\cup I_2}\log x(\nu_{X_z}(dx)-\nu_{G_z}(dx) )  ,
\end{equation}
where $I_1=[\ell^{-22n},(n\ell)^{-240}]$ and $I_2=[(n\ell)^{-240},K^2]$ (note that we are using squared singular values).
By Fact \ref{fact4,690}, the integration on $I_2$ is bounded in absolute value by
$$
 2[2|\log K|+|\log ((n\ell)^{-240})|]\|\nu_{X_z}(\cdot)-\nu_{G_z}(\cdot)\|_{[0,\infty)}=o(\ell^{-x_d/2})\to 0
.$$
The integration on $I_1$ is bounded in absolute value by
$$
\frac{1}{n\ell}O(n\log\ell\frac{(n\ell)^{1+c}}{\ell^{0.1}})=o(\ell^{-0.01})\to 0,
$$whenever we take $c>0$ sufficiently small, and we use the assumption $\ell\geq n^{12}$. Combining the two integrations on $I_1,I_2$ completes the proof of convergence in probability.
\end{proof}

Finally, we complete the proof of Corollary \ref{proofcorollary4.5}:

\begin{proof}[\proofname\ of Corollary \ref{proofcorollary4.5}]
Recall that we assume, for some $d\in(0,\frac{1}{2})$, that $ n\geq\ell^d$ and $\ell\geq n^{2}$. Then we can find some $x_d\in(0,(1-d)/16)$ such that $\frac{x_d+d/2}{x_d+1}\leq d$.
Choose a fixed $K$ large enough that $\|X_z\|,\|G_z\|\leq K$ with probability at least $1-(n\ell)^{-100}$, using Lemmas \ref{operatornormbound} and \ref{astronger}. Fix $D>\max\{K^2,(2+|z|)^2\}+1$, so that on this event $[-D,D]$ contains the supports of $\nu_{X_z},\nu_{G_z}$ and $\nu_z$. Consider the horizontal line segment $$
L:=\{\xi=\theta+i(\frac{n}{\ell})^{x_d}:\quad -D\leq\theta\leq D\}.
$$Then on $L$ we have $n\Im\xi= n(\frac{n}{\ell})^{x_d}\geq\ell^{d/2}$, so that Lemma \ref{lemma3.1111} yields, for all $\xi\in L$:
$$
|m_{n+2,z}(\xi)-m_{n+2,z}^{per}(\xi)|\leq \frac{1}{\ell^{d/2}}.
$$

By Proposition \ref{stieltjesconvergence}, we have for each $\xi\in L$, for any $c>0$ with probability $1-(\ell n)^{-100}$,
$$
|m_{n+2,z}^{per}(\xi)-m_z(\xi)|\leq \frac{C(\ell n)^c}{n^{8x_d}\ell^{1/2-8x_d}}\leq\frac{C(\ell n)^c}{\ell^{d/2}}.
$$We can take $c>0$ sufficiently small relative to $d$ so that the right hand side is at most $C\ell^{-d/3}$.
Since both $m_{n+2,z}^{per}$ and $m_z$ are Lipschitz in $\xi$ with Lipschitz constant bounded by $|\Im\xi|^{-2}$, we can extend the convergence to all of $L$ and deduce that with probability at least $1-(\ell n)^{-90}$, uniformly for all $\xi\in L$ and for a possibly different constant $C$,
\begin{equation}\label{finalstieltjesbound}
|m_{n+2,z}(\xi)-m_z(\xi)|\leq C\ell^{-d/3}.
\end{equation} Then by \cite{bai2010spectral}, Corollary B.15 and Lemma 11.9, just as in the proof of \cite{jain2021circular}, Lemma 4.4, we have the following estimate with $\eta=\Im\xi$:
$$
\|\nu_{X_z}(\cdot)-\nu_z(\cdot)\|_{[0,\infty)}\leq C\left[\int_{-D}^D|m_{n+2,z}(\xi)-m_z(\xi)|d\xi+\sqrt{\eta}\right]
$$for some $C>0$ depending on the fixed support interval and $z$. Taking the estimate \eqref{finalstieltjesbound} inside, we get that with probability at least $1-(n\ell)^{-90}$,
$$\|\nu_{X_z}(\cdot)-\nu_z(\cdot)\|_{[0,\infty)}\leq C(\ell^{-d/3}+(\frac{n}{\ell})^{x_d/2})\leq C(\ell^{-d/3}+\ell^{-x_d/4}),
$$where the last inequality follows from $\ell\geq n^2$. Then we replace $x_d$ by $\min(d/3,x_d/4)$ and complete the proof.\end{proof}

\subsection{Uniform integrability of the log determinant}\label{uniformintegrability}

In this section, we upgrade the convergence of log determinants in Proposition \ref{proposition1.999} and Theorem \ref{theorem1.115} to convergence in expectation, by establishing uniform integrability.

\begin{proof}[\proofname\ of Proposition \ref{proposition1.999}]

    The part of almost-sure convergence is essentially an intermediate technical step in the proof of circular law in \cite{WOS:000281425000010}, where they proved that, taking $G$ the same size matrix with i.i.d. complex Gaussian entries of unit variance, then the following limit holds both in probability and almost surely:
    $$
\frac{1}{n}\log|\det (n^{-1/2}A)|-\frac{1}{n}\log|\det (n^{-1/2}G)|\to 0,\quad n\to\infty.
    $$
    It is well-known that, by the Ginibre potential computations, 
    $$
\frac{1}{n}\log|\det(n^{-1/2}G)|\to\int_{|w|\leq 1} \log|w|\frac{dA(w)}{\pi}=-\frac{1}{2}, a.s.,\quad n\to\infty,
    $$where $dA(w)$ is the two-dimensional Lebesgue measure. 

    To prove convergence in expectation, we work separately on good and bad events. We let $s_1\geq s_2\geq\cdots\geq s_n$ denote the singular values of $n^{-1/2}A$ in decreasing order, then by Cauchy-Schwarz we have
    $$|\frac{1}{n}\log|\det n^{-1/2}A||^2\leq\frac{1}{n}\sum_{i=1}^n|\log s_i|^2.
    $$ Let $\nu_A$ denote the empirical measure of squared singular values of $(n^{-1/2}A)(n^{-1/2}A)^*$. Then by the same argument as in Corollary \ref{proofcorollary4.5}, we can find some $c_0>0$ such that with probability $1-n^{-10}$, $\|\nu_A(\cdot)-\nu_0(\cdot)\|_{[0,\infty)}\leq n^{-c_0}$. Let $\Omega_A$ denote the event where this convergence holds, and moreover that $\|n^{-1/2}A\|\leq K$ for some large $K>0$, and $s_{min}(n^{-1/2}A)\geq n^{-9}$. Then by the previous theorems, we have $\mathbb{P}(\Omega_A)\geq 1-n^{-5}$.

    On the event $\Omega_A$, we apply Fact \ref{fact4,690} with $\beta=2$ to verify that
    \begin{equation}\label{toverifyquadratic}
\left|\frac{1}{n}\sum_{i=1}^n|\log s_i|^2-\frac{1}{4}\int_{n^{-18}}^{K^2} |\log x|^2\nu_0(dx)\right|\leq C[(\log K)^2+(\log n)^2]n^{-c_0}=o(1),
\end{equation}and it also follows from a standard computation that $\int_0^\infty |\log x|^2\nu_0(dx)<\infty$.

    On the event $\Omega_A^c$, we take the trivial upper bound $$\frac{1}{n}|\log|\det n^{-1/2}A||\leq \max(|\log s_{max}(n^{-1/2}A)|,|\log s_{min}(n^{-1/2}A)|),$$
and we use Cauchy-Schwarz in the following inequality to get
\begin{equation}\label{upperboundtheauantity}
\mathbb{E}[1_{\Omega_A^c}|\frac{1}{n}\log|\det n^{-1/2}A||]\leq \mathbb{P}(\Omega_A^c)^{3/4}\mathbb{E}[\max(|\log s_{max}(n^{-1/2}A)|,|\log s_{min}(n^{-1/2}A)|)^4]^\frac{1}{4},\end{equation}
 which is upper bounded by $Cn^{-15/4}\log^2n\to 0$ thanks to Lemma \ref{astronger} and \ref{lemma2.6}. For the second moment, put $L_n=\max(|\log s_{max}(n^{-1/2}A)|,|\log s_{min}(n^{-1/2}A)|)$. The same bounds give $\mathbb E L_n^4\leq C(\log n)^4$, and hence
\[
\mathbb E\left[1_{\Omega_A^c}\left|\frac1n\log|\det n^{-1/2}A|\right|^2\right]
\leq\mathbb P(\Omega_A^c)^{1/2}(\mathbb E L_n^4)^{1/2}
\leq Cn^{-5/2}(\log n)^2.
\]
Together with \eqref{toverifyquadratic} on $\Omega_A$, this verifies the uniform boundedness of
$\mathbb{E}[|\frac{1}{n}\log|\det n^{-1/2}A||^2]
,$ which completes the proof of uniform integrability and thus the convergence in expectation.

Finally, we note that to obtain a quantitative convergence rate in expectation from a convergence rate in probability, we can use the following simple consequence of Hölder inequality. Suppose that random variables $X_n$ converge to $X$ in probability with a tail $\mathbb{P}(|X_n-X|\geq t)\leq r_n(t)$ for all $t\geq 0$ and that $\mathbb{E}|X_n|^2+\mathbb{E}|X|^2\leq M$ for some $M>0$, then for any $t>0$,
$
\mathbb{E}|X_n-X|\leq t+2\sqrt{M\cdot r_n(t)}.$
\end{proof}
The convergence in expectation of Theorem \ref{theorem1.115} can be proven similarly:

\begin{proof}[\proofname\ of Theorem \ref{theorem1.115} : convergence in expectation.] 
We let $N=(n+2)\ell$ and $s_1,\cdots,s_N$ the singular values of the matrix $\mathcal{T}_{n+2}-zI_{mid}$. Let $\Omega_T$ be the event where all the four criteria (1) to (4) in the (in probability convergence part of the) proof of Theorem \ref{theorem1.115} hold true. Then on $\Omega_T$, we can bound $\sum_{i=1}^N\frac{1}{N}\log s_i$ by separating the sum over singular values in $I_1$ and $I_2$ defined there, with the summation over $I_1$ converges to 0 in absolute value (since $\frac{1}{N}\sum_{i:s_i^2\in I_1}|\log s_i|\leq\frac{1}{n\ell}\cdot \frac{(n\ell)^{1+c}}{\ell^{0.1}}\cdot O(n\log\ell)=o(1)$) and the summation over $I_2$ can be estimated similarly to \eqref{toverifyquadratic}.

We have $\mathbb{P}(\Omega_T)\geq 1-C\ell^{-5}$, and on $\Omega_T^c$ we similarly apply Cauchy-Schwarz to upper bound the quantity as in \eqref{upperboundtheauantity}. By Theorem \ref{updatedtails}, we can verify that
$$
\mathbb{E}[|\log |s_{min}(\mathcal{T}_{n+2}-zI_{mid})||^4]\leq C(N\log N)^4
$$ for some $C>0$ depending only on $\zeta$ (and a better tail for $s_{max}$ follows from Lemma \ref{operatornormbound} and \ref{astronger}), so that $\mathbb{E}[1_{\Omega_T^c}\frac{1}{N}\sum_{i=1}^N|\log s_i|]=O(\ell^{-15/4}N\log N)$.
Since $\ell^{-15/4}N\log N\to 0$ thanks to $\ell\geq n^{12}$ and $N=n\ell$, the expectation over $\Omega_T^c$ again vanishes.
\end{proof}

\section{Singular value rigidity estimate for a general boundary case}\label{959rigidity}

In this section, we prove Proposition \ref{propositionrigidity}, the rigidity bound of small singular values of $\mathcal{T}_{n+2}-zI_{mid}$.

\subsection{Removal of the boundary}
The standard method to prove rigidity estimates is to derive a local law for $\mathcal{T}_{n+2}$ and bound the imaginary part of its Stieltjes transform. However, the general boundary on the first row of $\mathcal{T}_{n+2}$ makes the derivation of local law extremely complicated. Here we take a few steps to remove the boundary and reduce to rigidity estimates of a more canonical model.

\begin{lemma}\label{lemma5.1}
We can find a $2\ell\times\ell$ matrix $R$ with orthonormal columns, where we denote by $R=\begin{bmatrix}
    R_1\\R_2
\end{bmatrix}$ its two components of size $\ell$, such that the following statement holds on the event $\{\|\mathcal{T}_{n+2}-zI_{mid}\|\leq K\}$ for some $K>1$. For any $k\in\mathbb{N}_+$ and $\epsilon\in(0,\frac{1}{36K^2})$, suppose that $\mathcal{T}_{n+2}-zI_{mid}$ has at least $k$ singular values in the interval $[0,\frac{\epsilon}{3K\sqrt{2n\ell}}]$, then the following matrix $Y^R_z$ has at least $k$ singular values in the interval $[0,\epsilon]$,
$$
Y^R_z:=\begin{bmatrix}
    \begin{bmatrix}C_1,A_1-zI_\ell\end{bmatrix}R&B_1&&\\C_2 R_2&A_2-zI_\ell&B_2&\\&&\ddots\ddots\ddots\\&&C_n&A_n-zI_\ell
\end{bmatrix}.
$$
\end{lemma}

\begin{proof}Let \(U_0\in U(2\ell)\) be a unitary matrix such that
\[
[V,U]U_0=[I_\ell,0].
\]
Write
\[
U_0=[R^\perp,R],
\qquad R^\perp,R\in\mathbb C^{2\ell\times \ell}.
\]
Then \(R\) has orthonormal columns and \([V,U]R=0\). We further decompose
\[
R=\begin{pmatrix}R_1\\ R_2\end{pmatrix},
\qquad R_1,R_2\in\mathbb C^{\ell\times\ell}.
\]
Let \(\widetilde T_z\) be the matrix obtained from
\(\mathcal{T}_{n+2}-zI_{\rm mid}\) after applying this unitary change of variables
on the first two block columns which transforms the upper boundary
\([V,U]\) into \([I_\ell,0]\). This unitary change does not change the
singular values. Under this identification, we delete the first and the
last boundary coordinate blocks and then restrict to the interior rows: this
gives exactly the matrix \(Y_z^R\).

Set
\[
\delta:=\frac{\varepsilon}{3K\sqrt{2n\ell}}.
\]
Assume that \(\widetilde T_z\) has at least \(k\) singular values in
\([0,\delta]\). By the min-max principle, there exists a \(k\)-dimensional
subspace \(E\subset \mathbb C^{(n+2)\ell}\) such that
\[
\|\widetilde T_z v\|\le \delta\|v\|,\qquad v\in E.
\]
We let \(P_{\rm bd}\) be the orthogonal projection onto the two boundary
coordinate blocks, and let \(P_{\rm int}=I-P_{\rm bd}\). Since the two
boundary rows of \(\widetilde T_z\) are the coordinate projections onto
these two boundary blocks, we have
\[
\|P_{\rm bd}v\|\le \|\widetilde T_zv\|\le \delta\|v\|,
\qquad v\in E.
\]
Hence \(P_{\rm int}|_E\) is injective. Let
\[
W:=P_{\rm int}E.
\]
Then \(\dim W=k\), and for every \(u\in W\) there exists \(v\in E\) with
\(u=P_{\rm int}v\). Moreover,
\[
\|u\|\ge (1-\delta)\|v\|.
\]

Using \(\|\widetilde T_z\|\le K\), we get
\[
\|Y_z^R u\|
\le
\|\widetilde T_zP_{\rm int}v\|
\le
\|\widetilde T_zv\|+\|\widetilde T_zP_{\rm bd}v\|
\le
\delta\|v\|+K\delta\|v\|.
\]
Therefore
\[
\|Y_z^R u\|
\le
\frac{(K+1)\delta}{1-\delta}\|u\|
\le \varepsilon\|u\|
\]
for \(\ell\) large, since \(K>1\) and
\(\delta=\varepsilon/(3K\sqrt{2n\ell})\). Thus \(Y_z^R\) is bounded by
\(\varepsilon\) on the \(k\)-dimensional subspace \(W\). Therefore, by the min-max
principle, \(Y_z^R\) has at least \(k\) singular values in
\([0,\varepsilon]\).
\end{proof}

Next, we reduce the rigidity of $Y_z^R$ to the case where $B_1=C_2R_2=0$, by showing that rigidity of $Y_z^R$ can be deduced from rigidity of its lower $[2,n]\times[2,n]$ block-principal minor. For this purpose we first prove an auxiliary lemma:

\begin{lemma}\label{lemmas1231}
    Let $A\in M_N(\mathbb{C})$ be a deterministic matrix and let $P$ be the orthogonal projection onto an output block of dimension $\ell$, and we write $Q=I-P$. Let $B$ be an independent $\ell\times\ell$ random block with the same distribution as $B_i,C_i$, and consider 
    $$
\Delta x=\iota BLx,
    $$where $L:\mathbb{C}^N\to\mathbb{C}^\ell$ is a deterministic contraction and $\iota$ embeds $\mathbb{C}^\ell$ into the selected output block. Assume that for some $K>0$ we have $\|A\|\leq K$. Then we can find two constants $c,C>0$ depending on $K$ and the density of $\zeta$ such that, for any $M>0$ and $a\in(0,1)$, with probability at least $1-C\ell^{-M}$ on the event $\{\|B\|\leq K\}$, we have
    $$
N_{A+\Delta}(ca^2\ell^{-M-9})\leq N_A(a)
    ,$$ where $N_A(t)$ is the number of singular values of $A$ in $[0,t]$.
\end{lemma}

\begin{proof}
Let $E$ be the large singular subspace of $A$, such that 
$$
\|Av\|\geq a\|v\|,\quad v\in E.
$$Since $\Delta$ is only supported in the $P$-row block, we have that 
$$
Q(A+\Delta)v=QAv.
$$
Then we let $F\subset E$ be the subspace spanned by right singular vectors of $QA|_E$, with singular values at most $a/2$. Then $\dim F\leq \ell$ since $PA|_E$ has rank at most $\ell$ (Suppose $\dim F>\ell$, then since $\operatorname{Rank}(PA|_E)\leq\ell$, then $F\cap\ker(PA|_E)\neq 0$. Then take $0\neq f\in F\cap\ker(PA|_E)$, we have $PAf=0$, so $\|Af\|_2^2=\|QAf\|_2^2\leq \frac{a^2}{4}\|f\|^2$. But $f\in F\subset E$, so $\|Af\|\geq a\|f\|$, contradiction). Then on $E\cap F^\perp$, we have
$$
\|QAv\|\geq\frac{a}{2}\|v\|
$$and on $F$, we must have 
$$
\|PAf\|\ge\frac{1}{2}a\|f\|,\quad f\in F.
$$So we have reduced the finite-dimensional problem to the following:
$$
H+BL_F:F\to\mathbb{C}^\ell,\quad H:=PA|_F,L_F=L|_F,
$$with $$s_{\text{min}}(H)\geq\frac{1}{2}a.$$
Since we can rewrite this as
$$
H+BL_F=\begin{bmatrix}
    B,I
\end{bmatrix}\begin{pmatrix}
    L_F\\H
\end{pmatrix}
$$ and that 
$$
s_{\text{min}}(\begin{pmatrix}
    L_F\\H
\end{pmatrix})\geq s_{\text{min}}(H)\geq\frac{1}{2}a,
$$we only need to prove that, uniformly over all deterministic isometry $S:F\to\mathbb{C}^{2\ell}$, $\dim F\leq \ell$, with probability at least $1-\ell^{-M}$ we have
\begin{equation}\label{lines1272}
s_{\text{min}}(\begin{bmatrix}
    B,I
\end{bmatrix}S)\geq\ell^{-M-9}.
\end{equation}
Let $r=\dim F$ and we extend $S:F\simeq\mathbb{C}^r\to\mathbb{C}^{2\ell}$ to an isometry $\widetilde{S}:\mathbb{C}^\ell\to\mathbb{C}^{2\ell}$. Then $[B,I]S$ is the restriction of $[B,I]\widetilde{S}$ to a subspace, so that $s_{\text{min}}([B,I]S)\geq s_{\text{min}}([B,I]\widetilde{S})$. Thus we prove the estimate assuming $\dim F=\ell$.

We denote by $\mathcal{M}:=BS_1+S_2$, where $S=\begin{bmatrix}
    S_1\\S_2
\end{bmatrix}$. Let $\mathcal{M}_i$ be the $i$-th row and $H_i$ the span of all the other rows except $\mathcal{M}_i$. Also let $n_i$ be the unit normal to $H_i$. By the negative second moment argument, 
$$
s_{\text{min}}(\mathcal{M})^{-2}\leq \sum_{i=1}^\ell \operatorname{dist}(\mathcal{M}_i,H_i)^{-2}.
$$Condition on all rows except the $i$-th row, we have 
$$
\operatorname{dist}(\mathcal{M}_i,H_i)\geq |\langle b_i,S_1n_i\rangle+(S_2n_i)_i|.
$$Case (1): $\|S_1n_i\|\geq\delta$ for some $\delta>0$ to be fixed later. Then the bounded density assumption gives, for some $C>0$ depending only on the density bound $L$ of $\zeta$:
$$
\mathbb{P}(|\langle b_i,S_1n_i\rangle+(S_2n_i)_i|\leq u\mid\{b_j:j\neq i\})\leq C\frac{\sqrt{\ell}u}{\delta}.
$$
Case (2): $\|S_1n_i\|<\delta$. Then the orthogonality relation $\mathcal{M}_jn_i=0$ for $j\neq i$ implies that, on the event $\|B\|\leq K$,
$$
|(S_2n_i)_j|=|\langle b_j,S_1n_i\rangle|\leq K\delta,\quad j\neq i.
$$
Since $S$ is an isometry, 
$$
\|S_1n_i\|^2+\|S_2n_i\|^2=1.
$$ Then if we choose $\delta=(10K\sqrt{\ell})^{-1}$, the mass of $S_2n_i$ is not concentrated on the $j\neq i$ coordinates and we compute
$$
|(S_2n_i)_i|\geq\frac{1}{2}.
$$
Also, on the operator norm event, 
$$
|\langle b_i,S_1n_i\rangle|\leq K\delta\leq\frac{1}{3}.
$$ Then in case (2), we have deterministically on the operator norm event that 
$$
\operatorname{dist}(\mathcal{M}_i,H_i)\geq \frac{1}{6}.
$$
Then only case (1) contributes the small ball probability. We take $u=\ell^{-M-8}$ and union bound over all $i\leq\ell$ to conclude that with probability $1-C\ell^{-M}$, 
$$
\operatorname{dist}(\mathcal{M}_i,H_i)\geq u\quad\text{for all }i.
$$Therefore, on the same high probability event, 
$$
s_{\text{min}}(\mathcal{M})\geq u\ell^{-1/2}\geq\ell^{-M-9}.
$$This completes the proof of \eqref{lines1272}. Then we have, with probability at least $1-\ell^{-M}$, $$s_{\text{min}}(H+BL_F)\geq\frac{1}{2}a\ell^{-M-9}.$$

Finally, we lift the vector back from $F$ to $E$. If we write $v=f+g\in E,f\in F,g\in F^\perp\cap E$ with $\|v\|=1$, such that $\|(A+\Delta)v\|\leq b$, then the $Q$-part implies $\|g\|\leq 2b/a$ (This is because $F$ is a spectral space of $(QA|_E)^*(QA|_E)$, so that the subspace $QAF$ and $QA(F^\perp\cap E)$ are orthogonal.) Then on the operator norm event, we have $\|A+\Delta\|\leq 2K$, and hence
\begin{equation}\label{lines1380}
\|(A+\Delta)f\|\leq \|(A+\Delta)g\|+\|(A+\Delta)v\|\leq b+4Kb/a.
\end{equation}
But we already know that on $F$, with probability at least $1-C\ell^{-M}$,
$$
\|(A+\Delta)f\|\geq \frac{1}{2}a\ell^{-M-9}\|f\|,
$$then we choose $b\leq c a^2\ell^{-M-9}/K$ for a small $c>0$ to guarantee that $\|g\|\leq\frac{1}{2}$ and then $\|f\|\geq \frac{1}{2}$, so that the last expression contradicts \eqref{lines1380}. Therefore, $A+\Delta$ is bounded from below by $b$ on $E$, so that by min-max principle, $N_{A+\Delta}(b)\leq N_A(a)$.
\end{proof}

\begin{lemma}\label{lemma5.2}
Let $Y_z^{R,diag}$ denote the block diagonal matrix obtained from $Y_z^R$ by setting $B_1=0$ and $C_2R_2=0$. Then with probability at least $1-C\ell^{-8}$, the following is true when $\ell\geq n$: Suppose that $Y_z^{R,diag}$ has at most $k$ singular values in the interval $[0,\ell^{-10}]$, then $Y_z^R$ has at most $k$ singular values in the interval $[0,\ell^{-110}]$ for $\ell$ sufficiently large.     
\end{lemma}In other words, this lemma shows that additional randomness in $B_1$ and $C_2$ will not destroy the singular value rigidity of $Y_z^{R,diag}$ too much. 
\begin{proof}
By Lemma \ref{astronger} and the block structure, for $n\leq\ell$ we may choose $K>1$, depending only on $z$ and the law of $\zeta$, such that $\|Y_z^{R,\mathrm{diag}}\|\leq K$, $\|B_1\|\leq K$, and $\|C_2\|\leq K$, except on an event of probability $O(\ell^{-9})$. On this norm event, after switching on $B_1$, the intermediate matrix has norm at most $2K$. Conditioning first on all blocks other than $B_1,C_2$, we apply Lemma \ref{lemmas1231} twice with $M=9$ and norm parameters $K$ and $2K$, respectively. First we start from $Y_z^{R,\text{diag}}$ and switch on $B_1$ with $L_Bx=x[2]$. Then conditioning on $B_1$, we switch on $C_2R_2$ with the map $L_Cx=R_2x[1]$. The fact that $R$ has unitary columns implies $\|R_2\|\leq 1$ and thus $L_C$ is a contraction. After the first application, the number of singular values below
c\(\ell^{-40}\) is still at most \(k\). After conditioning on \(B_1\)
and applying the lemma a second time with \(L_Cx=R_2x[1]\), the number
of singular values below \(c'\ell^{-100}\) is still at most \(k\). For
\(\ell\) sufficiently large we bound $\ell^{-110}\leq c'\ell^{-100}$ for any $c'>0$. The two failure probabilities and the norm exceptional event are absorbed
into \(C\ell^{-8}\). No lower bound on $R_2$ is needed in this proof.
\end{proof}

Now that $Y_z^{R,diag}$ has a block diagonal form and for its top left corner $[C_1,A_1-zI_\ell]R$ we already have good control of its least singular value by Lemma \ref{lemma3.2}, so that we only need to obtain rigidity estimates for its $[2,n]$ principal minor. The latter matrix $[Y_z^{R,diag}]_{[2,n]\times[2,n]}$ has exactly the same form as our block tridiagonal matrix $T$ \eqref{equationforalargeT}.

\subsection{The canonical block tridiagonal case: convergence to MDE}\label{section1463}
We proceed using a local law approach to derive rigidity estimates for $[Y_z^{R,diag}]_{[2,n]\times[2,n]}$, which is a block tridiagonal matrix. There have been well-developed techniques to show that the Green function of $[Y_z^{R,diag}]_{[2,n]\times[2,n]}$ converges to the solution to a set of matrix Dyson equations (MDE), but the main challenge in our case is that the system is not translationally invariant due to the boundaries, so the solution to the MDE cannot be obtained in a closed form.

We let $\mathcal{T}_n$ be the block tridiagonal matrix \eqref{equationforalargeT} of size $n\ell$ and block size $\ell$. We use the following set of parameters: $W=\ell, \text{ and }  N=n\ell=nW.$

We take the notation \begin{equation}Y_z=\mathcal{T}_{n}-zI_N,\end{equation} and we sometimes abbreviate the subscript $z$ by simply writing $Y=Y_z$.
We define the entry-wise variance of $Y_z$ as the following matrix $S\in\mathbb{R}^{N\times N}$:
$$
S=(s_{ij}),\quad  s_{ij}=\operatorname{Var}(\mathcal{T}_n)_{ij}.
$$

The main object of our study is the Green function of $Y_z^*Y_z$, and its trace:
\begin{equation}\label{greenfunctiondef1}
G(w):=G(w,z)=(Y_z^*Y_z-w)^{-1},\quad m(w):=m(w,z)=\frac{1}{N}\operatorname{Tr}G(w,z),\quad w=E+i\eta.
\end{equation}

We also consider the Green function $Y_zY_z^*$:
\begin{equation}\label{greenfunctiondef2}
\mathcal{G}(w):=\mathcal{G}(w,z)=(Y_zY_z^*-w)^{-1}.
\end{equation}

A standard route to study $G,\mathcal{G}$ is to take a reduction to the Hermitization $$\mathcal{H}_z:=\begin{bmatrix}
    0&Y_z\\Y_z^*&0
\end{bmatrix}$$and denote by $$\mathcal{HG}(w):=\mathcal{HG}(w,z)=(\mathcal{H}_z-w)^{-1}.$$

These resolvents satisfy the following linear algebra identities: for any $\lambda\in\mathbb{C}:\Im\lambda>0$,
\begin{equation}\label{equation4.444}
    \mathcal{G}(\lambda^2)={\frac{1}{\lambda}[\mathcal{HG}(\lambda)}]_{[1,N]\times[1,N]},\quad {G}(\lambda^2)={\frac{1}{\lambda}[\mathcal{HG}(\lambda)}]_{[N+1,2N]\times[N+1,2N]}.
\end{equation}

The strategy of this subsection is as follows. We first show, by matrix concentration inequality in Proposition \ref{proposition4.1}, that $\mathcal{HG}(w)$ will converge to the unique solution $M^\mathcal{H}(w)$ with positive imaginary part to the following equation 
\begin{equation}\label{MHWMHW}
  -M^\mathcal{H}(w)^{-1}=wI_{2N}+\begin{bmatrix}0&zI_N\\\bar{z} I_N&0
      \end{bmatrix}+\mathcal{S}^\mathcal{H}[M^\mathcal{H}(w)],
\end{equation}where $\mathcal{S}^\mathcal{H}:\mathbb{C}^{2N\times 2N}\to \mathbb{C}^{2N\times 2N}$ is the following defined self-energy operator 
\begin{equation}
\mathcal{S}^\mathcal{H}[X]=\mathbb{E}\left[\begin{bmatrix}
    0&\mathcal{T}_n\\\mathcal{T}_n^*&0
\end{bmatrix}X\begin{bmatrix}
    0&\mathcal{T}_n\\\mathcal{T}_n^*&0
\end{bmatrix}^*\right].
\end{equation}After fixing the variance profile, the additional off-diagonal part in the self-energy map $\mathcal{S}^\mathcal{H}$ is determined by the pseudovariance $\mathbb{E}[\zeta^2]$.
In the special case $\mathbb{E}[\zeta^2]=0$, we can check that $\mathcal{S}_0^H:=\mathcal{S}^\mathcal{H}$ is block diagonal and acts via two diagonal maps:
$$
\mathcal{S}_0^\mathcal{H}\begin{bmatrix}
    X_{11}&X_{12}\\X_{21}&X_{22}
\end{bmatrix}=\begin{bmatrix}
    \Phi[X_{22}]&0\\0&\widetilde{\Phi}[X_{11}]
\end{bmatrix},
$$where each block has size $N\times N$. The two maps $\Phi,\widetilde{\Phi}$ act diagonally with the expression 
$$
(\Phi[Z])_{ii}=\sum_c s_{ic}Z_{cc},\quad (\widetilde{\Phi}[W])_{cc} =\sum_is_{ic}W_{ii},
$$where the sums are over $[N]$. Indeed,
we have
\[
(\mathcal S^\mathcal H[X])_{12}
=
\mathbb E\,\mathcal T_n X_{21}\mathcal T_n,\qquad
(\mathcal S^\mathcal H[X])_{21}
=
\mathbb E\,\mathcal T_n^* X_{12}\mathcal T_n^*.
\]
These two blocks vanish when \(\mathbb E\zeta^2=0\), because the entries are centered,
independent, and no conjugate pairing appears. The diagonal blocks are
\[
\mathbb E\,\mathcal T_n X_{22}\mathcal T_n^*=\Phi[X_{22}],\qquad
\mathbb E\,\mathcal T_n^* X_{11}\mathcal T_n=\widetilde\Phi[X_{11}].
\]

For general $\zeta$ with $\mathbb{E}[\zeta^2]\neq 0$, we let $\rho=\mathbb{E}[\zeta^2]$ and denote by $\mathcal{S}_\rho^\mathcal{H}$ the self-energy operator associated to the pseudovariance $\rho$. We will use a perturbation argument for general $\rho$ in Lemma \ref{lem:pseudovariance-perturbation}.

Then from \eqref{equation4.444}, we deduce that $(\mathcal{G}(w),G(w))$ should converge with high probability to  $(M_1(w),M_2(w))$ such that 
\begin{equation}\label{m1w2}
M_1(w^2)=\frac{1}{w}[M^\mathcal{H}(w)]_{[1,N]\times[1,N]},\quad M_2(w^2)=\frac{1}{w}[M^\mathcal{H}(w)]_{[N+1,2N]\times [N+1,2N]}.
\end{equation}

The precise matrix concentration result for convergence of $\mathcal{HG}(w)$ towards $M^\mathcal{H}(w)$ is stated as follows: 
\begin{Proposition}\label{proposition4.1}
    Let $\zeta$ satisfy the moment assumptions in Theorem \ref{theorem1.112} (but we do not assume $\mathbb{E}[\zeta^2]=0$), and assume $W\geq \sqrt{N}$. Then for any sufficiently small $c>0$, we can find a constant $C>0$ depending only on $\zeta$ and $c$ such that whenever $N$ is large enough, the following holds with probability at least $1-N^{-100}$: uniformly for $w$ in the spectral domain \(N^{-C_0}\le \Im w\le 1\) for any fixed $C_0>0$,
    $$
\sup_{1\leq i\leq N}|[\mathcal{HG}(w)]_{ii}-[M^\mathcal{H}(w)]_{ii}|\leq \frac{C N^{c}}{\sqrt{W}(\Im w)^5},
    $$
     $$
\sup_{1\leq i\leq N}|[\mathcal{G}(w^2)]_{ii}-[M_1(w^2)]_{ii}|\leq \frac{C N^{c}}{\sqrt{W}(\Im w)^6}.
    $$
\end{Proposition}

\begin{proof}[\proofname\ of Proposition \ref{proposition4.1}]
For related matrix concentration arguments, see \cite{han2024outliers,han2024circular}.
Write $\eta=\Im w$, $\delta=c/8$, and
$D_z=\begin{bmatrix}0&-zI_N\\-\bar zI_N&0\end{bmatrix}$.
We first truncate the entries while preserving their full covariance.
Let $x=(\Re\zeta,\Im\zeta)^T$, $Q=\mathbb E xx^T$, and
$u=Q^{-1/2}x$, with inverses taken on $\operatorname{Ran}Q$.
For $L=N^\delta$, set
\[
 u_L=u\mathbf 1_{\{\|u\|\le L\}}-\mathbb E[u\mathbf 1_{\{\|u\|\le L\}}],
 \qquad V_L=\mathbb E u_Lu_L^T,
 \qquad \widehat x=Q^{1/2}V_L^{-1/2}u_L.
\]
All moments are finite, so $V_L=I+O(N^{-p})$ on $\operatorname{Ran}Q$
for every fixed $p>0$. The complex variable $\widehat\zeta$ corresponding
to $\widehat x$ is centered, has the same variance and pseudovariance as
$\zeta$, and satisfies $|\widehat\zeta|\le C_\zeta N^\delta$.
Let $\widehat H$ be the resulting Hermitization and
$\widehat R(w)=(\widehat H-w)^{-1}$.
The moment bounds, a union bound over at most $N^2$ entries, and the
normalization above imply, for any fixed $A,D>0$ and sufficiently large $N$,
\[
 \mathbb P\bigl(\|\mathcal H_z-\widehat H\|>N^{-A}\bigr)\le N^{-D}.
\]
Indeed, outside the entry truncation event, the centering and covariance
normalization change each entry by $O(N^{-p})$ for arbitrarily large $p$.

Write $\widehat H=D_z+\sum_\alpha Z_\alpha$, where the independent centered
Hermitian summands correspond to the allowed entries of $\mathcal T_n$.
Each is supported on one pair $(a,N+b),(N+b,a)$ and has norm at most
$r=C_\zeta N^\delta/\sqrt W$. In the notation of
\cite{brailovskaya2024universality},
\[
 \sigma^2:=\Bigl\|\mathbb E(\widehat H-D_z)^2\Bigr\|\le1,
 \qquad
 v^2:=\bigl\|\operatorname{Cov}(\operatorname{vec}(\widehat H))\bigr\|
 \le\frac{2}{3W},
 \qquad R(\widehat H)\le r.
\]
The first bound follows from the row and column variance sums, which are
at most one also at the boundary. For the second bound, the covariance
splits into blocks with eigenvalues $(1\pm|\rho|)/(3W)$, where
$\rho=\mathbb E\zeta^2$.
Let $H^{\mathrm g}$ have Gaussian entries with the same mean and full
covariance. By \cite{brailovskaya2024universality}, Theorem 2.10,
\[
 \bigl\|\mathbb E\widehat R(w)-\mathbb E(H^{\mathrm g}-w)^{-1}\bigr\|
 \le\frac{C(r\sigma^2+r^3\log(2N))}{\eta^4}
 \le\frac{Cr}{\eta^4},
\]
where the last inequality uses $W\ge\sqrt N$ and sufficiently small $c$.

Represent $H^{\mathrm g}=D_z+\sum_\beta A_\beta g_\beta$, with independent
standard real Gaussians $g_\beta$ and Hermitian $A_\beta$, and let
$H^{\mathrm{free}}=D_z\otimes1+\sum_\beta A_\beta\otimes s_\beta$, where
$s_\beta$ are freely independent standard semicircular elements.
The noncrossing Wick expansion, initially for large $|w|$ and then by
analytic continuation, shows that
$M(w)=(\mathrm{id}\otimes\tau)[(H^{\mathrm{free}}-w)^{-1}]$ satisfies
\[
 -M(w)^{-1}=wI_{2N}-D_z+\sum_\beta A_\beta M(w)A_\beta
 =wI_{2N}-D_z+\mathcal S^\mathcal H_\rho[M(w)].
\]
For $\Im w>0$, the resolvent representation gives $\Im M(w)>0$; uniqueness in \cite{helton2007operator}, Theorem 2.1,
identifies $M(w)$ with $M^\mathcal H(w)$ in \eqref{MHWMHW}.
By \cite{bandeira2023matrix}, Theorem 2.8,
\[
 \bigl\|\mathbb E(H^{\mathrm g}-w)^{-1}-M^\mathcal H(w)\bigr\|
 \le\frac{C\sigma^2v^2}{\eta^5}\le\frac{C}{W\eta^5}.
\]

We next give a concentration estimate for each diagonal entry
$f_i=e_i^*\widehat R(w)e_i$. Replace one entry by an arbitrary independent
copy, and denote the resulting matrix and resolvent by $\widehat H^{(\alpha)}$
and $\widehat R^{(\alpha)}$. Set
$\Delta_\alpha=\widehat H^{(\alpha)}-\widehat H$.
The pair support and the bound of $3W$ allowed entries per row imply,
uniformly over all original and replacement values,
\[
 \|\Delta_\alpha\|\le2r,\qquad
 \sup_{\|u\|=\|v\|=1}\sum_\alpha|u^*\Delta_\alpha v|^2\le Cr^2,
 \qquad \Bigl\|\sum_\alpha\Delta_\alpha^2\Bigr\|\le CWr^2.
\]
Twice the resolvent identity gives
\[
 \widehat R-\widehat R^{(\alpha)}
 =\widehat R\Delta_\alpha\widehat R
 -\widehat R\Delta_\alpha\widehat R^{(\alpha)}\Delta_\alpha\widehat R.
\]
Using $\|\widehat R\|,\|\widehat R^{(\alpha)}\|\le\eta^{-1}$, the squares
of the first terms sum to at most $Cr^2\eta^{-4}$, while
\[
 \sum_\alpha|e_i^*\widehat R\Delta_\alpha\widehat R^{(\alpha)}
                 \Delta_\alpha\widehat Re_i|^2
 \le\eta^{-2}\sup_\alpha\|\Delta_\alpha\widehat R^*e_i\|^2
                 \sum_\alpha\|\Delta_\alpha\widehat Re_i\|^2
 \le CWr^4\eta^{-6}.
\]
Consequently $\sum_\alpha|f_i-f_i^{(\alpha)}|^2\le
V:=C(r^2\eta^{-4}+Wr^4\eta^{-6})$ deterministically.
For $F=\Re f_i$ or $F=\Im f_i$, entropy tensorization and conditional
symmetrization yield, for $\lambda\ge0$,
\[
 \operatorname{Ent}(e^{\lambda F})
 \le\lambda^2\mathbb E\Bigl[e^{\lambda F}
       \sum_\alpha\mathbb E_\alpha'(F-F^{(\alpha)})_+^2\Bigr]
 \le\lambda^2V\mathbb E e^{\lambda F},
\]
where $\mathbb E_\alpha'$ averages over the replacement entry. The same
bound holds for $-F$. Integrating this inequality gives
$\log\mathbb E e^{\lambda(F-\mathbb EF)}\le V\lambda^2$ for real $\lambda$,
and hence
\[
 \mathbb P\bigl(|f_i-\mathbb Ef_i|\ge\sqrt{8Vt}\bigr)\le4e^{-t}.
\]
Taking $t=N^{c/2}$ gives $\sqrt{8Vt}\le
CN^{c/2}/(\sqrt W\eta^3)$ when $\eta\le1$.
Together with the two expectation comparisons this proves, for fixed $w$,
the first estimate with failure probability at most $4e^{-N^{c/2}}$
for the truncated model.

To make the estimate uniform, take a mesh of spacing $N^{-(2C_0+2)}$
in $|\Re w|\le N+2|z|$, $N^{-C_0}\le\Im w\le1$, and union bound
over its polynomially many points and all $2N$ diagonal entries.
Both resolvents, including the conditional free resolvent, are
$N^{2C_0}$-Lipschitz on this domain, so interpolation costs $O(N^{-2})$.
Outside this rectangle, $\|\widehat H\|\le|z|+C N^\delta\sqrt W=o(N)$
and $\|H^{\mathrm{free}}\|\le|z|+2\sigma$, so the comparison is $O(N^{-1})$.
These errors are absorbed by the claimed bound. Finally choose
$A>2C_0+3$ and $D=105$ in the truncation estimate. The resolvent identity
costs at most $N^{-A}\eta^{-2}$, and the total failure probability is
at most $N^{-100}$ for sufficiently large $N$.
The second estimate follows from \eqref{equation4.444} and \eqref{m1w2},
using $|w|^{-1}\le\eta^{-1}$.
\end{proof}

The precise polynomial rate of convergence in Proposition \ref{proposition4.1} is not important, so long as the rate is polynomial in $W$ and $\Im w$.

We now derive the self-consistency equations solved by the MDE entries $M_1$ and $M_2$. It turns out that the MDE solution can be easily analyzed when $\rho=0$, so we restrict ourselves to this case and handle general $\rho$ via perturbation. All MDEs in this subsection are finite-volume deterministic equations. The dimension is
\(N=n\ell\), and \(M^\mathcal{H}(w)=M_N^\mathcal{H}(w)\in M_{2N}(\mathbb C)\) is the unique solution to the MDE \eqref{MHWMHW} with
positive imaginary part.
The boundary of the block tridiagonal model is encoded in the finite-dimensional
variance maps. Thus we do not use any infinite-volume MDE in
this subsection.

\begin{lemma}\label{proofoflemma4.2}Assume $\rho=\mathbb{E}[\zeta^2]=0$.
Since $M^\mathcal{H}(w)$ solves \eqref{MHWMHW}, then $M_1(w),M_2(w)$ solve the following two (coupled) matrix Dyson equations:

\begin{equation}\label{equaiton1093}
    -M_1(w)^{-1}=wI_N+w\Phi[M_2(w)]-|z|^2(I_N+\widetilde{\Phi}[M_1(w)])^{-1},
\end{equation}and likewise 
\begin{equation}
    -M_2(w)^{-1}=wI_N+w\widetilde{\Phi}[M_1(w)]-|z|^2(I_N+{\Phi}[M_2(w)])^{-1}.
\end{equation}
\end{lemma}
The proof is presented at the end of this section.

We now check that the solution $M_1,M_2$ are diagonal matrices with constant diagonal value per block. For each $1\leq k\leq n$ we denote by $m_k$ the Stieltjes transform of the $k$-th diagonal block of $M_1$, i.e. $$m_k=\frac{1}{\ell}\operatorname{Tr}[M_1(w)]_{[(k-1)\ell+1,k\ell]\times [(k-1)\ell+1,k\ell]}.$$
We define $\widetilde{m}_1,\cdots,\widetilde{m}_{n}$ similarly as the normalized Stieltjes transform of each diagonal block of $M_2(w)$. 

\begin{fact}Assume $\rho=\mathbb{E}[\zeta^2]=0$.
The solution $M_1(w)$ is diagonal and has the form
$${M_1(w)}=\operatorname{diag}(m_1I_\ell,m_2I_\ell,\cdots,m_nI_\ell). $$ The other solution $M_2(w)$ also has a similar form $$M_2(w)=\operatorname{diag}(\widetilde{m}_1I_\ell,\widetilde{m}_2I_\ell,\cdots,\widetilde{m}_n I_\ell).$$ 
Moreover, $m_i(w)=\tilde{m}_i(w)$ for all $1\leq i\leq n$, so we are reduced to only one set of equation.
\end{fact}
\begin{proof}
    
    Since each random block in $\mathcal{T}_n$ is i.i.d., one can check that the map $\Phi[M_2(w)]$ depends only on the trace of each block and is constant on each block. Thus $M_1(w)$ is also constant on each diagonal block and is a block identity matrix with diagonal entry $m_k$ on the $k$-th block. One can check that $M_2(w)$ also has the same form. Then this implies that the four blocks on the right hand side of \eqref{MHWMHW}, labeled by $[1,N]\times[1,N],[1,N]\times[N+1,2N],[N+1,2N]\times[1,N],[N+1,2N]\times[N+1,2N]$ are all diagonal matrices, so that $-M^\mathcal{H}(w)$, being the inverse of this matrix, also satisfies that $[M^\mathcal{H}(w)]_{12},[M^\mathcal{H}(w)]_{21}$ are diagonal matrices, where we denote by $M^\mathcal{H}=\begin{bmatrix} [M^\mathcal{H}]_{11}&[M^\mathcal{H}]_{12}\\ [M^\mathcal{H}]_{21}&[M^\mathcal{H}]_{22}
    \end{bmatrix}$. By the homogeneous variance structure of $\mathcal{T}_{n}$, we can verify that $\hat{M}^\mathcal{H}:=\begin{bmatrix} [M^\mathcal{H}]_{22}&[M^\mathcal{H}]_{12}\\ [M^\mathcal{H}]_{21}&[M^\mathcal{H}]_{11}
    \end{bmatrix}$ is also a solution to the MDE \eqref{MHWMHW} with positive imaginary part. By \cite{helton2007operator}, Theorem 2.1, the  solution to MDE with $\Im w>0$ and positive imaginary part is unique. Thus we have $M^\mathcal{H}=\hat{M}^\mathcal{H}$, which implies, via \eqref{m1w2}, that $m_i=\tilde{m}_i$ for each $1\leq i\leq n$.\end{proof}

We can now extract the equations solved by $m_i$ as follows:
\begin{equation}\label{miavg}
\frac{1}{m_i}=-w(1+m_i^{avg})+|z|^2(1+m_i^{avg})^{-1},\quad 1\leq i\leq n,
\end{equation}
where we denote by 
\begin{equation}
m_i^{avg}=\frac{m_{i-1}+m_i+m_{i+1}}{3},
\end{equation}and we set the zero boundary condition $m_0=m_{n+1}=0.$

When the model is fully translationally invariant with no boundary, we can check that each $m_k,k\in[n+2]$ should take the same value $m_c$, where $m_c=m_c(w,z)$ is the Stieltjes-transform solution of the following equation, characterized by $\Im m_c>0$, $\Im(\sqrt w\,m_c)>0$ for $w\in\mathbb C_+$, and $m_c(w,z)\sim-1/w$ at infinity (with the principal square root) 
\begin{equation}\label{positiveimaginaryfollowing}
    m_c^{-1}=-w(1+m_c)+|z|^2(1+m_c)^{-1}.
\end{equation}
This can be verified via taking the candidate solution into \eqref{miavg} and ignoring all boundary non-isotropic effects. This computation that $m_k=\widetilde{m}_k=m_c$ has been used several times in prior circular law papers such as \cite{MR3230002} and \cite{han2025circular}. 

However, for our matrix $Y_z$, the boundary is not periodic so we have no closed form solutions. 
Nevertheless, we can still prove the following uniform upper bound:
\begin{lemma}\label{lemma5.6boundedness} For all $w=i\eta,\eta>0$, the solution to \eqref{miavg} satisfies 
$$
|\Im m_i|\leq \max(6|z|,\sqrt{6})\eta^{-1/2},\quad \forall 1\leq i\leq n.
$$
\end{lemma}

\begin{proof}
To upper bound $\Im m_i$, it is enough to show that the real part of $1/m_i$ is not too small, since 
$$
|m_i|\leq\frac{1}{|\Re(1/m_i)|},\quad \Im m_i\leq |m_i|.
$$We denote by $s_i:=m_i^{avg}$, then the dissipative term $-i\eta(1+s_i)$ contributes $\eta\Im(1+s_i)$ to $\Re(1/m_i)$. The coupling term $|z|^2(1+s_i)^{-1}$ may possibly reduce the real part, but by an amount at most $|z|^2/\Im(1+s_i)$. Thus taking the real part of the equation, we have
$$
\Re\frac{1}{m_i}\geq\eta\Im(1+s_i)-\frac{|z|^2}{\Im(1+s_i)}.
$$Now we denote by $u_i=\Im m_i$ and let $u_*=\max_i u_i$ (take $m_0=m_{n+1}=0$). Then let $i_*$ be the index achieving the maximum, we have 
\begin{equation}\label{wehaveline404}
\Im(1+s_{i_*})=\frac{u_{i_*-1}+u_{i_*}+u_{i_*+1}}{3}\geq\frac{u_*}{3}.
\end{equation}
Now, if $\eta\Im(1+s_{i_*})\leq 4\frac{|z|^2}{\Im(1+s_{i_*})}$, this directly yields the desired upper bound $$\Im(1+s_{i_*})\leq 2|z|\eta^{-1/2}.$$ Otherwise, $\eta\Im(1+s_{i_*})\geq4\frac{|z|^2}{\Im(1+s_{i_*})}$, then 
$$
\Re\frac{1}{m_{i_*}}\geq \frac{1}{2}\eta\Im(1+s_{i_*}).
$$This implies that, using \eqref{wehaveline404} for the last inequality in the next equation, 
$$
u_*\leq |m_{i_*}|\leq\frac{2}{\eta\Im(1+s_{i_*})}\leq\frac{6}{\eta u_*},
$$so that $u_*\leq\sqrt{6}\eta^{-1/2}$.
Combining both cases completes the proof. 
\end{proof}
For the general case $\rho=\mathbb{E}[\zeta^2]\neq 0$ we use a perturbation argument.

The following estimate is the block-band analogue of the pseudovariance
comparison in \cite{han2026brown}, Lemma 6.11. We include the short proof because our
variance profile has blocks of size $\ell$ and hence the small parameter is \(1/\ell\).

\begin{lemma}[Pseudovariance comparison]\label{lemmas5.10}
Let \(H_z^\rho\) be the Hermitization of the Gaussian block tridiagonal model
with the same variance profile as above, and the entry law satisfies
\[
        \mathbb E g^2=\rho,\qquad \mathbb E|g|^2=1.
\]
Let \(H_z^0\) be the circular Gaussian case \(\rho=0\). For
\(\lambda\in\mathbb C_+\), set
\[
        G_\rho(\lambda):=(H_z^\rho-\lambda)^{-1}.
\]
Then, for every fixed compact set of \(z\)'s and \(0<\Im\lambda\leq 1\),
\[
 \frac1N\left|
 \mathbb E\operatorname{Tr}\,[\lambda^{-1}G_\rho(\lambda)]_{11}
 -
 \mathbb E\operatorname{Tr}\,[\lambda^{-1}G_0(\lambda)]_{11}
 \right|
 \leq
 \frac{C_z|\rho|}{\ell(\Im\lambda)^8}.
\]
The same estimate holds for the \(22\)-block.
\end{lemma}

\begin{proof}
We use the same interpolation as in the proof of \cite{han2026brown}, Lemma 6.11. Let \(\mathcal B\) denote the set of allowed block positions of the block
tridiagonal matrix. Thus for each block row there are at most three choices of the
block column. Write
$\mathbb C^N=\mathbb C^n\otimes\mathbb C^\ell .$

For \((a,b)\in\mathcal B\) and \(u,v\in[\ell]\), denote the following matrix
\[
        E_{ab}^{uv}=|a,u\rangle\langle b,v|\in M_N(\mathbb C),
\]
and set
\[
        A_{ab}^{uv}:=
        \begin{pmatrix}
        0&E_{ab}^{uv}\\
        0&0
        \end{pmatrix}\in M_{2N}(\mathbb C).
\]
Let \(\sigma=(3\ell)^{-1/2}\). We interpolate between the circular Gaussian model and
the Gaussian model with pseudovariance \(\rho\) by choosing centered complex Gaussian
variables \(g_{ab}^{uv}(t)\) satisfying (note that when $\rho=1$ we call this degenerate complex Gaussian)
\[
        \mathbb E g_{ab}^{uv}(t)\overline{g_{a'b'}^{u'v'}(t)}
        =
        \delta_{aa'}\delta_{bb'}\delta_{uu'}\delta_{vv'},
\]
and
\[
        \mathbb E g_{ab}^{uv}(t)g_{a'b'}^{u'v'}(t)
        =
        t\rho\,
        \delta_{aa'}\delta_{bb'}\delta_{uu'}\delta_{vv'}.
\]
Let \(H_z(t)\) be the corresponding Hermitization and set
\[
        G_t(\lambda):=(H_z(t)-\lambda)^{-1}.
\]
For the \(11\)-block, define
\[
        F(t):=\frac1N\mathbb E\operatorname{Tr}
        \left[\lambda^{-1}G_t(\lambda)\right]_{11}.
\]
It suffices to bound \(F'(t)\).

By the complex Gaussian interpolation formula, for a smooth bivariate function $\Phi$,
\[
        \frac{d}{dt}\mathbb E\Phi(g(t),\bar g(t))
        =
        \frac{\rho}{2}\sum_\alpha
        \mathbb E\,\partial_{g_\alpha}^2\Phi
        +
        \frac{\bar\rho}{2}\sum_\alpha
        \mathbb E\,\partial_{\bar g_\alpha}^2\Phi,
\]
where \(\alpha=(a,b,u,v)\). That is, the differentiation in \(t\) pairs \(g_\alpha\) with
\(g_\alpha\), and \(\bar g_\alpha\) with \(\bar g_\alpha\). This is the
pseudovariance contribution. The usual covariance pairing
\(\mathbb E g_\alpha\bar g_\alpha\) is fixed along the interpolation and therefore does
not appear in \(F'(t)\).

Since
\[
        \partial_{g_\alpha}H_z(t)=\sigma A_\alpha,
        \qquad
        \partial_{\bar g_\alpha}H_z(t)=\sigma A_\alpha^*,
\]
the resolvent identities give
\[
        \partial_{g_\alpha}G_t
        =
        -\sigma G_tA_\alpha G_t,
        \qquad
        \partial_{g_\alpha}^2G_t
        =
        2\sigma^2G_tA_\alpha G_tA_\alpha G_t,
\]
and similarly
\[
        \partial_{\bar g_\alpha}^2G_t
        =
        2\sigma^2G_tA_\alpha^*G_tA_\alpha^*G_t.
\]
Therefore
\[
\begin{aligned}
        |F'(t)|
        &\leq
        \frac{C|\rho|}{N|\lambda|}
        \mathbb E\left|
        \operatorname{Tr}P_1G_t
        \left(\sigma^2\sum_\alpha A_\alpha G_tA_\alpha\right)
        G_t
        \right| \\
        &\quad+
        \frac{C|\rho|}{N|\lambda|}
        \mathbb E\left|
        \operatorname{Tr}P_1G_t
        \left(\sigma^2\sum_\alpha A_\alpha^*G_tA_\alpha^*\right)
        G_t
        \right|,
\end{aligned}
\]
where \(P_1\) is the projection onto the first \(N\) coordinates in the Hermitization.

We now bound the pseudovariance maps. Let \(X\in M_N(\mathbb C)\), and decompose it
into \(\ell\times\ell\) blocks \(X=(X_{ab})_{a,b=1}^n\). For a fixed allowed block
position \((a,b)\), one checks directly that
\[
        \sum_{u,v=1}^{\ell} E_{ab}^{uv}X E_{ab}^{uv}
\]
has only one nonzero block, namely the \((a,b)\)-block, and this block is
$       X_{ba}^{\mathsf t}.$
Indeed, its \((u,v)\)-entry equals \(X_{(b,v),(a,u)}\). Hence
\[
        \left\|
        \sigma^2\sum_{u,v=1}^{\ell}E_{ab}^{uv}X E_{ab}^{uv}
        \right\|
        \leq
        \frac{1}{3\ell}\|X\|.
\]
Thus the pseudovariance map transposes the internal \(\ell\times\ell\) block \(X_{ba}\)
and places it in the \((a,b)\)-block. This is the partial transpose on the
\(\ell\)-dimensional block coordinate. Since transposition preserves singular values, it
does not change the operator norm of each internal block. Since there are at most three
allowed block positions in each block row and each block column, the block Schur bound
implies
\[
        \left\|
        \sigma^2\sum_\alpha A_\alpha Y A_\alpha
        \right\|
        \leq
        \frac{C}{\ell}\|Y\|
\]
for every \(Y\in M_{2N}(\mathbb C)\). The same estimate holds with \(A_\alpha\) replaced
by \(A_\alpha^*\).

Using \(\|G_t(\lambda)\|\leq(\Im\lambda)^{-1}\), \(|\lambda|^{-1}\leq(\Im\lambda)^{-1}\),
and \(N^{-1}|\operatorname{Tr}P_1Y|\leq \|Y\|\), we obtain
\[
        |F'(t)|
        \leq
        \frac{C|\rho|}{\ell(\Im\lambda)^4}
        \leq
        \frac{C|\rho|}{\ell(\Im\lambda)^8},
        \qquad 0<\Im\lambda\leq 1.
\] 
Integrating over \(t\in[0,1]\) justifies the claim. The proof for the 22-block is similar.
\end{proof}

This implies the following comparison estimate:
\begin{lemma}[Pseudovariance perturbation of the finite-volume MDE]
\label{lem:pseudovariance-perturbation}
Let \(M^\mathcal{H}_\rho(\lambda)\) and \(M^\mathcal{H}_0(\lambda)\) be the finite-volume MDE
solutions corresponding respectively to \(\mathcal S^\mathcal{H}_\rho\) and
\(\mathcal S^\mathcal{H}_0\). Define
\[
        M_{\rho,1}(\lambda^2)
        :=
        \lambda^{-1}[M^\mathcal{H}_\rho(\lambda)]_{[1,N]\times[1,N]},
        \qquad
        M_{0,1}(\lambda^2)
        :=
        \lambda^{-1}[M^\mathcal{H}_0(\lambda)]_{[1,N]\times[1,N]}.
\]
Then for every fixed compact set of \(z\)'s there exists \(C_z>0\) such that,
for all \(\lambda\in\mathbb C_+\) with \(0<\Im\lambda\le1\),
\[
        \frac1N
        \left|
        \operatorname{Tr}M_{\rho,1}(\lambda^2)
        -
        \operatorname{Tr}M_{0,1}(\lambda^2)
        \right|
        \leq
        \frac{C_z|\rho|}{\ell(\Im\lambda)^8}+\frac{C_zN^c}{W^{1/2}(\Im\lambda)^6}.
\]
The same estimate holds for the \(22\)-block deterministic equivalent
\(M_{\rho,2}\) obtained from the $[N+1,2N]\times[N+1,2N]$-th corner of $\lambda^{-1}M_\rho^\mathcal{H}(\lambda)$.
\end{lemma}

\begin{proof}Lemma \ref{lemmas5.10} provides a perturbation estimate for the expectation of the trace of $\lambda^{-1}(\mathcal{H}_z^0-\lambda)^{-1}$ and $\lambda^{-1}(\mathcal{H}_z^\rho-\lambda)^{-1}$. For these Gaussian matrices, \cite{bandeira2023matrix}, Theorem 2.8, and the free-resolvent identification in the proof of Proposition \ref{proposition4.1} give
\[
\|\mathbb E G_\rho(\lambda)-M^\mathcal H_\rho(\lambda)\|
\leq\frac{C}{W(\Im\lambda)^5},
\]
since $\sigma^2\leq1$ and $v^2\leq C/W$. This comparison holds for every $W\geq1$. Using $|\lambda|\geq\Im\lambda$ and taking a normalized block trace gives
\[
\left|
\frac1N\mathbb E\operatorname{Tr}[\lambda^{-1}G_\rho(\lambda)]_{11}
-\frac1N\operatorname{Tr}M_{\rho,1}(\lambda^2)
\right|
\leq
\frac{C_z}{W(\Im\lambda)^6}
\leq\frac{C_zN^c}{W^{1/2}(\Im\lambda)^6},
\]
and the same holds for \(\rho=0\). Combining the two estimates completes the proof. 
\end{proof}

Combining the previous deterministic estimates, we conclude that
\begin{corollary}[Trace bound for the finite-volume MDE]
\label{lem:block-tridiagonal-MDE-bound}
Let \(M^\mathcal{H}_\rho(\lambda)\) be the finite-volume MDE solution associated with
the canonical block tridiagonal model, and define
\[
        M_{\rho,1}(\lambda^2)
        :=
        \lambda^{-1}[M^\mathcal{H}_\rho(\lambda)]_{[1,N]\times[1,N]}.
\]
Then for every fixed compact set of \(z\)'s there exists \(C_z>0\) such that,
for all \(0<\eta\leq 1\), with
\[
        \lambda=\frac{1+i}{\sqrt2}\eta,
        \qquad \lambda^2=i\eta^2,
\]
we have
\[
        \Im \operatorname{Tr} M_{\rho,1}(i\eta^2)
        \leq
        C_z N\left(\eta^{-1}+\frac{1}{W\eta^8}+\frac{N^c}{W^{1/2}\eta^6}\right),
        \qquad W=\ell.
\]
The same estimate holds for the \(22\)-block deterministic equivalent
\(M_{\rho,2}\).
\end{corollary}

\begin{proof}
When \(\rho=0\), Lemma \ref{lemma5.6boundedness} gives, for \(w=i\eta^2\),
\[
        \Im \operatorname{Tr} M_{0,1}(i\eta^2)
        \leq C_zN\eta^{-1}.
\]
For general \(\rho\), Lemma \ref{lem:pseudovariance-perturbation} gives
\[
        \frac1N
        \left|
        \operatorname{Tr}M_{\rho,1}(i\eta^2)
        -
        \operatorname{Tr}M_{0,1}(i\eta^2)
        \right|
        \leq
        \frac{C_z|\rho|}{\ell \eta^8}+\frac{C_zN^c}{W^{1/2}\eta^6}.
\]
Since \(|\rho|\leq 1\) and \(W=\ell\), the desired estimate follows. The proof for
the \(22\)-block is identical.
\end{proof}

Combining these uniform upper bounds of Stieltjes transform, and matrix concentration inequality in Proposition \ref{proposition4.1}, we can deduce singular value rigidity for $Y_z$:

\begin{corollary}\label{rigiditylargeboundcor}Assume $W\geq\sqrt N$. With the same assumption as above, for any sufficiently small constant $c>0$, we have that 
with probability at least $1-N^{-100}$, the number of singular values of $Y_z$ in $[0,W^{-1/10}N^{c/5}]$ is at most $C_zN^{1+c/5}W^{-1/10}$ for a constant $C_z>0$ depending only on $z$ and the entry law.    
\end{corollary}

\begin{proof}
Apply Proposition \ref{proposition4.1} with exponent $c/2$ and absorb its fixed constant into $N^{c/2}$ for sufficiently large $N$. With probability at least $1-N^{-100}$, we bound the number of singular values by the Stieltjes transform via the Poisson kernel formula and get 
$$\begin{aligned}
\#\{\text{Number of singular values of }Y_z\text{ in }[0,\eta]\}&\leq 
2\eta^2\Im(\operatorname{Tr}\mathcal{G}(i\eta^2))\\&=2\eta^2\Im (\operatorname{Tr}\mathcal{G}([(1+i)\eta/\sqrt{2}]^2))
\\&\leq  2\eta^2(\Im\operatorname{Tr}M_{\rho,1}(i\eta^2)+\frac{2N^{1+c}}{W^{1/2}\eta^6})\\&\leq C(z)N\eta^2(\eta^{-1}+\frac{1}{W\eta^8}+\frac{2N^c}{W^{1/2}\eta^6})
\end{aligned}$$ for some constant $C(z)>0$ depending only on $z$. In the third line we used Proposition \ref{proposition4.1} and in the fourth line we used Corollary \ref{lem:block-tridiagonal-MDE-bound}. We take $\eta=W^{-1/10}N^{c/5}$ and this leads to the desired estimate. 
\end{proof}

Now we return to the matrix $\mathcal{T}_{n+2}-zI_{mid}$ and complete the proof of Proposition \ref{propositionrigidity}.

\begin{proof}[\proofname\ of Proposition \ref{propositionrigidity}]We work on an event $\|\mathcal{T}_{n+2}-zI_{mid}\|\leq K$ for some $K$ that holds with probability at least $1-N^{-100}$: the existence of such $K$ follows from the operator norm bound in Lemma \ref{operatornormbound} and \ref{astronger}.
    Then we apply Lemma \ref{lemma5.1} to reduce to singular value rigidity of $Y_z^R$, and apply Lemma \ref{lemma5.2} to reduce to singular value rigidity of the block diagonal matrix $Y_z^{R,diag}$ with probability at least $1-C\ell^{-7}$. Specifically, take $m:=\frac{(n\ell)^{1+c}}{\ell^{0.1}}$ and $t=(n\ell)^{-120}$, if $\mathcal{T}_{n+2}-zI_{\text{mid}}$ has more than $m$ singular values below $t$ then $Y_z^R$ has more than $m$ singular values below $\epsilon=3K\sqrt{2n\ell}t$. Under the assumption $n\leq\ell^{1/12}$, this gives $\epsilon\leq\ell^{-110}$ for $\ell$ large enough, so we can apply Lemma \ref{lemma5.2}.

    Then the rigidity of the upper block follows from Lemma \ref{lemma3.2} and rigidity for the lower large block follows from Corollary \ref{rigiditylargeboundcor}. Specifically, the top block of $Y_z^{R,\text{diag}}$ is block diagonal with top block $[C_1,A_1-zI_\ell]R$, whose least singular value is at least $\ell^{-10}$ with probability at least $1-C\ell^{-8}$. Then since $\ell^{-10}\ll W^{-1/10}N^{c/5}$, the lower block $[Y_z^{R,\text{diag}}]_{[2,n]\times[2,n]}$ has at most $C_zN^{1+c/5}W^{-1/10}\leq\frac{(n\ell)^{1+c}}{\ell^{0.1}}$ number of singular values in the interval $[0,\ell^{-10}]$ after shrinking $c$.
\end{proof}

We also complete the proof of Proposition \ref{stieltjesconvergence}:
\begin{proof}[\proofname\ of Proposition \ref{stieltjesconvergence}]
  We first consider the circular case \(\rho=\mathbb E\zeta^2=0\). Here $N=(n+2)\ell\leq3\ell^2$, so $W=\ell\geq\sqrt{N/3}$; the proof of Proposition \ref{proposition4.1} remains valid under this constant-factor relaxation of its bandwidth assumption. Applying the
same local-law argument as in Proposition \ref{proposition4.1} to the periodic matrix
\(\mathcal{T}^{\rm per}_{n+2}\), and using the translation invariance of the periodic variance
profile, we see that the associated finite-volume MDE solution is constant on the diagonal blocks.
The solution is therefore given by the scalar solution \(m_c\) of \eqref{positiveimaginaryfollowing}. Projecting the
Hermitized MDE as in \eqref{m1w2}, and using the standard identification of this scalar
solution with the Stieltjes transform \(m_z\) of the circular-law singular-value
measure, gives the desired estimate in the case \(\rho=0\).

For a general entry distribution, let \(M^\mathcal{H}_\rho\) and \(M^\mathcal{H}_0\) denote the
finite-volume MDE solutions corresponding to the same periodic variance profile
but with pseudovariance \(\rho\) and \(0\), respectively. Set $\xi=\lambda^2$. The same perturbation
argument as in Lemma \ref{lem:pseudovariance-perturbation} gives
\[
\frac1N\left|\operatorname{Tr} M_{\rho,1}(\lambda^2)
-\operatorname{Tr}M_{0,1}(\lambda^2)\right|
\leq
\frac{C_z|\rho|}{\ell(\Im\lambda)^8}+\frac{N^c}{\ell^{1/2}(\Im\lambda)^6}.
\]
For $|\Re\xi|\leq D$ and $0<\Im\xi\leq1$, the principal root $\lambda=\sqrt\xi$ satisfies $\Im\lambda\geq c_D\Im\xi$. The resolvent comparisons above also hold for $1<\Im\lambda\leq C_D$, with constants depending on $D$, by the same estimates. Thus this error is bounded by the error term in Proposition \ref{stieltjesconvergence}. Hence the same
convergence to \(m_z\) holds for general \(\rho\). Combining this deterministic
comparison with the local-law estimate for the MDE solution proves the claim for $\Im\xi\geq N^{-1}$. For $0<\Im\xi<N^{-1}$, the deterministic bound $|m_{n+2,z}^{per}(\xi)-m_z(\xi)|\leq2/(\Im\xi)$ is already smaller than the asserted error.
\end{proof}

We list here the (standard) proof of Lemma \ref{proofoflemma4.2}.

\begin{proof}[\proofname\ of Lemma \ref{proofoflemma4.2}]
    We denote by $$M^\mathcal{H}=\begin{bmatrix}
        A&B\\C&D
    \end{bmatrix},\quad R:=\begin{bmatrix}
        w I_N+\Phi[D]&zI_N\\\bar{z}I_N& w I_N+\widetilde{\Phi
        }[A]
    \end{bmatrix},$$ so that we can rewrite $M^\mathcal{H}(w)=-R^{-1}$.

    Applying Schur complement formula to $R=\begin{bmatrix}
        R_{11}&R_{12}\\R_{21}&R_{22}
    \end{bmatrix}$ with $R_{22}$ invertible, we have
    $$
(R^{-1})_{11}=(R_{11}-R_{12}R_{22}^{-1}R_{21})^{-1}.
    $$ Applied with the following choices
$$R_{11}=w I_N+\Phi[D],R_{12}=z I_N
,R_{21}=\bar{z}I_N,R_{22}=w I_N+\widetilde{\Phi}[A],$$ we get that 
$$
-A^{-1}=wI_N+\Phi[D]-|z|^2(w I_N+\widetilde{\Phi}[A])^{-1}.
$$ A similar equation can be derived for $D^{-1}$. Denote $A=A[w]$ to signify the dependence on $w$, then we use $M_1(w^2)=\frac{1}{w}A[w]$ in \eqref{m1w2} to rescale and then verify the two identities satisfied by $M_1(w),M_2(w)$.
    \end{proof}

\section*{Acknowledgments}The author thanks Charles Bordenave, Giorgio Cipolloni and Tatyana Shcherbina for many insightful discussions on this topic.

\section*{Funding}
The author receives a fellowship from IAS provided by the S.S. Chern Foundation for Mathematical Research Fund and the Fund for Mathematics.

Part of this work was completed while the author was financially supported by a Simons Foundation Grant (601948, DJ) at Massachusetts Institute of Technology.

\begingroup
\emergencystretch=2em
\printbibliography
\endgroup

\end{document}